\documentclass[a4paper]{article}
\usepackage{amsmath,amssymb,amsfonts,amsthm} 
\usepackage{mathrsfs} 
\usepackage{enumerate} 
\usepackage{cases} 
\usepackage{float} 
\usepackage{graphicx} 
\allowdisplaybreaks[4] 
\usepackage{wrapfig}  
\usepackage{url} 
\usepackage[usenames]{color} 
\usepackage{setspace} 
\def\b1{\mbox{\boldmath $1$}} 

\makeatletter
\newcommand{\Biggg}{\bBigg@{3.5}} 
\makeatother

\theoremstyle{plain} 
\newtheorem{theorem}{\bf Theorem}[section] 
\newtheorem{lemma}{\bf Lemma}[section] 
\newtheorem{proposition}{\bf Proposition}[section] 
\newtheorem{definition}{\bf Definition}[section] 
\newtheorem{remark}{\bf Remark}[section]  
\newtheorem{example}{\bf Example}[section] 

\makeatother
\allowdisplaybreaks 

\usepackage[numbers,sort&compress]{natbib} 
\usepackage[bookmarksnumbered,bookmarksopen,colorlinks,citecolor=blue,linkcolor=blue]{hyperref} 

\begin{document}
 	\title{Zero-sum stochastic linear-quadratic Stackelberg differential games of Markovian regime-switching system} 
	
\author{Fan Wu \thanks{School of Mathematics, Southeast University, Nanjing 211189, China}
\qquad Xun Li \thanks{Department of Applied Mathematics, Hong Kong Polytechnic University, Hong Kong, China}
\qquad Jie Xiong \thanks{Department of Mathematics, Southern University of Science and Technology, Shenzhen 518055, China}
\qquad Xin Zhang \thanks{Corresponding Author: School of Mathematics, Southeast University, Nanjing 211189, China; E-mail: x.zhang.seu@gmail.com}
} 

\maketitle
\begin{abstract} This paper investigates a zero-sum stochastic linear-quadratic (SLQ, for short) Stackelberg differential game problem, where the coefficients of the state equation and the weighting matrices in the performance functional are regulated by a Markov chain. By utilizing the findings in \citet{Zhang.X.2021_ILQM}, we directly present the feedback representation to the rational reaction of the follower. For the leader's problem, we derive the optimality system through the variational method and study its unique solvability from the Hilbert space point of view. We construct the explicit optimal control for the leader based on the solution to coupled differential Riccati equations (CDREs, for short) and obtain the solvability of CDREs under the one-dimensional framework.  Finally, we provide two concrete examples to illustrate the results developed in this paper.
\end{abstract}

\noindent\textbf{2020 Mathematics Subject Classification:} 91A15, 49N10, 93E20.

\noindent\textbf{Keywords: }Zero-sum stochastic  Stackelberg differential games, Linear-quadratic optimal control, Backward stochastic differential equation, Riccati equations

\maketitle

\section{Introduction}\label{section-introduction}
SLQ optimal control problem plays an important role in stochastic control theory. The forward and backward SLQ control problems have been widely studied in recent decades. The forward stochastic linear-quadratic (FSLQ, for short) was initially studied by \citet{kushner.H.1962_OSC} via the dynamic programming method. Later,  \citet{Wonham.W.M.1968} discussed the generalized version of the matrix Riccati and quadratic matrix equations, which arise in problems of stochastic control and filtering.  \citet{Tang.S.J.2003_LQR} investigated the existence and uniqueness result for the associated Riccati equation for a general SLQ problem, which solves Bismut and Peng's long-standing open problem. In addition, \citet{Tang.S.J.2003_LQR} also provided a rigorous derivation of the Riccati equation from the stochastic Hamilton system, which completes the interrelationship between the Riccati equation and the stochastic Hamilton system as two different but equivalent tools for the SLQ problem.  \citet{Wu.Z.2003_LQJR} considered the linear-quadratic stochastic optimal control with random jumps and derived existence and uniqueness result of forward-backward stochastic differential equations (FBSDE, for short) with Brownian motion and Poisson process. \citet{Hu.Y.Z.2008_LQJP} generalized the result of \citet{Wu.Z.2003_LQJR} to the case of partial information SLQ problem. \citet{Ji.Y.D.1990_LQM,Ji.Y.D.1991_LQGJ} formulated a class of continuous-time SLQ optimal controls with Markovian jumps.  \citet{Zhang.Q.1999_LQG} developed approximation schemes with hybrid controls of a class of linear quadratic Gaussian (LQG) systems modulated by a finite-state Markov chain.

The SLQ optimal control problem for backward stochastic differential equation (BSDE, for short) was initially investigated by \citet{Lim.A.E.B.2001_BLQ}, where the coefficients are deterministic, and all the weighting matrices are non-negative definiteness. They obtained the explicit solution using an approach based primarily on the completion-of-squares technique. Along this line, a couple of follow-up works appeared afterward. \citet{Huang.J.H.2009_MPBSDEP} and  \citet{Wang.G.C.2012_BLQP} considered the backward stochastic linear-quadratic (BSLQ, for short) optimal control problem under partial information. \citet{Sun.J.R.2021_BLQR} thoroughly investigated the BSLQ problem with random coefficients.  \citet{Huang.J.H.2016_BLQGMFP} studied a backward mean-field linear-quadratic-Gaussian game with complete/partial information.  \citet{Wang.G.C.2018_NZBLQ} analyzed a kind of SLQ nonzero-sum differential game with asymmetric information for BSDEs.	A dynamic game of $N$ weakly-coupled linear BSDE systems involving mean-field interactions was studied by \citet{Du.K.2018_BLQMF}.

It is worth pointing out that the above-mentioned works assume the positive/non-negative definiteness condition imposed on the weighting matrices.
To the best of our knowledge, \citet{Chen.S.P.1998_ILQ} was the first work to study FSLQ optimal control problems with an indefinite quadratic weighting control matrix, which is useful in solving continuous-time mean-variance portfolio selection problems. Since then, there has been an increasing interest in the so-called indefinite SLQ optimal control.  \citet{Li.N.2018_ILQJ} considered the indefinite FSLQ optimal control problem with Poisson jumps.  \citet{Li.X.2001_ILQJ,Li.X.2003_ILQMI} formulated indefinite FSLQ optimal controls with regime-switching jumps and tackled it using semidefinite programming. Recently, \citet{Sun.J.R.2014_NILQ} and  \citet{Sun.J.R.2016_open-closed} carefully studied the open-loop and closed-loop solvabilities for FSLQ optimal control problems under the uniform convexity condition. It was shown that open-loop solvability is equivalent to the existence of an adapted solution to an FBSDE with constraint, and closed-loop solvability is equivalent to the existence of a regular solution to the Riccati equation. In addition, the equivalence between the strongly regular solvability of the Riccati equation and the uniform convexity of the cost functional was established.  \citet{Zhang.X.2021_ILQM} successfully generalized their results to the case within the framework of regime-switching jump diffusion process.  \citet{Sun.J.R.2021_IBLQ} considered a homogeneous indefinite BSLQ optimal control problem and obtained the explicit optimal control by solving a Riccati-type equation.  \citet{Sun.J.R.2021_GIBLQ} generalized the results of \citet{Sun.J.R.2021_IBLQ} to the case of non-homogeneous indefinite BSLQ problem.

Since the pioneering work \citet{Stackelberg.H.V.1934} by Stackelberg, the theory of Stackelberg games has been widely used in economics, finance, and engineering. The Stackelberg SLQ differential game was initially studied by \citet{Bagchi.A.1981_SLQ}. From then on,  \citet{Yong.J.M.2002_SLQ} investigated the Stackelberg SLQ problem under a general framework and showed that the open-loop solution admits a state feedback representation if the corresponding stochastic Riccati equation is solvable.  A leader-follower stochastic differential game with asymmetric information was considered in \citet{Shi.J.T.2016_SLQAI}.  \citet{Moon.J.2020.SLQMFTC} and \citet{Moon.J.2021_SLQJ} analyzed the Stackelberg SLQ problem under the time-inconsistent mean-field frame and jump-diffusion model, respectively.  \citet{Du.K.2019_SBLQMF} carried out a new kind of Stackelberg differential game of mean-field backward stochastic differential equations and obtained the open-loop Stackelberg equilibrium with state feedback representation.

Although the Stackelberg SLQ differential game is widely studied, the zero-sum Stackelberg SLQ problem is rarely studied. The zero-sum SLQ Stackelberg differential games can be considered as the combination of the FSLQ problem and BSLQ problem in the sense that the follower's problem is an FSLQ problem while the leader's problem is a BSLQ problem. \citet{lin_stochastic_2012} formulated an optimal portfolio selection problem with model uncertainty as a zero-sum stochastic differential game between the investor and the market. Using techniques of SLQ control theory, they obtained the closed-form solutions to corresponding game problems. \citet{Sun.J.R.2021_ZSILQ} studied a zero-sum Stackelberg SLQ  differential game in which the state coefficients and cost weighting matrices are given deterministic functions. Under the so-called uniform convexity-concavity condition, the associated Riccati equations are solvable, and the Stackelberg equilibrium admits a linear state feedback representation.  \citet{wu_zero-sum_2024} investigated a zero-sum Stackelberg SLQ differential game with Poisson jumps, in which the coefficients of the state equation and the weighting matrices in the performance functional are allowed to be random.

In this paper, we shall generalize the result in \citet{Sun.J.R.2021_ZSILQ} to consider a zero-sum Stackelberg SLQ  differential game, in which the state equation and cost weighting matrices are affected by a finite-state Markov chain. Although the main difference between our model and the one in \cite{Sun.J.R.2021_ZSILQ} is the regime-switching jump diffusion processes, 
the two papers are different in terms of research content and technical difficulties. Below, we carefully compared our work with that in \cite{Sun.J.R.2021_ZSILQ} from several perspectives.
\begin{itemize}
\item Firstly, the main research goals between our paper and that in \citet{Sun.J.R.2021_ZSILQ} are different.  \citet{Sun.J.R.2021_ZSILQ} focus on constructing the relation between the zero-sum SLQ Nash differential game and the corresponding Stackelberg differential game based on their early research work investigating the zero-sum SLQ Nash differential game (see \citet{Sun.J.R.2014_NILQ}) and the studies of BSLQ problem (see \citet{Sun.J.R.2021_IBLQ, Sun.J.R.2021_GIBLQ}). Our paper focuses on solving the zero-sum SLQ Stackelberg differential game. We transformed the leader's problem into a BSLQ problem regulated by a Markov chain. We further solved this BSLQ problem and obtained the optimal equilibrium control for the zero-sum SLQ Stackelberg differential game, which can be seen as the main contribution of this paper. 
\item Secondly, we allow the inhomogeneous terms to appear in state process and performance functional while \citet{Sun.J.R.2021_ZSILQ} only consider the homogeneous case. As we shall see in the rest of the paper, inhomogeneous terms make constructing optimal control for the leader's problem and equilibrium value function more complex than counterparts of \citet{Sun.J.R.2021_ZSILQ}. Additionally, in the context of Markov regime switching, decoupling the optimality system of the leader's problem requires us to solve a system of CDREs, which is more complex than solving a Riccati equation alone in \citet{Sun.J.R.2021_ZSILQ}.
\item Finally, the Riccati equation corresponding to the leader's problem in \citet{Sun.J.R.2021_ZSILQ} is solved by using the forward formulation and limiting procedure developed in \cite{Lim.A.E.B.2001_BLQ,Sun.J.R.2021_IBLQ}.
In our paper, we can also represent the solution to the CDREs \eqref{L-RE} as the limit of the inverse of the solutions to another system of CDREs (see Theorem \ref{thm-F-L-REs}), which is derived from a more general FSLQ problem with the jumps of Markov chain entered into the state process (see \eqref{problem-G-M-FSLQ-state}-\eqref{problem-G-M-FSLQ}).  
Such a finding differs from the BSLQ problem studied by \citet{Sun.J.R.2021_IBLQ} for the diffusion model case. The main reason for this difference is due to the coupling term $\sum_{k=1}^{D}\lambda_{ik}(s)\Sigma(s,k)$ in CDREs \eqref{L-RE}. 
Under the one-dimensional case, we simplify the original CDREs \eqref{L-RE} into a new one \eqref{L-RE-rewritten} and provide its solvability based on the findings in \citet{Zhang.X.2021_ILQM} (see Theorem \ref{RE-solvability}).
\end{itemize}
The main results of our paper are as follows.
	\begin{itemize}
	    \item We convert the leader's problem into a BSLQ problem regulated by a Markov chain and characterize its open-loop optimal control by the solvability of a system of constrained FBSDEs and the convexity  of the cost functional. 
	    \item We also prove that the BSLQ problem admits a unique optimal control when the corresponding cost functional is uniform convex from the Hilbert space point of view.
	    \item Under the uniform convexity-concavity condition (see assumptions (H3)-(H4)), we obtain  
	    the feedback representation of the follower's rational reaction and explicit optimal control for the leader based on the solutions to  \eqref{L-RE}-\eqref{phi} (see, Section \ref{section-construction}). 
	    \item We prove that the solution to the CDREs \eqref{L-RE} corresponding to leader’s problem can be represented as a limit of the inverse of solutions to another system of Riccati differential equations \eqref{FSLQ-REs}, which can be derived from a class of more general FSLQ problems with the jumps of Markov chain entered into the state process. Moreover,  we provide the solvability result for CDREs \eqref{L-RE} under the one-dimensional framework.
	\end{itemize}
	The rest of the paper is organized as follows. Section \ref{section-Preliminaries-model} introduces some useful notations and formulates the zero-sum Stackelberg SLQ problem with regime-switching jumps. In section \ref{section-Follower},  we derive the follower's rational reaction based on the results in \citet{Zhang.X.2021_ILQM}. Section \ref{section-Leader} aims to solve the leader's problem. The open-loop solvability is obtained in this section, and the explicit optimal control is constructed based on the solution to a system of CDREs, whose solvability is discussed in Section \ref{section-RE-solvability}.
	Finally,  Section \ref{section-example} concludes the paper by giving two concrete examples to illustrate the results developed in the earlier sections.

	\section{Preliminaries and model formulation} \label{section-Preliminaries-model}
	Let $(\Omega,\mathcal{F},\mathbb{P} )$ be a complete probability space with the natural filtration $\mathbb{F} :=
	\{ \mathcal{F}_t, 0 \leq t \leq T\}$ generated by the following two mutually independent stochastic
	processes and augmented by all the $\mathbb{P}$-null sets in $\mathcal{F}$ :
	\begin{itemize}
		\item A one-dimensional standard Brownian motion $W=\left\{W(t);0\leq t\leq T\right\}$;
		\item A continuous time, finite-state, irreducible Markov chain $\alpha=\left\{\alpha(t);0\leq t\leq T\right\}$.
	\end{itemize}
	We denote the state space of the Markov chain as $\mathcal{S}:=\left\{1,2,...,D\right\}$, where $D$ is a finite natural number, and define the generator of the chain under $\mathbb{P}$ as $\lambda(t):=\left[\lambda_{ij}(t)\right]_{i,j=1,2,...,D}$. Here,  for $i\neq j$,  $\lambda_{ij}(t)\geq 0$ is the determinate transition intensity of the chain from  state $i$ to state $j$ at time $t$
	and $\sum_{j=1}^{D}\lambda_{ij}(t)=0$ for any fixed $i$. In the following, let $N_{k}(t)$ be the number of jumps  into state $k$ up to time $t$ and set
	$$
	\widetilde{N}_{k}(t)=N_{k}(t)-\int_{0}^{t}\sum_{i\neq k}\lambda_{ik}(s)I_{\{\alpha(s-)=i\}}(s)ds.
	$$
 Then for each $k=1,2,...,D$, the term $\widetilde{N}_{k}(t)$ is an $\left(\mathbb{F},\mathbb{P}\right)$-martingale. For any given D-dimensional vector process $\mathbf{\Gamma}(\cdot)=\left[\Gamma_1(\cdot),\Gamma_2(\cdot),\cdots,\Gamma_D(\cdot)\right]$, we define $$\mathbf{\Gamma}(s)\cdot d\mathbf{\widetilde{N}}(s)\triangleq\sum_{k=1}^{D}\Gamma_{k}(s)d\widetilde{N}_{k}(s).$$

Throughout the paper, we denote $\mathbb{R}^{n \times m}$ as the Euclidean space consisting of all $n \times m$ matrices and endowed with the Frobenius inner product $\big<M,N\big>=tr\left(M^{\top}N\right)$, where $M^{\top}$ and $tr(M)$ represent the transpose and trace of $M$, respectively. When $m=1$, we simplify $\mathbb{R}^{n \times 1}$ as $\mathbb{R}^{n}$. The identity matrix of size $n$ is denoted by $I_{n}$, and often written as $I$ for simplicity when no confusion arises. Let $\mathbb{S}^n \, \left(\mathbb{S}_{+}^n,\, \overline{\mathbb{S}_{+}^n}\right)$  be the set consisting of all $n \times n$ symmetric (positive-definite, positive semi-definite) matrices. For $M, N \in \mathbb{S}^n$, we write $M \geqslant N$ (respectively, $M>N$) if $M-N$ is positive semi-definite (respectively, positive definite). And for a $\mathbb{S}^n$-valued measurable function $F$ on $[0, T]$, we write
$$\left\{\begin{array}{lll}F \geqslant 0 & \text { if } \quad F(s) \geqslant 0, & \text { a.e. } s \in[0, T], \\ F>0 & \text { if } \quad F(s)>0, & \text { a.e. } s \in[0, T], \\ F \gg 0 & \text { if } \quad F(s) \geqslant \delta I_n, & \text { a.e. } s \in[0, T], \text { for some } \delta>0 .\end{array}\right.$$
Moreover, we use $F\leq 0$, $F<0$ and $F\ll 0$ to indicated that $-F\geq 0$, $-F>0$ and $-F\gg 0$, respectively.

Let $\mathcal{P}$ be the $\mathbb{F}$ predictable $\sigma$-field on $[0,T]\times\Omega$ and for a given process $\varphi(\cdot)$, we write $\varphi \in \mathbb{F}$ (respectively, $\varphi \in \mathcal{P}$) if it is $\mathbb{F}$-progressively measurable (respectively, $\mathcal{P}$-measurable). Then, for any given Euclidean space $\mathbb{H}$, we introduce the following space:
	\begin{align*}
		C(0, T; \mathbb{H}) &=\{\varphi:[0, T] \rightarrow \mathbb{H} \mid \varphi(\cdot) \text { is continuous }\}, \\
		L^{\infty}(0, T ; \mathbb{H}) &=\left\{\varphi:[0, T] \rightarrow \mathbb{H} \mid \operatorname{esssup}_{s \in[0, T]}| \varphi(s)|<\infty\right\},
	\end{align*}
	and
	\begin{align*}
		L_{\mathcal{F}_{T}}^{2}(\Omega ; \mathbb{H}) &=\left\{\xi: \Omega \rightarrow \mathbb{H} \mid \xi \text { is } \mathcal{F}_{T} \text {-measurable, } \mathbb{E}|\xi|^{2}<\infty\right\},\\
		\mathcal{S}_{\mathbb{F}}^{2}(0, T ; \mathbb{H}) &=\left\{\varphi:[0, T] \times \Omega \rightarrow \mathbb{H} \mid \varphi(\cdot) \in \mathbb{F} \text {, } \mathbb{E}\left[\sup _{s \in[0, T]}|\varphi(s)|^{2}\right]<\infty\right\}, \\
		L_{\mathbb{F}}^{2}(0, T ; \mathbb{H}) &=\left\{\varphi:[0, T] \times \Omega \rightarrow \mathbb{H} \mid \varphi(\cdot) \in \mathbb{F}\text{, } \mathbb{E} \int_{0}^{T}|\varphi(s)|^{2} ds<\infty\right\}, \\
		L_{\mathcal{P}}^{2}\left(0, T ; \mathbb{H}\right) &=\left\{\varphi:[0, T] \times \Omega \rightarrow \mathbb{H} \mid \varphi(\cdot) \in \mathcal{P}\text{, } \mathbb{E} \int_{0}^{T}|\varphi(s)|^{2} ds<\infty\right\} .
	\end{align*}
In addition, for any given Banach space $\mathbb{B}$, we denote
$$\mathcal{D}\left(\mathbb{B}\right)=\left\{\mathbf{\Lambda}=\left(\Lambda(1),\cdots,\Lambda(D)\right) \mid \Lambda(i) \in \mathbb{B}\text{, } \forall i\in \mathcal{S}\right\}.$$

Now,  we consider the following controlled stochastic differential equation (SDE) on $[0,T]$:
	\begin{equation}\label{state}
		\left\{
		\begin{aligned}
			dX(s)&=\left[A(s,\alpha(s))X(s)+B_{1}(s,\alpha(s))u_{1}(s)+B_{2}(s,\alpha(s))u_{2}(s)+b(s)\right]ds\\
			&\quad+\left[C(s,\alpha(s))X(s)+D_{1}(s,\alpha(s))u_{1}(s)+D_{2}(s,\alpha(s))u_{2}(s)+\sigma(s)\right]dW(s)\\
			X(0)&=x, \quad \alpha(0)=i
		\end{aligned}
		\right.
	\end{equation}
	For fixed $i\in \mathcal{S}$, $A(\cdot,i)$, $B_{1}(\cdot,i)$, $B_{2}(\cdot,i)$, $C(\cdot,i)$, $D_{1}(\cdot,i)$, $D_{2}(\cdot,i)$, are matrix-valued deterministic functions of proper dimensions and $b(\cdot)$ and $\sigma(\cdot)$ are vector-valued stochastic process. In the above, $X(\cdot)\equiv X(\cdot;x,i,u_{1},u_{2})$, valued in $\mathbb{R}^{n}$, is called the \emph{state process} with the initial value $x$. For $k=1,2$, $u_{k}(\cdot)$, valued in $\mathbb{R}^{m_{k}}$, is called the \emph{control process} of Player $k$ and we denote $\mathcal{U}_{k}=L_{\mathbb{F}}^{2}(0,T;\mathbb{R}^{m_{k}})$ as the space consisting all admissible controls for  Player $k$. We suppose the two players have opposing aims, and each player can affect the evolution of the state process \eqref{state} by selecting his/her own control.

	To measure the performance of the controls $u_{1}(\cdot)$ and $u_{2}(\cdot)$, we introduce the following criterion functional:
	\begin{equation}\label{costfunction}
		\begin{aligned}
			J\left(x,i; u_{1}(\cdot), u_{2}(\cdot)\right) \triangleq  
			\mathbb{E}\bigg\{&\int_{0}^{T}
			\big[\langle Q(s,\alpha(s))X(s), X(s)\rangle+\sum_{i=1}^{2}\langle R_{i}(s,\alpha(s))u_{i}(s), u_{i}(s)\rangle\big] ds\\
			&+\langle M(T,\alpha(T)) X(T), X(T)\rangle+2\langle m, X(T)\rangle\bigg\}.
		\end{aligned}
	\end{equation}
	For any fixed $i\in\mathcal{S}$, $M(T,i)$ is a  symmetric matrix, $Q(\cdot,i)$, $R_{1}(\cdot,i)$, $R_{2}(\cdot,i)$ are given determinate matrix-valued functions of proper dimensions, and $m$ is $\mathcal{F}_{T}$-measurable random variables.

	In our setting, the criterion functional \eqref{costfunction} is regarded as the loss of Player $1$ and the gain of Player $2$. The  Player $2$ is the leader, who announces his/her control $u_{2}(\cdot)$ first, and Player $1$ is the follower, who chooses his/her rational reaction accordingly. So whatever the leader announces $u_{2}(\cdot)$, the follower will find the optimal reaction $\bar{u}_{1}[x,i,u_{2}](\cdot)$ (depending on $u_{2}(\cdot)$ and initial value $x,\,i$) to minimize  the functional $	J\left(x,i; u_{1}(\cdot), u_{2}(\cdot)\right)$. We assume the leader can obtain the full information of the follower and predict the follower's optimal reaction. Hence, he/she will choose the optimal control $u_{2}^{*}(\cdot)$ to maximize the functional $J(x,i;\bar{u}_{1}[x,i,u_{2}](\cdot),u_{2}(\cdot))$. This constitutes a two-person zero-sum SLQ  Stackelberg differential game of Markovian regime-switching system, and we denote it as Problem (M-ZLQ). The $\left(u_{1}^{*},u_{2}^{*}(\cdot) \right)\triangleq \left(\bar{u}_{1}[x,i,u_{2}^{*}](\cdot),u_{2}^{*}(\cdot)\right)$ becomes the Stackelberg equilibrium of Problem (M-ZLQ), whose mathematical definition is provided as follows.


	\begin{definition}
	A control pair $(u_{1}^{*}(\cdot),u_{2}^{*}(\cdot))\in \mathcal{U}_{1}\times \mathcal{U}_{2}$ is called a Stackelberg equilibrium of Problem (M-ZLQ) for initial value $(x,i)\in\mathbb{R}^{n}\times\mathcal{S}$ if
	\begin{equation}
	    \inf_{u_{1}(\cdot)\in\mathcal{U}_{1}} J(x,i;u_{1}(\cdot),u_{2}^{*}(\cdot))=J(x,i;u_{1}^{*}(\cdot),u_{2}^{*}(\cdot))=\sup_{u_{2}(\cdot)\in\mathcal{U}_{2}}\inf_{u_{1}(\cdot)\in\mathcal{U}_{1}}J(x,i;u_{1}(\cdot),u_{2}(\cdot))\triangleq V(x,i).
	\end{equation}
	In the above, we denote $V(x,i)$ as the equilibrium value function of Problem (M-ZLQ). If $b(\cdot),\,\sigma(\cdot),\,m=0$, then corresponding criterion functional, equilibrium value function and problem are denoted by $J^{0}\left(x,i; u_{1}(\cdot), u_{2}(\cdot)\right)$, $V^{0}(x,i)$ and Problem (M-ZLQ)$^{0}$, respectively.
	\end{definition}

	Throughout this paper, we will work under the following standard assumptions:

	\textbf{(H1)} For each $i\in \mathcal{S}$, the coefficients of the state equation \eqref{state} satisfy the following
	\begin{equation*}
		\left\{
		\begin{aligned}
			&A(\cdot, i) \in L^{\infty}\left(0, T ; \mathbb{R}^{n \times n}\right), \quad
			B_{k}(\cdot, i) \in L^{\infty}\left(0, T ; \mathbb{R}^{n \times m_{k}}\right), \quad
			b(\cdot) \in L_{\mathbb{F}}^{2}\left(0, T ;\mathbb{R}^{n}\right), \\
			&C(\cdot, i) \in L^{\infty}\left(0, T ; \mathbb{R}^{n \times n}\right), \quad
			D_{k}(\cdot, i) \in L^{\infty}\left(0, T ; \mathbb{R}^{n \times m_{k}}\right), \quad
			\sigma(\cdot) \in L_{\mathbb{F}}^{2}\left(0, T ;\mathbb{R}^{n}\right)\quad k=1,2
		\end{aligned}
		\right.
	\end{equation*}

	\textbf{(H2)}  For each $i\in \mathcal{S}$, the weighting coefficients in the cost functional \eqref{costfunction} satisfy the following:
	\begin{equation*}
		\left\{
		\begin{aligned}
			&Q(\cdot, i) \in L^{\infty}\left(0, T ; \mathbb{S}^{n}\right), \quad
			R_{k}(\cdot, i) \in L^{\infty}\left(0, T ; \mathbb{S}^{m_{k}}\right), \quad k=1,2\\
			&M(\cdot,i) \in \mathbb{S}^{n}, \quad
			m(\cdot) \in L_{\mathbb{F}_{T}}^{2}\left(\Omega ;\mathbb{R}^{n}\right)
		\end{aligned}
		\right.
	\end{equation*}
	Clearly, for any $(x,i)\in\mathbb{R}^{n}\times \mathcal{S}$ and $(u_{1}(\cdot),u_{2}(\cdot))\in \mathcal{U}_{1}\times \mathcal{U}_{2}$, the state equation \eqref{state} admits a unique solution in $\mathcal{S}_{\mathbb{F}}^{2}(0, T ; \mathbb{R}^{n})$ under the assumption (H1).  Then, under assumption (H2), the quadratic performance functional $J(x,i;u_{1}(\cdot),u_{2}(\cdot))$ is well-defined and consequently, Problem (M-ZLQ) is well-posed. Note that the performance coefficients $M(\cdot,i)$, $Q(\cdot,i)$, $R_{1}(\cdot,i)$ and $R_{2}(\cdot,i)$ in the assumption (H2) are not need to be positive (semi)definite matrices. Hence, we are about to solve an indefinite zero-sum Stackelberg SLQ differential games problem.

\section{The follower's problem}\label{section-Follower}
In this section, we first derive the rational control of the follower. Let $u_{2}(\cdot)\in\mathcal{U}_{2}$ be a given announced strategy of the leader. Then the follower aims  to find an optimal reaction $\bar{u}_{1}[x,i,u_{2}](\cdot)$ to minimize the performance functional \eqref{costfunction}, which is equivalent to minimize the following cost functional:
\begin{equation}\label{F-cost}
	\begin{aligned}
		J_{F}\left(x,i ; u_{1}(\cdot)\right) \triangleq  
		\mathbb{E}\bigg\{&\int_{0}^{T}
		\big[\langle Q(s,\alpha(s))X(s), X(s)\rangle+\langle R_{1}(s,\alpha(s))u_{1}(s), u_{1}(s)\rangle\big] ds\\
		&+\langle M(T,\alpha(T)) X(T), X(T)\rangle+2\langle m, X(T)\rangle\bigg\}.
	\end{aligned}
\end{equation}
Therefore, the follower's problem is, in fact, a forward stochastic linear-quadratic problem of the Markovian regime-switching system (M-FSLQ problem, for short), which can be formally summarized as follows.

\textbf{Problem (M-ZLQ-F)} For any given initial state $(x,i)\in \mathbb{R}^{n}\times\mathcal{S}$ and announced strategy $u_{2}(\cdot)\in\mathcal{U}_{2}$ of the leader, find an optimal reaction $\bar{u}_{1}[x,i,u_{2}](\cdot)$ such that
\begin{equation}
    J_{F}\left(x,i;\bar{u}_{1}[x,i,u_{2}](\cdot)\right)=\inf_{u_{1}(\cdot)\in\mathcal{U}_{1}}J_{F}\left(x,i ; u_{1}(\cdot)\right)\triangleq V_{F}(x,i).
\end{equation}
 In the above, we denote $V_{F}(x,i)$ as the value function of the Problem (M-ZLQ-F). If $u_{2}(\cdot)=0$, $b(\cdot)=\sigma(\cdot)=0$, and $m=0$, then the corresponding cost functional,  value function and problem are denoted by $J_{F}^{0}\left(x,i; u_{1}(\cdot)\right)$, $V_{F}^{0}(x,i)$ and Problem (M-ZLQ-F)$^{0}$, respectively. It is worth mentioning that the cost functional $J_{F}\left(x,i ; u_{1}(\cdot)\right)$ and value function $V_{F}(x,i)$ defined here is related to announced strategy $u_{2}(\cdot)$ of the leader since the state process $X(\cdot)$ is affected by $u_{2}(\cdot)$.

We introduce the following assumption before we solve the Problem (M-ZLQ-F).

 \textbf{(H3)} There exists a  constant $\lambda_{F}>0$ such that
\begin{equation}\label{J-uni-convex}
	J^{0}(0,i;u_{1}(\cdot),0)\geq \lambda_{F} \mathbb{E} \int_{0}^{T}|u_{1}(s)|^{2}ds, \qquad  \forall u_{1}(\cdot)\in\mathcal{U}_{1}, \quad \forall  i\in \mathcal{S}.
\end{equation}
The above assumption is called the uniformly convex condition for Problem (M-ZLQ). Under this condition, one can easily verify that
\[J_{F}^{0}(0,i;u_{1}(\cdot))=J^{0}(0,i;u_{1}(\cdot),0)\geq \lambda_{F} \mathbb{E} \int_{0}^{T}|u_{1}(s)|^{2}ds, \qquad  \forall u_{1}(\cdot)\in\mathcal{U}_{1}, \quad \forall  i\in \mathcal{S},\]
which implies that the Problem (M-ZLQ-F) is also uniformly convex.

The (M-FSLQ) problem has been perfectly solved by \citet{Zhang.X.2021_ILQM}. Hence, we can directly obtain the optimal reaction of the follower based on the results in \cite{Zhang.X.2021_ILQM}. For notation simplicity, let
\begin{align*}
	\widehat{S}_{1}(s, i) &:=B_{1}(s, i)^{\top} P(s, i)+D_{1}(s, i)^{\top} P(s, i) C(s, i),\\
	\widehat{R}_{1}(s, i) &:=R_{1}(s, i)+D_{1}(s, i)^{\top} P(s, i) D_{1}(s, i),\\
	\Xi(s,i)&:=D_{1}(s,i)^{\top}P(s,i)D_{2}(s,i),\\
	\widehat{A}(s,i)&:=\widehat{S}_{1}(s, i)^{\top} \widehat{R}_{1}(s, i)^{-1} B_{1}(s, i)^{\top}-A(s, i)^{\top},\\
	\widehat{C}(s,i)&:=\widehat{S}_{1}(s, i)^{\top} \widehat{R}_{1}(s, i)^{-1} D_{1}(s,i)^{\top}-C(s,i)^{\top},\\
	\widehat{H}(s,i)&:=\widehat{C}(s,i)P(s,i)D_{2}(s,i)-P(s,i)B_{2}(s,i),\\
	\widehat{\rho}_{1}(s)&:=B_{1}(s,\alpha(s))^{\top}Y(s)+D_{1}(s,\alpha(s))^{\top}Z(s)+D_{1}(s,\alpha(s))^{\top}P(s,\alpha(s))\sigma,\\
	\widehat{f}(s)&:=\widehat{C}(s,\alpha(s))P(s,\alpha(s))\sigma(s)-P(s,\alpha(s))b(s).
\end{align*}
With the above notations, we further introduce the following CDREs:
\begin{equation}\label{F-RE}
	\left\{
	\begin{aligned}
		\dot{P}(s, i) &=-P(s, i) A(s, i)-A(s, i)^{\top} P(s, i)-C(s, i)^{\top} P(s, i) C(s, i) \\
		&+\widehat{S}_{1}(s, i)^{\top} \widehat{R}_{1}(s, i)^{-1} \widehat{S}_{1}(s, i)-Q(s, i)-\sum_{k=1}^{D} \lambda_{i k}(s) P(s, k), \quad \text { a.e. } \quad s \in[0, T] \\
		P(T, i) &=M(T, i), \quad i \in \mathcal{S},
	\end{aligned}
	\right.
\end{equation}
and linear BSDE:
\begin{equation}\label{L-state}
	\left\{
	\begin{aligned}
		dY(s)&=\big[\widehat{A}(s,\alpha(s))Y(s)+\widehat{C}(s,\alpha(s))Z(s)+\widehat{H}(s,\alpha(s))u_{2}(s)+\widehat{f}(s)\big]ds\\
		&\quad+Z(s) dW(s)+\mathbf{\Gamma}(s)\cdot d\mathbf{\widetilde{N}}(s), \quad s \in[0, T], \\
		Y(T)&=m.
	\end{aligned}
	\right.
\end{equation}
Then, combining Theorem $5.2$, Theorem $6.3$, and Corollary $6.5$ in \citet{Zhang.X.2021_ILQM}, we have the following theorem.

\begin{theorem}\label{thm-M-SG-F}
Let assumptions (H1)-(H3) hold. Then for any given $(x,i)\in \mathbb{R}^{n}\times \mathcal{S}$ and $u_{2}(\cdot)\in\mathcal{U}_{2}$, the following results hold:
\begin{enumerate}
	\item  The CDREs \eqref{F-RE} admit a unique solution $\mathbf{P}(\cdot)\equiv\left[P(\cdot,1),\cdots,P(\cdot,D)\right]\in\mathcal{D}\left( C(0, T; \mathbb{S}^{n})\right)$ such that
	\begin{equation}\label{M-F-SRC}
		\widehat{R}_{1}(\cdot, i)\gg 0 \quad  a.e.\quad \forall i\in\mathcal{S} .
	\end{equation}
	\item The optimal reaction of follower admits a feedback representation:
	\begin{equation}\label{F-control}
		\begin{aligned}
			\bar{u}_{1}[x,i,u_{2}](s)&=-\widehat{R}_{1}(s, \alpha(s))^{-1}
			\big[\widehat{S}_{1}(s,\alpha(s))\bar{X}(s)+\Xi (s,\alpha(s))u_{2}(s)+\widehat{\rho}_{1}(s)\big],
		\end{aligned}
	\end{equation}
	where $\bar{X}(s)$ solves the  SDE:
	\begin{equation}\label{F-optimal-state}
		\left\{
		\begin{aligned}
			d\bar{X}(s)&=\bigg[\big(A(s,\alpha(s))
			-B_{1}(s,\alpha(s))\widehat{R}_{1}(s, \alpha(s))^{-1}
			\widehat{S}_{1}(s,\alpha(s))\big)\bar{X}(s)\\
			&\quad+\bigg(B_{2}(s,\alpha(s))-B_{1}(s,\alpha(s))\widehat{R}_{1}(s, \alpha(s))^{-1}\Xi(s,\alpha(s))\bigg)u_{2}(s)\\
			&\quad-B_{1}(s,\alpha(s))\widehat{R}_{1}(s, \alpha(s))^{-1}\widehat{\rho}_{1}(s)
			+b(s)\bigg]ds\\
			&\quad+\bigg[\bigg(C(s,\alpha(s))-D_{1}(s,\alpha(s))\widehat{R}_{1}(s,\alpha(s))^{-1}\widehat{S}_{1}(s,\alpha(s))\bigg)\bar{X}(s)\\
			&\quad+\bigg(D_{2}(s,\alpha(s))-D_{1}(s,\alpha(s))\widehat{R}_{1}(s, \alpha(s))^{-1}\Xi(s,\alpha(s))\bigg)u_{2}(s)\\
			&\quad-D_{1}(s,\alpha(s))\widehat{R}_{1}(s, \alpha(s))^{-1}\widehat{\rho}_{1}(s)+\sigma(s)\bigg]dW(s),\\
			X(0)&=x, \quad \alpha(0)=i.
		\end{aligned}
		\right.
	\end{equation}
	\item The value function  of the follower is given by
	\begin{equation}\label{F-value function}
		\begin{aligned}
			\hspace{-0.8cm}V_{F}(x,i)
			&=\mathbb{E}\bigg\{\int_{0}^{T}\big[
			2\big<Y(s),B_{2}(s,\alpha(s))u_{2}(s)+b(s)\big>+2\big<Z(s),D_{2}(s,\alpha(s))u_{2}(s)+\sigma(s)\big>\\
			&\qquad+\big<P(s,\alpha(s))\big[D_{2}(s,\alpha(s))u_{2}(s)+\sigma(s)\big],
			D_{2}(s,\alpha(s))u_{2}(s)+\sigma(s)\big>\\
			&\qquad-\big<\widehat{R}_{1}(s,\alpha(s))^{-1}\big[\widehat{\rho}_{1}(s)+\Xi(s,\alpha(s))u_{2}(s)\big],
			\widehat{\rho}_{1}(s)+\Xi(s,\alpha(s))u_{2}(s)\big>\big]ds\\
			&\qquad+\big<P(0,i)x,x\big>+2\big<Y(0),x\big>\bigg\}.
		\end{aligned}
	\end{equation}
\end{enumerate}
\end{theorem}
\section{The  leader's problem}\label{section-Leader}
We now return to solve the leader's problem. Both open-loop solvability and explicit optimal control are obtained in this section.
\subsection{Problem reduction}\label{subsection-simplify}
In previous section, we have obtained the optimal reaction $\bar{u}_{1}[x,i,u_{2}](\cdot)$ (see, equation \eqref{F-control}) of the follower. Knowing this, the leader's problem is to find an optimal control  $u_{2}^{*}(\cdot)\in \mathcal{U}_{2}$ to maximize $J(x,i,\bar{u}_{1}[x,i,u_{2}](\cdot),u_{2}(\cdot))$.  For simplicity, we denote
\begin{align*}
	&G(s,i):=B_{1}(s,i)\widehat{R}_{1}(s,i)^{-1}B_{1}(s,i)^{\top},\quad
	S_{1}(s,i):=D_{1}(s,i)\widehat{R}_{1}(s,i)^{-1}B_{1}(s,i)^{\top},\\
	&S_{2}(s,i):=\Xi(s,i)^{\top}\widehat{R}_{1}(s,i)^{-1}B_{1}(s,i)^{\top}-B_{2}(s,i)^{\top},\quad
	T_{11}(s,i):=D_{1}(s,i)\widehat{R}_{1}(s,i)^{-1}D_{1}(s,i)^{\top},\\
	&T_{21}(s,i)=T_{12}(s,i)^{\top}:=\Xi(s,i)^{\top}\widehat{R}_{1}(s,i)^{-1}D_{1}(s,i)^{\top}-D_{2}(s,i)^{\top},\\
	&T_{22}(s,i):=\Xi(s,i)^{\top}\widehat{R}_{1}(s,i)^{-1}\Xi(s,i)-R_{2}(s,i)-D_{2}(s,i)^{\top}P(s,i))D_{2}(s,i),\\
	&q(s):=B_{1}(s,\alpha(s))\widehat{R}_{1}(s,\alpha(s))^{-1}D_{1}(s,\alpha(s))^{\top}P(s,\alpha(s))\sigma(s)-b(s),\\
	&\rho_{1}(s):=D_{1}(s,\alpha(s))\widehat{R}_{1}(s,\alpha(s))^{-1}D_{1}(s,\alpha(s))^{\top}P(s,\alpha(s))\sigma(s)-\sigma(s),\\
	&\rho_{2}(s):=\Xi(s,\alpha(s))^{\top}\widehat{R}_{1}(s,\alpha(s))^{-1}D_{1}(s,\alpha(s))^{\top}P(s,\alpha(s))\sigma(s)-D_{2}(s,\alpha(s))^{\top}P(s,\alpha(s))\sigma(s).
\end{align*}
Then we can rewrite the $J\left(x,i,\bar{u}_{1}[x,i,u_{2}](\cdot),u_{2}(\cdot)\right)$ as follows:
\begin{equation}\label{costfunction2}
    \begin{aligned}
	&\quad J(x,i,\bar{u}_{1}[x,i,u_{2}](\cdot),u_{2}(\cdot))=V_{F}(x,i)+\mathbb{E}\int_{0}^{T} \big<R_{2}(s,\alpha(s))u_{2}(s),u_{2}(s)\big>\\
	&=\mathbb{E}\bigg\{\int_{0}^{T}\big[-\big<\widehat{R}_{1}(s,\alpha(s))^{-1}D_{1}(s,\alpha(s))^{\top}P(s,\alpha(s))\sigma(s),
	D_{1}(s,\alpha(s))^{\top}P(s,\alpha(s))\sigma(s)\big> \\
	&\qquad\qquad+\big<P(s,\alpha(s))\sigma(s),\sigma(s)\big>\big]ds+\big<P(0,i)x,x\big>\bigg\}-J_{L}(m,i;u_{2}(\cdot)),
\end{aligned}
\end{equation}
where
\begin{equation}\label{L-cost}
	\begin{aligned}
		J_{L}(m,i;u_{2}(\cdot))\triangleq\mathbb{E}
		\int_{0}^{T}
		&\left[
		\left<
		\left(
		\begin{matrix}
			G(s,\alpha(s)) & S_{1}(s,\alpha(s))^{\top} & S_{2}(s,\alpha(s))^{\top} \\
			S_{1}(s,\alpha(s)) & T_{11}(s,\alpha(s)) & T_{12}(s,\alpha(s)) \\
			S_{2}(s,\alpha(s)) & T_{21}(s,\alpha(s)) & T_{22}(s,\alpha(s))
		\end{matrix}
		\right)
		\left(
		\begin{matrix}
			Y(s) \\
			Z(s) \\
			u_{2}(s)
		\end{matrix}
		\right),
		\left(\begin{matrix}
			Y(s) \\
			Z(s) \\
			u_{2}(s)
		\end{matrix}
		\right)
		\right>
		\right.\\
		&\left.+2\left<
		\left(
		\begin{matrix}
			q(s) \\
			\rho_{1}(s) \\
			\rho_{2}(s)
		\end{matrix}
		\right),
		\left(
		\begin{matrix}
			Y(s) \\
			Z(s) \\
			u_{2}(s)
		\end{matrix}
		\right)
		\right>
		\right] ds-2\langle Y(0),x \rangle.
	\end{aligned}
\end{equation}

Note that $J(x,i,\bar{u}_{1}[x,i,u_{2}](\cdot),u_{2}(\cdot))$ is independent of the forward state process $\bar{X}(\cdot)$ and all terms affected by $u_{2}(\cdot)$ are removed in $J_{L}(m,i;u_{2}(\cdot))$. Therefore, finding an optimal control to maximize the performance functional $J(x,i,\bar{u}_{1}[x,i,u_{2}](\cdot),u_{2}(\cdot))$ is equivalent to finding an optimal control to minimize the cost functional $J_{L}(m,i;u_{2}(\cdot))$ with state constraint \eqref{L-state}. Consequently, the leader's problem becomes a backward stochastic linear-quadratic problem of the Markovian regime-switching system (M-BSLQ problem, for short), which can be defined as follows.

\textbf{Problem (M-ZLQ-L)} For any given $(m,i)\in L_{\mathcal{F}_{T}}^{2}(\Omega;\mathbb{R}^{n})\times\mathcal{S}$, find an optimal control $u_{2}^{*}(\cdot)\in\mathcal{U}_{2}$ such that
\begin{equation}\label{L-value}
    J_{L}(m,i;u_{2}^{*}(\cdot))=\inf_{u_{2}(\cdot)\in\mathcal{U}_{2}} J_{L}(m,i;u_{2}(\cdot))\triangleq V_{L}(m,i).
\end{equation}
In the above, we denote $V_{L}(m,i)$ as the value function of the Problem (M-ZLQ-L). If $\widehat{f}(\cdot),\,q(\cdot),\, \rho_{1}(\cdot),\, \rho_{2}(\cdot),\, x=0$, then the corresponding cost functional,  value function and problem are denoted by $J_{L}^{0}\left(m,i; u_{2}(\cdot)\right)$, $V_{L}^{0}(m,i)$ and Problem (M-ZLQ-L)$^{0}$, respectively.

Obviously, we have $\widehat{f}(\cdot),\,q(\cdot),\, \rho_{1}(\cdot),\, \rho_{2}(\cdot)=0$ if $b(\cdot),\,\sigma(\cdot)=0$, and additionally, for any given $i \in \mathcal{S}$, the following holds:
	\begin{equation*}
		\begin{aligned}
			&\widehat{A}(\cdot,i),\text{ } \widehat{C}(\cdot,i) \in L^{\infty}(0,T;\mathbb{R}^{n \times n}),\quad
			\widehat{H}(\cdot,i) \in L^{\infty}(0,T;\mathbb{R}^{n \times m_{2}})\\
			& \left(\begin{array}{ccc}
				G(\cdot,i) & S_{1}(\cdot,i)^{\top} & S_{2}(\cdot,i)^{\top} \\
				S_{1}(\cdot,i) & T_{11}(\cdot,i) & T_{12}(\cdot,i) \\
				S_{2}(\cdot,i) & T_{21}(\cdot,i) & T_{22}(\cdot,i)
			\end{array}
			\right)\in L^{\infty}(0,T;\mathbb{S}^{2n+m_{2}}),\\
			&\widehat{f}(\cdot),\text{ } q(\cdot),\text{ } \rho_{1}(\cdot)\in L_{\mathbb{F}}^{2}(0,T;\mathbb{R}^{n}),\quad
			\rho_{2}(\cdot)\in L_{\mathbb{F}}^{2}(0,T;\mathbb{R}^{m_{2}}).
		\end{aligned}
	\end{equation*}

\subsection{Open-loop solvability}\label{subsection-open-solvability}

The following result characterizes the open-loop solvability of the Problem (M-ZLQ-L) in terms FBSDE.
\begin{theorem}\label{L-openloop-solvability}
Suppose that the assumptions (H1)-(H3) hold. An element $u_{2}^{*}(\cdot)\in \mathcal{U}_{2}$ is the optimal control for the leader if and only if:
\begin{enumerate}
	\item The following convex condition holds:
	\begin{equation}\label{L-convex}
		J_{L}^{0}(0,i;u_{2}(\cdot))\geq 0,\quad \forall u_{2}(\cdot)\in \mathcal{U}_{2};
	\end{equation}
	\item The following stationary condition holds:
	\begin{equation}\label{L-stationary}
		\begin{aligned}
			-\widehat{H}(s,\alpha(s))^{\top}&\phi^{*}(s)+S_{2}(s,\alpha(s))Y^{*}(s)+T_{21}(s,\alpha(s))Z^{*}(s)+T_{22}(s,\alpha(s))u_{2}^{*}(s)+\rho_{2}(s)=0,
		\end{aligned}
	\end{equation}
	where $\left(\phi^{*}(\cdot),Y^{*}(\cdot),Z^{*}(\cdot),\Gamma^{*}(\cdot)\right)$ is the solution of  FBSDE:
	\begin{equation}\label{L-FBSDE}
		\left\{
		\begin{aligned}
			d\phi^{*}(s)&=\big[-\widehat{A}(s,\alpha(s))^{\top}\phi^{*}(s)+G(s,\alpha(s))Y^{*}(s)+S_{1}(s,\alpha(s))^{\top}Z^{*}(s)\\
			&\quad+S_{2}(s,\alpha(s))^{\top}u_{2}^{*}(s)+q(s)\big]ds+\big[-\widehat{C}(s,\alpha(s))^{\top}\phi^{*}(s)+S_{1}(s,\alpha(s))Y^{*}(s)\\
			&\quad+T_{11}(s,\alpha(s))Z^{*}(s)+T_{12}(s,\alpha(s))u_{2}^{*}(s)+\rho_{1}(s)\big]dW(s)\\
			dY^{*}(s)&=\big[\widehat{A}(s,\alpha(s))Y^{*}(s)+\widehat{C}(s,\alpha(s))Z^{*}(s)+\widehat{H}(s,\alpha(s))u_{2}^{*}(s)
			+\widehat{f}(s)\big]ds\\
			&\quad+Z^{*}(s)dW(s)+\mathbf{\Gamma}^{*}(s)\cdot d\mathbf{\widetilde{N}}(s),\\
			\phi^{*}(0)&=-x,\quad \alpha(0)=i,\quad Y^{*}(T)=m.
		\end{aligned}
		\right.
	\end{equation}
\end{enumerate}
\end{theorem}
\begin{proof}
Clearly, $u_{2}^{*}(\cdot)\in \mathcal{U}_{2}$ is an optimal control for leader if and only if for any given $u_{2}(\cdot)\in \mathcal{U}_{2}$,
	\begin{equation}\label{lem2-1}
		J_{L}(m,i;u_{2}^{*}(\cdot)+\varepsilon u_{2}(\cdot))\geq J_{L}(m,i;u_{2}^{*}(\cdot)),
		\quad \forall \varepsilon \in \mathbb{R}.
	\end{equation}
	Let $\left(Y_{0}^{u_{2}},Z_{0}^{u_{2}},\Gamma_{0}^{u_{2}}\right)$ be the solution of the following BSDE:
	\begin{equation}\label{L-state0}
		\left\{
		\begin{aligned}
			dY_{0}^{u_{2}}(s)&=\big[\widehat{A}(s,\alpha(s))Y_{0}^{u_{2}}(s)+\widehat{C}(s,\alpha(s))Z_{0}^{u_{2}}(s)
			+\widehat{H}(s,\alpha(s))u_{2}(s)\big]ds\\
			&\quad+Z_{0}^{u_{2}}(s) dW(s)+\mathbf{\Gamma}_{0}^{u_{2}}(s)\cdot d\mathbf{\widetilde{N}}(s), \quad s \in[0, T], \\
			Y(T)&=0.
		\end{aligned}
		\right.
	\end{equation}
	Then we have
	\begin{equation}\label{lem2-2}
		\begin{aligned}
			&J_{L}(m,i;u_{2}^{*}(\cdot)+\varepsilon u_{2}(\cdot))-J_{L}(m,i;u_{2}^{*}(\cdot))\\
			=&2\varepsilon \mathbb{E}\left\{\int_{0}^{T}
			\left[
			\left<
			\left(
			\begin{array}{lll}
				G(s,\alpha(s)) & S_{1}(s,\alpha(s))^{\top} & S_{2}(s,\alpha(s))^{\top} \\
				S_{1}(s,\alpha(s)) & T_{11}(s,\alpha(s)) & T_{12}(s,\alpha(s)) \\
				S_{2}(s,\alpha(s)) & T_{21}(s,\alpha(s)) & T_{22}(s,\alpha(s))
			\end{array}
			\right)
			\left(
			\begin{array}{c}
				Y^{*}(s) \\
				Z^{*}(s) \\
				u_{2}^{*}(s)
			\end{array}
			\right),
			\left(\begin{array}{c}
				Y_{0}^{u_{2}}(s) \\
				Z_{0}^{u_{2}}(s) \\
				u_{2}(s)
			\end{array}
			\right)
			\right>\right.\right.\\
			&\left.\left.+\left<
			\left(
			\begin{array}{c}
				q(s) \\
				\rho_{1}(s) \\
				\rho_{2}(s)
			\end{array}
			\right),
			\left(\begin{array}{c}
				Y_{0}^{u_{2}}(s) \\
				Z_{0}^{u_{2}}(s) \\
				u_{2}(s)
			\end{array}
			\right)
			\right>\right]ds
			-\langle Y_{0}^{u_{2}}(0),x\rangle
			\right\}+\varepsilon^{2}J_{L}^{0}(0,i;u_{2}(\cdot)).
		\end{aligned}
	\end{equation}
	Applying It\^o's rule  to $\langle Y_{0}^{u_{2}}(s),\phi^{*}(s)\rangle$, one can further simplify the above equation as follows
	\begin{equation}\label{lem2-3}
		\begin{aligned}
			&J_{L}(m,i;u_{2}^{*}(\cdot)+\varepsilon u_{2}(\cdot))-J_{L}(t,m,i;u_{2}^{*}(\cdot))\\
			=& 2\varepsilon \mathbb{E}\int_{0}^{T}\bigg\{
			\big<-\widehat{H}(s)^{\top}\phi^{*}(s)+S_{2}(s)Y^{*}(s)
			+T_{21}(s)Z^{*}(s)+T_{22}(s)u_{2}^{*}(s)+\rho_{2}(s),u_{2}(s)\big>
			\bigg\}ds\\
			& +\varepsilon^{2}J_{L}^{0}(0,i;u_{2}^{*}(\cdot)),\quad \forall u_{2}(\cdot)\in\mathcal{U}_{2},\quad \forall\varepsilon\in\mathbb{R}.
		\end{aligned}
	\end{equation}
	Therefore, equation \eqref{lem2-1} holds if and only if the convex condition \eqref{L-convex}  and the  stationary condition \eqref{L-stationary} hold. This completes the proof.
	\end{proof}

	The equation \eqref{L-FBSDE}, together with the stationary condition \eqref{L-stationary}, constitute a coupled FBSDE system, which is referred to as the optimality system for the Problem (M-ZLQ-L).

	Although the above theorem provides a necessary and sufficient condition for open-loop solvability of Problem (M-ZLQ-L), it isn't easy to verify. Therefore, we will focus on identifying an easily verifiable sufficient condition that guarantees the unique solvability of Problem (M-ZLQ-L). To this end, we rewrite the leader's cost functional from the Hilbert space point of view.

	Denote $\left(Y^{u_{2}}(\cdot;m),Z^{u_{2}}(\cdot;m),\Gamma^{u_{2}}(\cdot;m)\right)$ as the solution of equation \eqref{L-state} and $\left(Y_{0}^{u_{2}}(\cdot;m),Z_{0}^{u_{2}}(\cdot;m),\Gamma_{0}^{u_{2}}(\cdot;m)\right)$ as the corresponding homogeneous solution when $\widehat{f}(\cdot)=0$. Then,  based on the linearity of BSDE \eqref{L-state}, one can verify that the solution to BSDE \eqref{L-state} admits the following decomposition:
	$$
	\left\{
	\begin{aligned}
		Y^{u_{2}}(\cdot;m)&=Y_{0}^{u_{2}}(\cdot;0)+Y_{0}^{0}(\cdot;m)+Y^{0}(\cdot;0),\\
		Z^{u_{2}}(\cdot;m)&=Z_{0}^{u_{2}}(\cdot;0)+Z_{0}^{0}(\cdot;m)+Z^{0}(\cdot;0),\\
		\mathbf{\Gamma}^{u_{2}}(\cdot;m)&=\mathbf{\Gamma}_{0}^{u_{2}}(\cdot;0)+\mathbf{\Gamma}_{0}^{0}(\cdot;m)+\mathbf{\Gamma}^{0}(\cdot;0).
	\end{aligned}
	\right.$$
	With the above notations, we define the following bounded linear operators:
	$$
	\begin{array}{ll}
		\mathcal{D}_{0}^{\alpha}:\text{ }\mathcal{U}_{2}\text{ }\rightarrow\text{ } L_{\mathcal{F}_{0}}^{2}(\Omega;\mathbb{R}^{n}),
		&\mathcal{D}_{0}^{\alpha}u_{2}\triangleq Y_{0}^{u_{2}}(0;0),\\
		\mathcal{D}_{1}^{\alpha}:\text{ }\mathcal{U}_{2}\text{ }\rightarrow\text{ } \mathcal{S}_{\mathbb{F}}^{2}(0,T;\mathbb{R}^{n}),
		&\mathcal{D}_{1}^{\alpha}u_{2}\triangleq Y_{0}^{u_{2}}(\cdot;0),\\
		\mathcal{D}_{2}^{\alpha}:\text{ }\mathcal{U}_{2}\text{ }\rightarrow\text{ }
		L_{\mathbb{F}}^{2}(0,T;\mathbb{R}^{n}),
		&\mathcal{D}_{2}^{\alpha}u_{2}\triangleq Z_{0}^{u_{2}}(\cdot;0),\\
		\mathcal{M}_{0}^{\alpha}:\text{ }L_{\mathcal{F}_{T}}^{2}(\Omega;\mathbb{R}^{n})\text{ }\rightarrow\text{ } L_{\mathcal{F}_{0}}^{2}(\Omega;\mathbb{R}^{n}),
		&\mathcal{M}_{0}^{\alpha}m\triangleq Y_{0}^{0}(0;m),\\
		\mathcal{M}_{1}^{\alpha}:\text{ }L_{\mathcal{F}_{T}}^{2}(\Omega;\mathbb{R}^{n})\text{ }\rightarrow\text{ } \mathcal{S}_{\mathbb{F}}^{2}(0,T;\mathbb{R}^{n}),
		&\mathcal{M}_{1}^{\alpha}m\triangleq Y_{0}^{0}(\cdot;m),\\
		\mathcal{M}_{2}^{\alpha}:\text{ }L_{\mathcal{F}_{T}}^{2}(\Omega;\mathbb{R}^{n})\text{ }\rightarrow\text{ } L_{\mathbb{F}}^{2}(0,T;\mathbb{R}^{n}),
		&\mathcal{M}_{2}^{\alpha}m\triangleq Z_{0}^{0}(\cdot;m),
	\end{array}
	$$
	and
	\begin{align*}
	 &K^{\alpha}\triangleq \left(
		\begin{matrix}
			G(\cdot,\alpha(\cdot)) & S_{1}(\cdot,\alpha(\cdot))^{\top} & S_{2}(\cdot,\alpha(\cdot))^{\top}\\
			S_{1}(\cdot,\alpha(\cdot)) & T_{11}(\cdot,\alpha(\cdot)) & T_{12}(\cdot,\alpha(\cdot))\\
			S_{2}(\cdot,\alpha(\cdot)) & T_{21}(\cdot,\alpha(\cdot)) & T_{22}(\cdot,\alpha(\cdot))
		\end{matrix}
		\right),
		\quad \xi\triangleq \left(
		\begin{array}{c}
			q(\cdot) \\
			\rho_{1}(\cdot) \\
			\rho_{2}(\cdot)
		\end{array}
		\right),
		\quad c \triangleq Y^{0}(0;0),\\
		&\mathcal{C}^{\alpha}\triangleq
		\left(
		\begin{array}{c}
			Y^{0}(\cdot;0)\\
			Z^{0}(\cdot;0)\\
			0
		\end{array}
		\right),
		\quad \mathcal{D}^{\alpha}\triangleq
		\left(
		\begin{array}{c}
			\mathcal{D}_{1}^{\alpha}\\
			\mathcal{D}_{2}^{\alpha}\\
			\mathcal{I}
		\end{array}
		\right),
		\quad \mathcal{M}^{\alpha}\triangleq
		\left(
		\begin{array}{c}
			\mathcal{M}_{1}^{\alpha}\\
			\mathcal{M}_{2}^{\alpha}\\
			0
		\end{array}
		\right),
	\end{align*}
	where $\mathcal{I}$ represents the identity operator.

	Now, let $\mathcal{A}^{*}$ be the adjoint operator of a linear operator $\mathcal{A}$. Then the cost functional of leader \eqref{L-cost}  can be rewritten  as follows:
	\begin{equation}\label{L-cost-hilbet}
		\begin{aligned}
			J_{L}\left(m,i;u_{2}(\cdot)\right)&=
			\big<K^{\alpha}
			\left(\mathcal{D}^{\alpha}u_{2}+\mathcal{M}^{\alpha}m+\mathcal{C}^{\alpha}\right),
			\mathcal{D}^{\alpha}u_{2}+\mathcal{M}^{\alpha}m+\mathcal{C}^{\alpha}\big>
			+2\big<\xi,\mathcal{D}^{\alpha}u_{2}+\mathcal{M}^{\alpha}m+\mathcal{C}^{\alpha}\big>\\
			&\quad-2\big<\mathcal{D}_{0}^{\alpha}u_{2}+\mathcal{M}_{0}^{\alpha}m+c,x\big>\\
			&=\left<\mathcal{R}_{L}^{\alpha}u_{2},u_{2}\right>+2\left<\mathcal{N}_{L}^{\alpha}m,u_{2}\right>
			+\left<\mathcal{L}_{L}^{\alpha}m,m\right>
			+2\left<\upsilon_{L}^{\alpha},u_{2}\right>
			+2\left<\omega_{L}^{\alpha},m\right>+c_{L}^{\alpha},
		\end{aligned}
	\end{equation}
	where
	$$
	\left\{
	\begin{aligned}
		&\mathcal{R}_{L}^{\alpha}=\left(\mathcal{D}^{\alpha}\right)^{*}K^{\alpha}\mathcal{D}^{\alpha},
		\quad
		\mathcal{N}_{L}^{\alpha}=\left(\mathcal{D}^{\alpha}\right)^{*}K^{\alpha}\mathcal{M}^{\alpha},
		\quad
		\mathcal{L}_{L}^{\alpha}=\left(\mathcal{M}^{\alpha}\right)^{*}K^{\alpha}\mathcal{M}^{\alpha},\\
		&\upsilon_{L}^{\alpha}=\left(\mathcal{D}^{\alpha}\right)^{*}\left(K^{\alpha}\mathcal{C}^{\alpha}
		+\xi\right)-\left(\mathcal{D}_{0}^{\alpha}\right)^{*}x,
		\quad
		\omega_{L}^{\alpha}=\left(\mathcal{M}^{\alpha}\right)^{*}\left(K^{\alpha}\mathcal{C}^{\alpha}+\xi\right)-\left(\mathcal{M}_{0}^{\alpha}\right)^{*}x,\\
		&
		c_{L}^{\alpha}=\big<K^{\alpha}\mathcal{C}^{\alpha},\mathcal{C}^{\alpha}\big>
		+2\big<\xi,\mathcal{C}^{\alpha}\big>-2\big<c,x\big>.
	\end{aligned}
	\right.$$

To guarantee the optimality system for Problem (M-ZLQ-L) exists a unique solution, we introduce the following assumption:

\textbf{(H4)} There exists a constant $\lambda_{L}>0$ such that
	\begin{equation}\label{J-uni-concave}
		J^{0}(0,i;0,u_{2})\leq -\lambda_{L} \mathbb{E} \int_{0}^{T}|u_{2}(s)|^{2}ds, \qquad  \forall u_{2}(\cdot)\in\mathcal{U}_{2},\quad \forall i\in\mathcal{S}.
	\end{equation}
	The equation \eqref{J-uni-concave} is called the uniformly concave condition for Problem (M-ZLQ), which can ensure that the Problem (M-ZLQ-L) admits a unique optimal control.

	\begin{proposition}\label{prop-L-open-solvability}
		Let assumptions (H1)-(H4) hold, then Problem (M-ZLQ-L) admits a unique optimal control. Consequently, the FBSDE \eqref{L-FBSDE}  admits a unique adapted solution $\left(\phi^{*}(\cdot),Y^{*}(\cdot),Z^{*}(\cdot),\Gamma^{*}(\cdot)\right)$ satisfying the stationary condition \eqref{L-stationary}.
	\end{proposition}

	\begin{proof}
			Let $\bar{u}_{1}[0,i,u_{2}](\cdot)$ be the rational control of the follower for Problem (M-ZLQ-F)$^{0}$. Then, we have 
			\begin{equation}\label{eq:J0JL0}
				-J_{L}^{0}(0,i;u_{2}(\cdot))=J^{0}(0,i;\bar{u}_{1}[0,i,u_{2}](\cdot),u_{2}(\cdot)).
			\end{equation}
			Since
			$$ J^{0}(0,i;\bar{u}_{1}[0,i,u_{2}](\cdot),u_{2}(\cdot))\leq J^{0}(0,i;u_{1}(\cdot),u_{2}(\cdot)), \qquad \forall u_{1}(\cdot)\in \mathcal{U}_{1},\quad a.s.,$$
			 taking $u_{1}(\cdot)=0$ and combing the equation \eqref{eq:J0JL0} lead to 
			\begin{equation*}
				\begin{aligned}
					J_{L}^{0}(0,i;u_{2})&=-J^{0}(0,i;\bar{u}_{1}[0,i,u_{2}](\cdot),u_{2}(\cdot))\geq -J^{0}(0,i;0,u_{2}(\cdot))
					\geq \lambda_{L} \mathbb{E} \int_{0}^{T}|u_{2}(s)|^{2}ds, \qquad  \forall u_{2}(\cdot)\in\mathcal{U}_{2}.
				\end{aligned}
			\end{equation*}
		Combining with the representation \eqref{L-cost-hilbet}, we can see that the map $u_{2}(\cdot) \mapsto J_{L}(m,i;u_{2}(\cdot))$ is strictly convex and
		$$J_{L}(m,i;u_{2}(\cdot))\rightarrow\infty\text{ as }\mathbb{E}\int_{0}^{T}|u_{2}(s)|^{2}ds\rightarrow\infty.$$
		Therefore, the unique solvability of Problem (M-ZLQ-L) for any given $(m,i)\in L_{\mathcal{F}_{T}}^{2}(\Omega;\mathbb{R}^{n})\times\mathcal{S}$ follows from the fundamental theorem of convex analysis. Furthermore, combining with Theorem \ref{L-openloop-solvability}, we obtain that the optimality system \eqref{L-FBSDE} admits a unique solution satisfying the stationary condition \eqref{L-stationary}.
	\end{proof}
Note that Theorem \ref{thm-M-SG-F} and Proposition \ref{prop-L-open-solvability} imply that the Problem (M-ZLQ) admits a unique Stackelberg equilibrium under the assumptions (H1)-(H4). The following lemma further shows that the weighting matrix $T_{22}(\cdot,\cdot)$ in \eqref{L-cost} is a uniform positive-definite function under the assumptions (H1)-(H4). As we can see in the next section, such a result plays a crucial role in constructing the explicit optimal control for the leader.

\begin{lemma}\label{L-lem}
	Suppose assumptions (H1)-(H4) hold, then $T_{22}(\cdot,\cdot) \gg 0$.
\end{lemma}

\begin{proof}
	Let us first consider the following M-FSLQ problem: for any given initial state $(a,i)\in \mathbb{R}^{n}\times\mathcal{S}$,  find a control pair $(v^{*}(\cdot),u_{2}^{*}(\cdot))\in L_{\mathbb{F}}^{2}\left(0,T;\mathbb{R}^{n}\right)\times \mathcal{U}_{2}$ such that
	\begin{equation}\label{M-FSLQ-cost}
	 \begin{aligned}
		&\quad J_{\lambda}\left(a,i;u_{2}^{*}(\cdot),v^{*}(\cdot)\right)
		=\inf_{(v,u_{2})}J_{\lambda}\left(a,i;u_{2}(\cdot),v(\cdot)\right)\\
		&\triangleq \inf_{(v,u_{2})}\mathbb{E}\left\{\int_{0}^{T}
		\left<
		\left(
		\begin{matrix}
			G(s,\alpha(s)) & S_{1}(s,\alpha(s))^{\top} & S_{2}(s,\alpha(s))^{\top} \\
			S_{1}(s,\alpha(s)) & T_{11}(s,\alpha(s)) & T_{12}(s,\alpha(s)) \\
			S_{2}(s,\alpha(s)) & T_{21}(s,\alpha(s)) & T_{22}(s,\alpha(s))
		\end{matrix}
		\right)
		\left(
		\begin{matrix}
			U(s) \\
			v(s) \\
			u_{2}(s)
		\end{matrix}
		\right),
		\left(
		\begin{matrix}
			U(s) \\
			v(s) \\
			u_{2}(s)
		\end{matrix}
		\right)
		\right>ds+\lambda |U(T)|^{2}
		\right\},
	\end{aligned}
	\end{equation}
	where $U(\cdot)$ is determined by the following stochastic differential equation: 
	\begin{equation}\label{M-FSLQ-state}
		\left\{
		\begin{aligned}
			dU(s)&=\big[\widehat{A}(s,\alpha(s))U(s)+\widehat{C}(s,\alpha(s))v(s)+\widehat{H}(s,\alpha(s))u_{2}(s)\big]ds\\
			&\quad+v(s)dW(s), \quad s \in[0, T], \\
			U(0)&=a, \quad \alpha(0)=i.
		\end{aligned}
		\right.
	\end{equation}
	We denote the above M-FSLQ problem as \textbf{Problem (M-LQ)$_{\lambda}$}, whose  corresponding CDREs is
	\begin{equation}\label{M-FSLQ-RE}
		\left\{
		\begin{aligned}
			&\dot{P}_{\lambda}\left( s,i \right) =-P_{\lambda}\left( s,i \right) \widehat{A}\left( s,i \right) -\widehat{A}\left( s,i \right) ^{\top}P_{\lambda}\left( s,i \right) -G\left( s,i \right)
			+\left( \begin{array}{c}
				\widehat{C}\left( s,i \right) ^{\top}P_{\lambda}\left( s,i \right) +S_1\left( s,i \right)\\
				\widehat{H}\left( s,i \right) ^{\top}P_{\lambda}\left( s,i \right) +S_2\left( s,i \right)\\
			\end{array} \right) ^{\top}\\
			&\times\left( \begin{matrix}
				T_{11}\left( s,i \right) +P_{\lambda}\left( s,i \right)&		T_{12}\left( s,i \right)\\
				T_{21}\left( s,i \right)&		T_{22}\left( s,i \right)\\
			\end{matrix} \right)^{-1} \left( \begin{array}{c}
				\widehat{C}\left( s,i \right) ^{\top}P_{\lambda}\left( s,i \right) +S_1\left( s,i \right)\\
				\widehat{H}\left( s,i \right) ^{\top}P_{\lambda}\left( s,i \right) +S_2\left( s,i \right)\\
			\end{array} \right) -\sum_{k=1}^D{\lambda _{ik}P_{\lambda}\left( s,k \right) }\\
			&P_{\lambda}(T)=\lambda I \quad i \in \mathcal{S} \quad \text { a.e. } \quad s \in[0, T].
		\end{aligned}
		\right.
	\end{equation}

	According to Theorem $6.3$ in \citet{Zhang.X.2021_ILQM}, if $J_{\lambda}\left(0,i;u_{2}(\cdot),v(\cdot)\right)$ 
	is uniformly convex, then the unique solution of Riccati equation \eqref{M-FSLQ-RE} satisfies
	$$
	\left( \begin{matrix}
		T_{11}\left( \cdot,i \right) +P_{\lambda}\left( \cdot,i \right)&
		T_{12}\left( \cdot,i \right)\\
		T_{21}\left( \cdot,i \right)&		T_{22}\left( \cdot,i \right)\\
	\end{matrix} \right)\gg 0,
	$$
	which implies that $T_{22}(\cdot,\cdot) \gg 0$ and so the proof is complete.

	Next, we will prove that $J_{\lambda}\left(0,i;u_{2}(\cdot),v(\cdot)\right)$ is uniformly convex when $\lambda$ is large enough, i.e., for any given $i\in\mathcal{S}$, there exists  a constant $c>0$ such that
	\begin{equation}\label{M-FSLQ-UCC}
		J_{\lambda}\left(0,i;u_{2}(\cdot),v(\cdot)\right)\geq
		c\mathbb{E}\int_{0}^{T}\big[|v(s)|^{2}+|u_{2}(s)|^{2}\big]dt,
		\quad \forall (v(\cdot),u_{2}(\cdot))\in L_{\mathbb{F}}^{2}(0,T;\mathbb{R}^{n})\times \mathcal{U}_{2}.
	\end{equation}
	To this end, for any given $\left(a,i,v,u_{2}\right)\in \mathbb{R}^{n}\times \mathcal{S}\times L_{\mathbb{F}}^{2}(0,T;\mathbb{R}^{n})\times \mathcal{U}_{2}$,
	let $U(\cdot)$ be the solution of \eqref{M-FSLQ-state}, 
	and consider the following two BSDEs
	\begin{equation}\label{thm-4-1-1}
		\left\{
		\begin{aligned}
			\begin{aligned}
				dY_{1}(s)&=\big[\widehat{A}(s,\alpha(s))Y_{1}(s)+\widehat{C}(s,\alpha(s))Z_{1}(s)\big]ds+Z_{1}(s) dW(s)+\mathbf{\Gamma}_{1}(s)\cdot d\mathbf{\widetilde{N}}(s), \quad s \in[0, T], \\
				Y_{1}(T)&=U(T),
			\end{aligned}
		\end{aligned}
		\right.
	\end{equation}
	and
	\begin{equation}\label{thm-4-1-2}
		\left\{
		\begin{aligned}
			\begin{aligned}
				dY_{2}(s)&=\big[\widehat{A}(s,\alpha(s))Y_{2}(s)+\widehat{C}(s,\alpha(s))Z_{2}(s)+\widehat{H}(s,\alpha(s))u_{2}(s)\big]ds+Z_{2}(s) dW(s)+\mathbf{\Gamma}_{2}(s)\cdot d\mathbf{\widetilde{N}}(s), \quad s \in[0, T], \\
				Y_{2}(T)&=0.
			\end{aligned}
		\end{aligned}
		\right.
	\end{equation}
	Then, according to the unique solvabilities of SDE \eqref{M-FSLQ-state} and BSDEs \eqref{thm-4-1-1}-\eqref{thm-4-1-2}, we have 
	\begin{equation}\label{wu}
		Y_{1}(s)+Y_{2}(s)=U(s),\quad Z_{1}(s)+Z_{2}(s)=v(s),\quad \mathbf{\Gamma}_{1}(s)+\mathbf{\Gamma}_{2}(s)=0, \quad s\in[0,T].
	\end{equation}
	Let
	\begin{align*}
		&\beta_{1}(s)=\left(Y_{1}(s),Z_{1}(s),0\right)^{\top},\quad
		\beta_{2}(s)=\left(Y_{2}(s),Z_{2}(s),u_{2}\right)^{\top},\\
		&K(s,\alpha(s))=\left(\begin{array}{lll}
			G(s,\alpha(s)) & S_{1}(s,\alpha(s))^{\top} & S_{2}(s,\alpha(s))^{\top} \\
			S_{1}(s,\alpha(s)) & T_{11}(s,\alpha(s)) & T_{12}(s,\alpha(s)) \\
			S_{2}(s,\alpha(s)) & T_{21}(s,\alpha(s)) & T_{22}(s,\alpha(s))
		\end{array}
		\right).
	\end{align*}
	Then, one has
	\begin{equation}\label{thm-4-2-3}
		J_{L}^{0}(0,i;u_{2}(\cdot))=\mathbb{E}\int_{0}^{T}\big<K(s,\alpha(s))\beta_{2}(s),\beta_{2}(s)\big>ds,
	\end{equation}
	and
	\begin{equation}\label{thm-4-2-4}
		\begin{aligned}
			J_{\lambda}\left(a,i;u_{2}(\cdot),v(\cdot)\right)
			&=\mathbb{E}\left\{\lambda|U(T)|^{2}+\int_{0}^{T}\big<K(s,\alpha(s))
			\bigl(\beta_{1}(s)+\beta_{2}(s)\bigr),\bigl(\beta_{1}(s)+\beta_{2}(s)\bigr)\big>ds\right\}\\
			&=J_{L}^{0}(0,i;u_{2}(\cdot))+\mathbb{E}\bigg\{
			\lambda|U(T)|^{2}+\int_{0}^{T}\big<K(s,\alpha(s))\beta_{1}(s),\beta_{1}(s)\big>ds\\
			&\quad+2\int_{0}^{T}\big<K(s,\alpha(s))\beta_{1}(s),\beta_{2}(s)\big>ds
			\bigg\}.
		\end{aligned}
	\end{equation}
	Note that $K(s,\alpha(s))\in L_{\mathbb{F}}^{\infty}\left(0,T,\mathbb{S}^{2n+m_{2}}\right)$. Therefore, there exists a constant $L>0$ such that $|K(s,\alpha(s))|\leq L$ for a.e.  $s\in[0,T]$. By the Cauchy-Schwarz inequality, we have
	\begin{equation}\label{thm-4-2-5}
		\begin{aligned}
			&\quad\left|\mathbb{E}\bigg\{
			\int_{0}^{T}\big<K(s,\alpha(s))\beta_{1}(s),\beta_{1}(s)\big>ds
			+2\int_{0}^{T}\big<K(s,\alpha(s))\beta_{1}(s),\beta_{2}(s)\big>ds\bigg\}\right|\\
			&\leq L \bigg\{\mathbb{E}\int_{0}^{T}|\beta_{1}(s)|^{2}ds
			+2\mathbb{E}\int_{0}^{T}|\beta_{1}(s)||\beta_{2}(s)|ds
			\bigg\}\\
			&\leq L \bigg\{(\mu + 1)\mathbb{E}\int_{0}^{T}|\beta_{1}(s)|^{2}ds
			+\frac{1}{\mu} \mathbb{E}\int_{0}^{T}|\beta_{2}(s)|^{2}ds
			\bigg\},
		\end{aligned}
	\end{equation}
	where $\mu >0$ is a constant to be chosen later. On the other hand, based on the estimation of the solution to BSDE, we can derive the following inequality with a sufficiently large value $L$:
	\begin{equation}\label{thm-4-2-6}
		\mathbb{E}\int_{0}^{T}|\beta_{1}(s)|^{2}ds\leq
		L \mathbb{E}|U(T)|^{2},\quad
		\mathbb{E}\int_{0}^{T}|\beta_{2}(s)|^{2}ds\leq
		L \mathbb{E}\int_{0}^{T}|u_{2}(s)|^{2}ds.
	\end{equation}
	On the other hand, noting that
	\begin{align*}
		\mathbb{E}\int_{0}^{T}|v(s)|^{2}ds
		&=\mathbb{E}\int_{0}^{T}|Z_{1}(s)+Z_{2}(s)|^{2}ds\\
		&\leq 2\left[\mathbb{E}\int_{0}^{T}|Z_{1}(s)|^{2}ds
		+\mathbb{E}\int_{0}^{T}|Z_{2}(s)|^{2}ds\right]\\
		&\leq 2\left[\mathbb{E}\int_{0}^{T}|\beta_{1}(s)|^{2}ds
		+\mathbb{E}\int_{0}^{T}|\beta_{2}(s)|^{2}ds\right]\\
		&\leq 2L\left[\mathbb{E}|U(T)|^{2}+\mathbb{E}\int_{0}^{T}|u_{2}(s)|^{2}ds\right],
	\end{align*}
	yields
	\begin{equation}\label{thm-4-2-7}
		\mathbb{E}|U(T)|^{2}\geq \frac{1}{2L}\mathbb{E}\int_{0}^{T}|v(s)|^{2}ds
		-\mathbb{E}\int_{0}^{T}|u_{2}(s)|^{2}ds.
	\end{equation}
	Combining equation \eqref{thm-4-2-5} - \eqref{thm-4-2-6}, we  obtain from \eqref{thm-4-2-5} that 
	\begin{equation}\label{thm-4-2-8}
		\begin{aligned}
			&\quad J_{\lambda}\left(a,i;u_{2}(\cdot),v(\cdot)\right)\\
			&\geq\lambda_{L} \mathbb{E}\int_{0}^{T}|u_{2}(s)|^{2}ds
			+\mathbb{E}\bigg\{
			\lambda|U(T)|^{2}-L \bigg\{(\mu + 1)\mathbb{E}\int_{0}^{T}|\beta_{1}(s)|^{2}ds
			+\frac{1}{\mu} \mathbb{E}\int_{0}^{T}|\beta_{2}(s)|^{2}ds
			\bigg\}
			\bigg\}\\
			&\geq\lambda_{L} \mathbb{E}\int_{0}^{T}|u_{2}(s)|^{2}ds
			+\mathbb{E}\bigg\{
			\lambda|U(T)|^{2}-L^{2} \bigg\{(\mu + 1)\mathbb{E}|U(T)|^{2}
			+\frac{1}{\mu} \mathbb{E}\int_{0}^{T}|u_{2}(s)|^{2}ds
			\bigg\}
			\bigg\} \\
			&=(\lambda_{L} -\frac{L^{2}}{\mu} )\mathbb{E}\int_{0}^{T}|u_{2}(s)|^{2}ds
			+(\lambda-L^{2}(\mu+1))\mathbb{E}|U(T)|^{2}.
		\end{aligned}
	\end{equation}
	Let
	\[\mu=\frac{2L^{2}}{\lambda_{L}},\qquad  \lambda\geq\lambda_{0}:=\frac{\lambda_{L} }{4}+L^{2}(\mu + 1).\]
	Then, with the help of equation \eqref{thm-4-2-6},  we have
	\begin{equation}\label{thm-4-2-9}
		\begin{aligned}
			J_{\lambda}\left(a,i;u_{2}(\cdot),v(\cdot)\right)
			&\geq\frac{\lambda_{L}}{2} \mathbb{E}\int_{0}^{T}|u_{2}(s)|^{2}ds
			+\frac{\lambda_{L}}{4}\mathbb{E}|U(T)|^{2}\\
			&\geq\frac{\lambda_{L}}{4} \mathbb{E}\int_{0}^{T}|u_{2}(s)|^{2}ds
			+\frac{\lambda_{L}}{8L}\mathbb{E}\int_{0}^{T}|v(s)|^{2}ds\\
			&\geq c\mathbb{E}\int_{0}^{T}\bigl[
			|u_{2}(s)|^{2}ds+|v(s)|^{2}ds\bigr]ds, \quad
			\forall (v(\cdot),u_{2}(\cdot))\in L_{\mathbb{F}}^{2}(0,T;\mathbb{R}^{n})\times \mathcal{U}_{2},
		\end{aligned}
	\end{equation}
	where $c=\min\{\frac{\lambda_{L}}{4},\frac{\lambda_{L}}{8L}\}>0$. This completes the proof.
\end{proof}

\begin{remark}
From the proof of Lemma \ref{L-lem}, we find that the uniform convexity of Problem (M-ZLQ-L) implies the uniform convexity of a family M-FSLQ problems: $\left\{\text{Problem (M-LQ)}_{\lambda}: \lambda\geq\lambda_{0}\right\}$. Such a finding also has been shown in \citet{Sun.J.R.2021_IBLQ}, in which they constructed the solution to Riccati equations for a BSLQ problem primarily based on this relation. However, as we can see in  Section \ref{section-RE-solvability}, such a method is ineffective for M-BSLQ problem.
\end{remark}

\subsection{Construction of the optimal control for the leader}\label{section-construction}

In this section,  we would like to construct the optimal control for the leader based on Theorem \ref{L-openloop-solvability}. By Lemma \ref{L-lem}, we know that $T_{22}(\cdot)\gg 0$. This enables us to do the transformations:
\begin{equation}\label{transformations}
\left\{
\begin{aligned}
&\upsilon(s)=u_{2}(s)+T_{22}(s,\alpha(s))^{-1}T_{21}(s,\alpha(s))Z(s),\\
&\widetilde{\rho}_{1}(s)=\rho_{1}(s)-T_{12}(s,\alpha(s))T_{22}(s,\alpha(s))^{-1}\rho_{2}(s),\\
&\widetilde{F}(s,i)=\widehat{C}(s,i)-\widehat{H}(s,i)T_{22}(s,i)^{-1}T_{21}(s,i),\\
&\widetilde{S}_{1}(s,i)=S_{1}(s,i)-T_{12}(s,i)T_{22}(s,i)^{-1}S_{2}(s,i),\\
&\widetilde{T}_{11}(s,i)=T_{11}(s,i)-T_{12}(s,i)T_{22}(s,i)^{-1}T_{21}(s,i).
\end{aligned}
\right.
\end{equation}
Then, according to unique solvability, we can obtain that the solution to \eqref{L-state} also solves the following BSDE:
\begin{equation}\label{L-state-2}
\left\{
\begin{aligned}
dY(s)&=\big[\widehat{A}(s,\alpha(s))Y(s)+\widetilde{F}(s,\alpha(s))Z(s)+\widehat{H}(s,\alpha(s))\upsilon(s)
+\widehat{f}(s)\big]ds\\
&\quad+Z(s) dW(s)+\mathbf{\Gamma}(s)\cdot d\mathbf{\widetilde{N}}(s), \quad s \in[0, T], \\
Y(T)&=m,
\end{aligned}
\right.
\end{equation}
and additionally, we have
\begin{equation}\label{L-cost-2}
\begin{aligned}
J_{L}(m,i;u_{2}(\cdot))
&=\mathbb{E}\left\{2\langle Y(0),-x \rangle+
\int_{0}^{T}
\left[
\left<
\left(
\begin{matrix}
G &S_{1}^{\top} & S_{2}^{\top} \\
S_{1} & T_{11} & T_{12} \\
S_{2} & T_{21} & T_{22}
\end{matrix}
\right)
\left(
\begin{matrix}
Y \\
Z \\
u_{2}
\end{matrix}
\right),
\left(\begin{matrix}
Y \\
Z \\
u_{2}
\end{matrix}
\right)
\right>+2\left<
\left(
\begin{matrix}
q\\
\rho_{1}\\
\rho_{2}
\end{matrix}
\right),
\left(
\begin{matrix}
Y\\
Z\\
u_{2}
\end{matrix}
\right)
\right>
\right]ds\right\}\\
&=\mathbb{E}\left\{2\langle Y(0),-x \rangle+
\int_{0}^{T}
\left[
\left<
\left(
\begin{matrix}
G &\widetilde{S}_{1}^{\top} & S_{2}^{\top} \\
\widetilde{S}_{1} & \widetilde{T}_{11} &  0 \\
S_{2} & 0 & T_{22}
\end{matrix}
\right)
\left(
\begin{matrix}
Y \\
Z \\
\upsilon
\end{matrix}
\right),
\left(\begin{matrix}
Y \\
Z \\
\upsilon
\end{matrix}
\right)
\right>+2\left<
\left(
\begin{matrix}
q \\
\widetilde{\rho}_{1} \\
\rho_{2}
\end{matrix}
\right),
\left(
\begin{matrix}
Y \\
Z \\
\upsilon
\end{matrix}
\right)
\right>
\right] ds\right\}\\
&\triangleq \widetilde{J}_{L}(m,i;\upsilon(\cdot)).
\end{aligned}
\end{equation}
In the above, we suppress the argument $(s,\alpha(s))$ or $s$ for
convenience and the suppressed notations will be frequently used in the rest of this paper when no confusion arises. We denote the problem with state \eqref{L-state-2} and cost functional $\widetilde{J}_{L}(m,i;\upsilon(\cdot))$ defined in \eqref{L-cost-2} as \textbf{Problem (M-BLQ)}. Obviously, the Problem (M-ZLQ-L) is equivalent to Problem (M-BLQ) in the sense that their corresponding optimal controls (denoted as $u_{2}^{*}(\cdot)$ and $\upsilon^{*}(\cdot)$) satisfy the following equation:
\begin{equation}\label{L-optimal-control-relation}
 u_{2}^{*}(s)=\upsilon^{*}(s)-T_{22}(s,\alpha(s))^{-1}T_{21}(s,\alpha(s))Z^{*}(s),\quad s\in[0,T],\quad a.s.,
\end{equation}
where $\left(Y^{*}(\cdot),Z^{*}(\cdot),\mathbf{\Gamma}^{*}(\cdot)\right)$ is the corresponding optimal state process.

Now, we return to construct an explicit optimal control for the leader's problem. To this end, for any given $\mathbb{S}^{n}$-valued function $\Sigma:[0,T]\times \mathcal{S}\rightarrow \mathbb{S}^{n}$, we define notations: 
\begin{equation}\label{L-notations-2}
\left\{
 \begin{aligned}
& \widehat{\mathcal{T}}\left(\Sigma(s,i)\right):=I+ \Sigma(s,i)\widetilde{T}_{11}(s,i),\\
& \widehat{\mathcal{F}}\left(\Sigma(s,i)\right):=\widetilde{F}(s,i)+\Sigma(s,i)\widetilde{S}_{1}(s,i)^{\top},\\
& \widehat{\mathcal{H}}\left(\Sigma(s,i)\right):=\widehat{H}(s,i)+\Sigma(s,i)S_{2}(s,i)^{\top},
\end{aligned}
\right.
\end{equation}
and formulate the following CDREs:
\begin{equation}\label{L-RE}
\left\{
\begin{aligned}
	&\dot{\Sigma}(s,i)=\widehat{A}(s,i)\Sigma(s,i)+\Sigma(s,i)\widehat{A}(s,i)^{\top}+\Sigma(s,i)G(s,i)\Sigma(s,i)
	-\sum_{k=1}^{D}\lambda_{ik}(s)\Sigma(s,k)\\
	&-\widehat{\mathcal{F}}\left(\Sigma(s,i)\right)\widehat{\mathcal{T}}\left(\Sigma(s,i)\right)^{-1}
	\Sigma(s,i)\widehat{\mathcal{F}}\left(\Sigma(s,i)\right)^{\top}-\widehat{\mathcal{H}}\left(\Sigma(s,i)\right)T_{22}(s,i)^{-1}
	\widehat{\mathcal{H}}\left(\Sigma(s,i)\right)^{\top},\\
	&\Sigma(T,i)=0, \qquad s\in[0,T],\quad i\in\mathcal{S}.
\end{aligned}
\right.
\end{equation}
Based on the solution to the above CDREs, we further introduce the following linear BSDE:
\begin{equation}\label{L-L-BSDE}
\left\{
\begin{aligned}
	d\varphi(s)&=\big\{\big[\widehat{A}(s,\alpha(s))-\widehat{\mathcal{F}}(\Sigma(s,\alpha(s)))\widehat{\mathcal{T}}(\Sigma(s,\alpha(s)))^{-1}\Sigma(s,\alpha(s)) \widetilde{S}_{1}(s,\alpha(s))+\Sigma(s,\alpha(s)) G(s,\alpha(s))\\
	&\quad-\widehat{\mathcal{H}}(\Sigma(s,\alpha(s)))T_{22}(s,\alpha(s))^{-1}S_{2}(s,\alpha(s))\big]\varphi(s)
	+\widehat{\mathcal{F}}(\Sigma(s,\alpha(s)))\widehat{\mathcal{T}}(\Sigma(s,\alpha(s)))^{-1}[\theta(s)-\Sigma(s,\alpha(s))\widetilde{\rho}_{1}(s)]\\
	&\quad -\widehat{\mathcal{H}}(\Sigma(s,\alpha(s)))T_{22}(s,\alpha(s))^{-1}\rho_{2}(s)+\Sigma(s,\alpha(s)) q(s)+\widehat{f}(s)\big\}ds+\theta(s) dW(s)+\mathbf{\gamma}(s)\cdot d\mathbf{\widetilde{N}}(s), \, s \in[0, T], \\
	\varphi(T)&=m.
\end{aligned}
\right.
\end{equation}

The following result provides an explicit optimal control for Problem (M-ZLQ-L) by using the solution $\mathbf{\Sigma}$ to CDREs \eqref{L-RE} and the adapted solution $\left(\varphi,\theta,\gamma\right)$ to BSDE \eqref{L-L-BSDE}.

\begin{theorem}\label{L-thm-control}
Let assumptions (H1)-(H4)  hold. Suppose that the CDREs \eqref{L-RE} admits a unique solution $\mathbf{\Sigma}(\cdot)\equiv\left[\Sigma(\cdot,1),\cdots,\Sigma(\cdot,D)\right]\in \mathcal{D}\left(C\left(0,T;\mathbb{S}^{n}\right)\right)$ such that $\widehat{\mathcal{T}}\left(\Sigma(\cdot,\alpha(\cdot))\right)^{-1}\in L^{\infty}(0, T ; \mathbb{R}^{n\times n})$. Let $\left(\varphi(\cdot),\theta(\cdot),\gamma(\cdot)\right)$ be the adapted solution to the linear BSDE \eqref{L-L-BSDE} and $\phi^{*}$ the solution to the following SDE:
\begin{equation}\label{phi}
	\left\{
	\begin{aligned}
		&d\phi^{*}(s)=\Big\{\big[-\widehat{A}(s,\alpha(s))^{\top}-G(s,\alpha(s))\Sigma(s,\alpha(s))+\widetilde{S}_{1}(s,\alpha(s))^{\top}\widehat{\mathcal{T}}(\Sigma(s,\alpha(s)))^{-1}\Sigma(s,\alpha(s))\widehat{\mathcal{F}}(\Sigma(s,\alpha(s)))^{\top}\\
		&\quad+S_{2}(s,\alpha(s))^{\top}T_{22}(s,\alpha(s))^{-1}\widehat{\mathcal{H}}(\Sigma(s,\alpha(s)))^{\top}\big]\phi^{*}(s)
		-\big[\widetilde{S}_{1}(s,\alpha(s))^{\top}\widehat{\mathcal{T}}(\Sigma(s,\alpha(s)))^{-1}\Sigma(s,\alpha(s)) \widetilde{S}_{1}(s,\alpha(s))\\
		&\quad+S_{2}(s,\alpha(s))^{\top}T_{22}(s,\alpha(s))^{-1}S_{2}(s,\alpha(s))-G(s,\alpha(s))\big]\varphi(s)+\widetilde{S}_{1}(s,\alpha(s))^{\top}\widehat{\mathcal{T}}(\Sigma(s,\alpha(s)))^{-1}\theta(s)\\
		&\quad-\widetilde{S}_{1}(s,\alpha(s))^{\top}\widehat{\mathcal{T}}(\Sigma(s,\alpha(s)))^{-1}\Sigma(s,\alpha(s))\widetilde{\rho}_{1}(s)
		-S_{2}(s,\alpha(s))^{\top}T_{22}(s,\alpha(s))^{-1}\rho_{2}(s)+q(s)\Big\}ds\\
		&\quad-\big[\widehat{\mathcal{T}}(\Sigma(s,\alpha(s)))^{-1}\big]^{\top}\big[\widehat{\mathcal{F}}(\Sigma(s,\alpha(s)))^{\top}\phi^{*}(s)-\widetilde{T}_{11}(s,\alpha(s))\theta(s)-\widetilde{S}_{1}(s,\alpha(s))\varphi(s)-\widetilde{\rho}_{1}(s)\big]dW(s),\\
		&\phi^{*}(0)=-x.
	\end{aligned}
	\right.
\end{equation}
Then, the optimal control for the leader is given by
\begin{equation}\label{L-optimal-control}
\begin{aligned}
u_{2}^{*}(s)&=T_{22}(s,\alpha(s))^{-1}\big[\widehat{\mathcal{H}}\left(\Sigma(s,\alpha(s))\right)^{\top}-T_{21}(s,\alpha(s))\widehat{\mathcal{T}}\left(\Sigma(s,\alpha(s))\right)^{-1}\Sigma(s,\alpha(s))\widehat{\mathcal{F}}\left(\Sigma(s,\alpha(s))\right)^{\top}\big]\phi^{*}(s)\\
&\quad-T_{22}(s,\alpha(s))^{-1}\Big\{T_{21}(s,\alpha(s))\widehat{\mathcal{T}}\left(\Sigma(s,\alpha(s))\right)^{-1}\big[\theta(s)-\Sigma(s,\alpha(s))\widetilde{S}_{1}(s,\alpha(s))\varphi(s)-\Sigma(s,\alpha(s))\widetilde{\rho}_{1}(s)\big]\\
&\quad +S_{2}(s,\alpha(s))\varphi(s)+\rho_{2}(s)\Big\}.
\end{aligned}
\end{equation}
\end{theorem}
\begin{proof}
To obtain an explicit optimal control for the leader, we only need to construct a solution to the corresponding optimality system for the Problem (M-BLQ). Then, the desired result follows from the Theorem \ref{L-openloop-solvability} and relation \eqref{L-optimal-control-relation}. Now, let us elaborate on this process.
Set
   \begin{align}
	\label{L-P-Y}
	&Y^{*}=-\Sigma\phi^{*}+\varphi,\\
	\label{L-P-Z}
	&Z^{*}=\widehat{\mathcal{T}}\left(\Sigma\right)^{-1}\big[\Sigma\widehat{\mathcal{F}}\left(\Sigma\right)^{\top}\phi^{*}-\Sigma\widetilde{S}_{1}\varphi-\Sigma\widetilde{\rho}_{1}+\theta\big],\\
	\label{L-P-Gamma}
	&\Gamma_{k}^{*}(s)=-\big[\Sigma(s,k)-\Sigma(s,\alpha(s-))\big]\phi^{*}(s)+\gamma_{k}(s),\quad s\in[0,T],\qquad k\in\mathcal{S},\\
	\label{L-optimal-upsilon}
	&\upsilon^{*}=T_{22}^{-1}\big[\widehat{\mathcal{H}}\left(\Sigma\right)^{\top}\phi^{*}-S_{2}\varphi-\rho_{2}\big].
\end{align}
We claim that $\left(\phi^{*},Y^{*},Z^{*},\mathbf{\Gamma}^{*},\upsilon^{*}\right)$ defined in the above solves the following coupled FBSDE:
		\begin{equation}\label{L-FBSDE-2}
			\left\{
			\begin{aligned}
				&d\phi^{*}(s)=\big[-\widehat{A}^{\top}\phi^{*}+GY^{*}+\widetilde{S}_{1}^{\top}Z^{*}+S_{2}^{\top}\upsilon^{*}+q\big]ds+\big[-\widetilde{F}^{\top}\phi^{*}+\widetilde{S}_{1}Y^{*}+\widetilde{T}_{11}Z^{*}+\widetilde{\rho}_{1}(s)\big]dW(s)\\
				&dY^{*}(s)=\big[\widehat{A}Y^{*}+\widetilde{F}Z^{*}+\widehat{H}\upsilon^{*}+\widehat{f}\big]ds+Z^{*}dW(s)+\Gamma^{*}\cdot d\widetilde{N}(s),\\
				&-\widehat{H}^{\top}\phi^{*}+S_{2}Y^{*}+T_{22}\upsilon^{*}+\rho_{2}=0,\\
				&\phi^{*}(0)=-x,\quad \alpha(0)=i,\quad Y^{*}(T)=m.
			\end{aligned}
			\right.
		\end{equation}
To see this, let
\begin{align*}
 \beta_{1}&=\big[-\widehat{A}^{\top}-G\Sigma+\widetilde{S}_{1}^{\top}\widehat{\mathcal{T}}(\Sigma)^{-1}\Sigma\widehat{\mathcal{F}}(\Sigma)^{\top}+S_{2}^{\top}T_{22}^{-1}\widehat{\mathcal{H}}(\Sigma)^{\top}\big]\phi^{*}
		-\big[\widetilde{S}_{1}^{\top}\widehat{\mathcal{T}}(\Sigma)^{-1}\Sigma \widetilde{S}_{1}\\
		&\quad+S_{2}^{\top}T_{22}^{-1}S_{2}-G\big]\varphi+\widetilde{S}_{1}^{\top}\widehat{\mathcal{T}}(\Sigma)^{-1}\theta-\widetilde{S}_{1}^{\top}\widehat{\mathcal{T}}(\Sigma)^{-1}\Sigma\widetilde{\rho}_{1}
		-S_{2}^{\top}T_{22}^{-1}\rho_{2}+q ,\\
\beta_{2}&=	-\big[\widehat{\mathcal{T}}(\Sigma)^{-1}\big]^{\top}\big[\widehat{\mathcal{F}}(\Sigma)^{\top}\phi^{*}-\widetilde{T}_{11}\theta-\widetilde{S}_{1}\varphi-\widetilde{\rho}_{1}\big].
\end{align*}
Then, we have
\begin{align*}
\beta_{1}&=-\widehat{A}^{\top}\phi^{*}+G(-\Sigma\phi^{*}+\varphi)+\widetilde{S}_{1}^{\top}\widehat{\mathcal{T}}(\Sigma)^{-1}\big[\Sigma\widehat{\mathcal{F}}(\Sigma)^{\top}\phi^{*}-\Sigma\big(\widetilde{S}_{1}\varphi+\widetilde{\rho}_{1}\big)+\theta\big]+S_{2}^{\top}T_{22}^{-1}\big[\widehat{\mathcal{H}}(\Sigma)^{\top}\phi^{*}-S_{2}\varphi-\rho_{2}\big]+q\\
	&= -\widehat{A}^{\top}\phi^{*}+GY^{*}+\widetilde{S}_{1}^{\top}Z^{*}+S_{2}^{\top}\upsilon^{*}+q,
\end{align*}
and
\begin{align*}
\beta_{2}&=\big[\widetilde{T}_{11}\widehat{\mathcal{T}}(\Sigma)^{-1}\Sigma-I\big]\widehat{\mathcal{F}}(\Sigma)^{\top}\phi^{*}+\widetilde{T}_{11}\widehat{\mathcal{T}}(\Sigma)^{-1}\theta+\big[I-\widetilde{T}_{11}\widehat{\mathcal{T}}(\Sigma)^{-1}\Sigma\big]\big(\widetilde{S}_{1}\varphi+\widetilde{\rho}_{1}\big)\\
&=-\widetilde{F}^{\top}\phi^{*}+\widetilde{S}_{1}(-\Sigma\phi^{*}+\varphi)+\widetilde{T}_{11}\widehat{\mathcal{T}}(\Sigma)^{-1}\big[\Sigma\widehat{\mathcal{F}}(\Sigma)^{\top}\phi^{*}-\Sigma\big(\widetilde{S}_{1}\varphi+\widetilde{\rho}_{1}\big)+\theta\big]+\widetilde{\rho}_{1}\\
  &=-\widetilde{F}^{\top}\phi^{*}+\widetilde{S}_{1}Y^{*}+\widetilde{T}_{11}Z^{*}+\widetilde{\rho}_{1}.
\end{align*}
In the above, we use the following fact in the second equation:
$$
I-\widetilde{T}_{11}\big[\widehat{\mathcal{T}}(\Sigma)\big]^{-1}\Sigma=\big[\widehat{\mathcal{T}}(\Sigma)^{-1}\big]^{\top},\quad
\widetilde{T}_{11}\widehat{\mathcal{T}}(\Sigma)^{-1}=\big[\widehat{\mathcal{T}}(\Sigma)^{-1}\big]^{\top}\widetilde{T}_{11}.
$$
On the other hand, it follows from \eqref{L-RE} that
\begin{equation}\label{L-P-RE}
	\begin{aligned}
		d\Sigma(s,\alpha(s))&=\big\{\dot{\Sigma}(s,\alpha(s))
		+\sum_{k=1}^{D}\lambda_{\alpha(s-),k}(s)\big[\Sigma(s,k)-\Sigma(s-,\alpha(s-))\big]\big\}ds\\
		&\quad+\sum_{k=1}^{D}\big[\Sigma(s,k)-\Sigma(s-,\alpha(s-))\big]d\widetilde{N}_{k}(s)\\
		&=\big\{\widehat{A}\Sigma
		+\Sigma\widehat{A}^{\top}+\Sigma G\Sigma
		-\widehat{\mathcal{F}}(\Sigma)\widehat{\mathcal{T}}(\Sigma)^{-1}\Sigma\widehat{\mathcal{F}}(\Sigma)^{\top}
		-\widehat{\mathcal{H}}(\Sigma)T_{22}^{-1}\widehat{\mathcal{H}}(\Sigma)^{\top}\big\}ds\\
		&\quad+\sum_{k=1}^{D}\big[\Sigma(s,k)-\Sigma(s-,\alpha(s-))\big]d\widetilde{N}_{k}(s).
	\end{aligned}
\end{equation}
Let
$$
\begin{aligned}
	\beta=&\,\big[\widehat{A}
	-\widehat{\mathcal{H}}(\Sigma)T_{22}^{-1}S_{2}
	-\widehat{\mathcal{F}}(\Sigma)\widehat{\mathcal{T}}(\Sigma)^{-1}
	\Sigma \widetilde{S}_{1}+\Sigma G\big]\varphi
	+\widehat{\mathcal{F}}(\Sigma)\widehat{\mathcal{T}}(\Sigma)^{-1}\theta\\
	&-\widehat{\mathcal{F}}(\Sigma)\widehat{\mathcal{T}}(\Sigma)^{-1}
	\Sigma\widetilde{\rho}_{1}
	-\widehat{\mathcal{H}}(\Sigma)T_{22}^{-1}\rho_{2}+\Sigma q+\widehat{f}.
\end{aligned}$$
Applying the It\^o's formula to $Y^{*}=-\Sigma\phi^{*}+\varphi$, we obtain: 
\begin{align*}
	dY^{*}&=\big\{\beta-\big[\widehat{A}\Sigma+\Sigma\widehat{A}^{\top}+\Sigma G\Sigma-\widehat{\mathcal{F}}(\Sigma)\widehat{\mathcal{T}}(\Sigma)^{-1}\Sigma\widehat{\mathcal{F}}(\Sigma)^{\top}-\widehat{\mathcal{H}}(\Sigma)T_{22}^{-1}\widehat{\mathcal{H}}(\Sigma)^{\top}\big]\phi^{*}-\Sigma\beta_{1}\big\}ds\\
	&\quad+\big[\theta-\Sigma\beta_{2}\big]dW(s)+\Gamma \cdot d\widetilde{N}(s),
\end{align*}
One can further verify that
\begin{align*}
	&\quad\beta-\big[\widehat{A}\Sigma+\Sigma\widehat{A}^{\top}+\Sigma G\Sigma-\widehat{\mathcal{F}}(\Sigma)\widehat{\mathcal{T}}(\Sigma)^{-1}\Sigma\widehat{\mathcal{F}}(\Sigma)^{\top}-\widehat{\mathcal{H}}(\Sigma)T_{22}^{-1}\widehat{\mathcal{H}}(\Sigma)^{\top}\big]\phi^{*}-\Sigma\beta_{1}\\
	&=\big[\widehat{A}-\widehat{\mathcal{H}}(\Sigma)T_{22}^{-1}S_{2}
	-\widehat{\mathcal{F}}(\Sigma)\widehat{\mathcal{T}}(\Sigma)^{-1}\Sigma \widetilde{S}_{1}+\Sigma G\big]\varphi
	+\widehat{\mathcal{F}}(\Sigma)\widehat{\mathcal{T}}(\Sigma)^{-1}\theta-\widehat{\mathcal{F}}(\Sigma)\widehat{\mathcal{T}}(\Sigma)^{-1}\Sigma\widetilde{\rho}_{1}
	\\
	&\quad-\widehat{\mathcal{H}}(\Sigma)T_{22}^{-1}\rho_{2}+\widehat{f}+\Sigma q-\Sigma\big[-\widehat{A}^{\top}-G\Sigma+\widetilde{S}_{1}^{\top}\widehat{\mathcal{T}}(\Sigma)^{-1}\Sigma\widehat{\mathcal{F}}(\Sigma)^{\top}+S_{2}^{\top}T_{22}^{-1}\widehat{\mathcal{H}}(\Sigma)^{\top}\big]\phi^{*}\\
	&\quad+\Sigma\big[\widetilde{S}_{1}^{\top}\widehat{\mathcal{T}}(\Sigma)^{-1}\Sigma \widetilde{S}_{1}+S_{2}^{\top}T_{22}^{-1}S_{2}-G\big]\varphi
	-\Sigma\big[\widetilde{S}_{1}^{\top}\widehat{\mathcal{T}}(\Sigma)^{-1}\theta-\widetilde{S}_{1}^{\top}\widehat{\mathcal{T}}(\Sigma)^{-1}\Sigma\widetilde{\rho}_{1}
	-S_{2}^{\top}T_{22}^{-1}\rho_{2}+q\big]\\
	&\quad-\big[\widehat{A}\Sigma
	+\Sigma\widehat{A}^{\top}+\Sigma G\Sigma-\widehat{\mathcal{F}}(\Sigma)\widehat{\mathcal{T}}(\Sigma)^{-1}
	\Sigma\widehat{\mathcal{F}}(\Sigma)^{\top}-\widehat{\mathcal{H}}(\Sigma)T_{22}^{-1}\widehat{\mathcal{H}}(\Sigma)^{\top}\big]\phi^{*}\\
	&=\big[\widehat{A}-\widehat{H}T_{22}^{-1}S_{2}-\widetilde{F}\widehat{\mathcal{T}}(\Sigma)^{-1}\Sigma \widetilde{S}_{1}\big]\varphi
	+\big[-\widehat{A}\Sigma+\widehat{H}T_{22}^{-1}\widehat{\mathcal{H}}(\Sigma)^{\top}
	+\widetilde{F}\widehat{\mathcal{T}}(\Sigma)^{-1}\Sigma\widehat{\mathcal{F}}(\Sigma)^{\top}\big]\phi^{*}\\
	&\quad+\widetilde{F}\widehat{\mathcal{T}}(\Sigma)^{-1}\theta-\widetilde{F}\widehat{\mathcal{T}}(\Sigma)^{-1}\Sigma\widetilde{\rho}_{1}
	-\widehat{H}T_{22}^{-1}\rho_{2}+\widehat{f}\\
	&=\widehat{A}\big[-\Sigma\phi^{*}+\varphi\big]+\widetilde{F}\widehat{\mathcal{T}}(\Sigma)^{-1}\big[\Sigma\widehat{\mathcal{F}}(\Sigma)^{\top}\phi^{*}-\Sigma \widetilde{S}_{1}\varphi(s)-\Sigma\widetilde{\rho}_{1}+\theta\big]\\
	&\quad+\widehat{H}T_{22}^{-1}\big[\widehat{\mathcal{H}}(\Sigma)^{\top}\phi^{*}-S_{2}\varphi-\rho_{2}\big]+\widehat{f}\\
	&=\widehat{A}Y^{*}+\widetilde{F}Z^{*}+\widehat{H}\upsilon^{*}+\widehat{f},
\end{align*}
and
\begin{align*}
	\theta-\Sigma\beta_{2}
	&=\theta+\Sigma\left[\widehat{\mathcal{T}}(\Sigma)^{-1}\right]^{\top}\big[\widehat{\mathcal{F}}(\Sigma)^{\top}\phi^{*}-\widetilde{T}_{11}\theta-\widetilde{S}_{1}\varphi-\widetilde{\rho}_{1}\big]\\
	&=\theta+\widehat{\mathcal{T}}(\Sigma)^{-1}\Sigma\big[\widehat{\mathcal{F}}(\Sigma)^{\top}\phi^{*}-\widetilde{T}_{11}\theta-\widetilde{S}_{1}\varphi-\widetilde{\rho}_{1}\big]\\
	&=\big[I-\widehat{\mathcal{T}}(\Sigma)^{-1}\Sigma \widetilde{T}_{11}\big]\theta+\widehat{\mathcal{T}}(\Sigma)^{-1}\Sigma \big[\widehat{\mathcal{F}}(\Sigma)^{\top}\phi^{*}-\big(\widetilde{S}_{1}\varphi+\widetilde{\rho}_{1}\big)\big]\\
	&=Z^{*}.
\end{align*}
Finally, by the definition of $\upsilon^{*}$ in equation \eqref{L-optimal-upsilon}, we can immediately obtain that
\begin{equation}\label{L-P-stationary}
	T_{22}\upsilon^{*}=\widehat{\mathcal{H}}(\Sigma)^{\top}\phi^{*}-S_{2}\varphi-\rho_{2}
	=\widehat{H}^{\top}\phi^{*}+S_{2}(\Sigma\phi^{*}-\varphi)-\rho_{2}
	=\widehat{H}^{\top}\phi^{*}-S_{2}Y^{*}-\rho_{2}.
\end{equation}
To sum up, we prove our claim.  Therefore, by Theorem \ref{L-openloop-solvability}, $\upsilon^{*}$ is the optimal control for Problem (M-BLQ). Consequently, the desired result follows from the relation \eqref{L-optimal-control-relation}. This completes the proof.
\end{proof}

Based on the above result, we can obtain the following presentation for the value function of Problem (M-ZLQ-L).

\begin{theorem}\label{L-thm-value}
Let the conditions of the Theorem \ref{L-thm-control} hold. Then the value function \eqref{L-value} of Problem (M-ZLQ-L) is given by
\begin{equation}\label{L-value-function}
	\begin{aligned}
		V_{L}(m,i)
		&=\mathbb{E}\langle \Sigma(0,i)x+2\varphi(0),-x \rangle
		+\mathbb{E}\Big\{\int_{0}^{T}\Big[
		-\big<\varphi,\big(\widetilde{S}_{1}^{\top}\widehat{\mathcal{T}}(\Sigma)^{-1}\Sigma \widetilde{S}_{1}
		+S_{2}^{\top}T_{22}^{-1}S_{2}-G\big)\varphi\big>\\
		&\quad+2\big<\varphi,\widetilde{S}_{1}^{\top}\widehat{\mathcal{T}}(\Sigma)^{-1}\theta-\widetilde{S}_{1}^{\top}\widehat{\mathcal{T}}(\Sigma)^{-1}\Sigma\widetilde{\rho}_{1}
		-S_{2}^{\top}T_{22}^{-1}\rho_{2}+q\big>+\big<\theta,\widetilde{T}_{11}\widehat{\mathcal{T}}(\Sigma)^{-1}\theta\big>\\
		&\quad-\big<\widetilde{\rho}_{1},\widehat{\mathcal{T}}(\Sigma)^{-1}\Sigma\widetilde{\rho}_{1}\big>+2\big<\widetilde{\rho}_{1},\widehat{\mathcal{T}}(\Sigma)^{-1}\theta\big>-\big<\rho_{2},T_{22}^{-1}\rho_{2}\big>\Big]ds\Big\}.
	\end{aligned}
\end{equation}
\end{theorem}

\begin{proof}
Let $u_{2}^{*}$ be the optimal control for the leader. Then, we have
\begin{align*}
V_{L}(m,i)&=J_{L}(m,i;u_{2}^{*}(\cdot))\\
&=\mathbb{E}\left\{2\langle Y^{*}(0),-x \rangle+
\int_{0}^{T}
\left[
\left<
\left(
\begin{matrix}
G &S_{1}^{\top} & S_{2}^{\top} \\
S_{1} & T_{11} & T_{12} \\
S_{2} & T_{21} & T_{22}
\end{matrix}
\right)
\left(
\begin{matrix}
Y^{*} \\
Z^{*} \\
u_{2}^{*}
\end{matrix}
\right),
\left(\begin{matrix}
Y^{*} \\
Z^{*} \\
u_{2}^{*}
\end{matrix}
\right)
\right>+2\left<
\left(
\begin{matrix}
q\\
\rho_{1}\\
\rho_{2}
\end{matrix}
\right),
\left(
\begin{matrix}
Y^{*}\\
Z^{*}\\
u_{2}^{*}
\end{matrix}
\right)
\right>
\right]ds\right\}\\
&=\mathbb{E}\left\{2\langle Y^{*}(0),-x \rangle+
\int_{0}^{T}
\left[
\left<
\left(
\begin{matrix}
G &\widetilde{S}_{1}^{\top} & S_{2}^{\top} \\
\widetilde{S}_{1} & \widetilde{T}_{11} &  0 \\
S_{2} & 0 & T_{22}
\end{matrix}
\right)
\left(
\begin{matrix}
Y^{*} \\
Z^{*} \\
\upsilon^{*}
\end{matrix}
\right),
\left(\begin{matrix}
Y^{*} \\
Z^{*} \\
\upsilon^{*}
\end{matrix}
\right)
\right>+2\left<
\left(
\begin{matrix}
q \\
\widetilde{\rho}_{1} \\
\rho_{2}
\end{matrix}
\right),
\left(
\begin{matrix}
Y^{*} \\
Z^{*} \\
\upsilon^{*}
\end{matrix}
\right)
\right>
\right] ds\right\}\\
&=\mathbb{E}\Big\{
	\int_{0}^{T}\Big[\big<GY^{*}+\widetilde{S}_{1}^{\top}Z^{*}+S_{2}^{\top}\upsilon^{*}+2q,Y^{*}\big>
	+\big<\widetilde{S}_{1}Y^{*}+\widetilde{T}_{11}Z^{*}+2\widetilde{\rho}_{1},Z^{*}\big>+\big<\widehat{H}^{\top}\phi^{*}+\rho_{2},\upsilon^{*}\big>\Big]ds\\
	&\quad+2\langle Y^{*}(0),-x \rangle\Big\},
\end{align*}
where $\left(\phi^{*}, Y^{*}, Z^{*}, \mathbf{\Gamma}^{*},\upsilon^{*}\right)$ are defined in Theorem \ref{L-thm-control} and satisfies equation \eqref{L-FBSDE-2}.
 Therefore, applying the integration by parts formula to $\left<Y^{*},\phi^{*}\right>$ and taking expectation yield that
\begin{align*}
	\mathbb{E}\left<Y^{*}(T),\phi^{*}(T)\right>&=\mathbb{E}\langle Y^{*}(0),-x \rangle
	+\mathbb{E}\Big\{\int_{0}^{T}\Big[
	\big<Y^{*},-\widehat{A}^{\top}\phi^{*}+GY^{*}+\widetilde{S}_{1}^{\top}Z^{*}+S_{2}^{\top}\upsilon^{*}+q\big>\\
	&\quad+\big<\widehat{A}Y^{*}+\widetilde{F}Z^{*}+\widehat{H}\upsilon^{*}+\widehat{f},\phi^{*}\big>+\big<Z^{*},-\widetilde{F}^{\top}\phi^{*}+\widetilde{S}_{1}Y^{*}+\widetilde{T}_{11}Z^{*}+\widetilde{\rho}_{1}\big>\Big]ds\Big\}\\
	&=V_{L}(m,i)-\mathbb{E}\langle Y^{*}(0),-x \rangle-\mathbb{E}\Big\{\int_{0}^{T}\Big[\big<q,Y^{*}\big>+\big<\widetilde{\rho}_{1},Z^{*}\big>+\big<\rho_{2},\upsilon^{*}\big>-\big<\widehat{f},\phi^{*}\big>\Big]ds\Big\}.
\end{align*}
Consequently, we have
\begin{equation}\label{L-VALUE-PROOF}
   V_{L}(m,i)=\mathbb{E}\left<Y^{*}(T),\phi^{*}(T)\right>
+\mathbb{E}\langle Y^{*}(0),-x \rangle
+\mathbb{E}\int_{0}^{T}\Big[\big<q,Y^{*}\big>+\big<\widetilde{\rho}_{1},Z^{*}\big>+\big<\rho_{2},\upsilon^{*}\big>-\big<\widehat{f},\phi^{*}\big>\Big]ds.
\end{equation}
On the other hand,
\begin{align*}
	&\quad\mathbb{E}\left<Y^{*}(T),\phi^{*}(T)\right>
	=\mathbb{E}\left<m,\phi^{*}(T)\right>
	=\mathbb{E}\left<\varphi(T),\phi^{*}(T)\right>\\
	&=\mathbb{E}\left<\varphi(0),-x\right>
	+\mathbb{E}\Big\{\int_{0}^{T}\Big[\big<\big[\widehat{A}-\widehat{\mathcal{H}}(\Sigma)T_{22}^{-1}S_{2}-\widehat{\mathcal{F}}(\Sigma)\widehat{\mathcal{T}}(\Sigma)^{-1}\Sigma \widetilde{S}_{1}+\Sigma G\big]\varphi+\widehat{\mathcal{F}}(\Sigma)\widehat{\mathcal{T}}(\Sigma)^{-1}\theta\\
	&\quad-\widehat{\mathcal{F}}(\Sigma)\widehat{\mathcal{T}}(\Sigma)^{-1}
	\Sigma\widetilde{\rho}_{1}-\widehat{\mathcal{H}}(\Sigma)T_{22}^{-1}\rho_{2}+\Sigma q+\widehat{f},\phi^{*}\big>
	+\big<\varphi,\big[-\widehat{A}^{\top}+\widetilde{S}_{1}^{\top}\widehat{\mathcal{T}}(\Sigma)^{-1}\Sigma\widehat{\mathcal{F}}(\Sigma)^{\top}\\
	&\quad+S_{2}^{\top}T_{22}^{-1}\widehat{\mathcal{H}}(\Sigma)^{\top}-G\Sigma\big]\phi^{*}
	-\big[\widetilde{S}_{1}^{\top}\widehat{\mathcal{T}}(\Sigma)^{-1}\Sigma \widetilde{S}_{1}+S_{2}^{\top}T_{22}^{-1}S_{2}-G\big]\varphi
	+\widetilde{S}_{1}^{\top}\widehat{\mathcal{T}}(\Sigma)^{-1}\theta\\
	&\quad-\widetilde{S}_{1}^{\top}\widehat{\mathcal{T}}(\Sigma)^{-1}\Sigma\widetilde{\rho}_{1}-S_{2}^{\top}T_{22}^{-1}\rho_{2}+q\big>-\big<\theta,\left[\widehat{\mathcal{T}}(\Sigma)^{-1}\right]^{\top}\big[\widehat{\mathcal{F}}(\Sigma)^{\top}\phi^{*}-\widetilde{T}_{11}\theta-\widetilde{S}_{1}\varphi-\widetilde{\rho}_{1}\big]\big>\Big]ds\Big\}\\
	&=\mathbb{E}\left<\varphi(0),-x\right>
	+\mathbb{E}\Big\{\int_{0}^{T}\Big[
	\big<\widehat{\mathcal{F}}(\Sigma)\widehat{\mathcal{T}}(\Sigma)^{-1}\theta-\widehat{\mathcal{F}}(\Sigma)\widehat{\mathcal{T}}(\Sigma)^{-1}
	\Sigma\widetilde{\rho}_{1}-\widehat{\mathcal{H}}(\Sigma)T_{22}^{-1}\rho_{2}+\Sigma q+\widehat{f},\phi^{*}\big>\\
	&\quad-\big<\varphi,\big(\widetilde{S}_{1}^{\top}\widehat{\mathcal{T}}(\Sigma)^{-1}\Sigma \widetilde{S}_{1}+S_{2}^{\top}T_{22}^{-1}S_{2}-G\big)\varphi\big>-\big<\widehat{\mathcal{T}}(\Sigma)^{-1}\theta,\widehat{\mathcal{F}}(\Sigma)^{\top}\phi^{*}-\widetilde{T}_{11}\theta-\widetilde{S}_{1}\varphi-\widetilde{\rho}_{1}\big>\\
	&\quad+\big<\varphi,\widetilde{S}_{1}^{\top}\widehat{\mathcal{T}}(\Sigma)^{-1}\theta
	-\widetilde{S}_{1}^{\top}\widehat{\mathcal{T}}(\Sigma)^{-1}\Sigma\widetilde{\rho}_{1}
	-S_{2}^{\top}T_{22}^{-1}\rho_{2}+q\big>\Big]ds\Big\}.
\end{align*}
Substituting the above equation into \eqref{L-VALUE-PROOF} and recalling \eqref{L-optimal-control}-\eqref{L-P-Z} yields
\begin{align*}
	V_{L}(m,i)&=\mathbb{E}\left<Y^{*}(T),\phi^{*}(T)\right>
	+\mathbb{E}\langle Y^{*}(0),-x \rangle
	+\mathbb{E}\Big\{\int_{0}^{T}\Big[\big<q,Y^{*}\big>+\big<\widetilde{\rho}_{1},Z^{*}\big>+\big<\rho_{2},\upsilon^{*}\big>-\big<\widehat{f},\phi^{*}\big>\Big]ds\Big\}\\
	&=\mathbb{E}\langle \Sigma(0,i)x+2\varphi(0),-x \rangle
	+\mathbb{E}\Big\{\int_{0}^{T}\Big[
	\big<\widehat{\mathcal{F}}(\Sigma)\widehat{\mathcal{T}}(\Sigma)^{-1}\theta-\widehat{\mathcal{F}}(\Sigma)\widehat{\mathcal{T}}(\Sigma)^{-1}\Sigma\widetilde{\rho}_{1}-\widehat{\mathcal{H}}(\Sigma)T_{22}^{-1}\rho_{2}\\
	&\quad+\Sigma q+\widehat{f},\phi^{*}\big>
	-\big<\varphi,\big(\widetilde{S}_{1}^{\top}\widehat{\mathcal{T}}(\Sigma)^{-1}\Sigma \widetilde{S}_{1}+S_{2}^{\top}T_{22}^{-1}S_{2}-G\big)\varphi\big>-\big<\widehat{\mathcal{T}}(\Sigma)^{-1}\theta,\widehat{\mathcal{F}}(\Sigma)^{\top}\phi^{*}-\widetilde{T}_{11}\theta\\
	&\quad-\widetilde{S}_{1}\varphi-\widetilde{\rho}_{1}\big>+\big<\varphi,\widetilde{S}_{1}^{\top}\widehat{\mathcal{T}}(\Sigma)^{-1}\theta
	-\widetilde{S}_{1}^{\top}\widehat{\mathcal{T}}(\Sigma)^{-1}\Sigma\widetilde{\rho}_{1}
	-S_{2}^{\top}T_{22}^{-1}\rho_{2}+q\big>
	+\big<q,-\Sigma\phi^{*}+\varphi\big>\\
	&\quad+\big<\widetilde{\rho}_{1},\widehat{\mathcal{T}}(\Sigma)^{-1}\big[\Sigma\widehat{\mathcal{F}}(\Sigma)^{\top}\phi^{*}-\Sigma \widetilde{S}_{1}\varphi-\Sigma\widetilde{\rho}_{1}+\theta\big]\big>\\
	&\quad+\big<\rho_{2},T_{22}^{-1}\big[\widehat{\mathcal{H}}(\Sigma)^{\top}\phi^{*}-S_{2}\varphi-\rho_{2}\big]\big>-\big<\widehat{f},\phi^{*}\big>\Big]ds\Big\}\\
	&=\mathbb{E}\langle \Sigma(0,i)x+2\varphi(0),-x \rangle
	+\mathbb{E}\Big\{\int_{0}^{T}\Big[
	-\big<\varphi,\big(\widetilde{S}_{1}^{\top}\widehat{\mathcal{T}}(\Sigma)^{-1}\Sigma \widetilde{S}_{1}+S_{2}^{\top}T_{22}^{-1}S_{2}-G\big)\varphi\big>\\
	&\quad+\big<\widehat{\mathcal{T}}(\Sigma)^{-1}\theta,\widetilde{T}_{11}\theta+\widetilde{S}_{1}\varphi+\widetilde{\rho}_{1}\big>+\big<\varphi,\widetilde{S}_{1}^{\top}\widehat{\mathcal{T}}(\Sigma)^{-1}\theta-\widetilde{S}_{1}^{\top}\widehat{\mathcal{T}}(\Sigma)^{-1}\Sigma\widetilde{\rho}_{1}
	-S_{2}^{\top}T_{22}^{-1}\rho_{2}+q\big>\\
	&\quad+\big<q,\varphi\big>
	-\big<\widetilde{\rho}_{1},\widehat{\mathcal{T}}(\Sigma)^{-1}\Sigma \widetilde{S}_{1}\varphi+\widehat{\mathcal{T}}(\Sigma)^{-1}\Sigma\widetilde{\rho}_{1}-\widehat{\mathcal{T}}(\Sigma)^{-1}\theta\big>\\
	&\quad-\big<\rho_{2},T_{22}^{-1}S_{2}\varphi+T_{22}^{-1}\rho_{2}\big>\Big]ds\Big\}\\
	&=\mathbb{E}\langle \Sigma(0,i)x+2\varphi(0),-x \rangle
	+\mathbb{E}\Big\{\int_{0}^{T}\Big[
	-\big<\varphi,\big(\widetilde{S}_{1}^{\top}\widehat{\mathcal{T}}(\Sigma)^{-1}\Sigma \widetilde{S}_{1}+S_{2}^{\top}T_{22}^{-1}S_{2}-G\big)\varphi\big>\\
	&\quad+2\big<\varphi,\widetilde{S}_{1}^{\top}\widehat{\mathcal{T}}(\Sigma)^{-1}\theta-\widetilde{S}_{1}^{\top}\widehat{\mathcal{T}}(\Sigma)^{-1}\Sigma\widetilde{\rho}_{1}
	-S_{2}^{\top}T_{22}^{-1}\rho_{2}+q\big>+\big<\theta,\widetilde{T}_{11}\widehat{\mathcal{T}}(\Sigma)^{-1}\theta\big>\\
	&\quad-\big<\widetilde{\rho}_{1},\widehat{\mathcal{T}}(\Sigma)^{-1}\Sigma\widetilde{\rho}_{1}\big>+2\big<\widetilde{\rho}_{1},\widehat{\mathcal{T}}(\Sigma)^{-1}\theta\big>-\big<\rho_{2},T_{22}^{-1}\rho_{2}\big>\Big]ds\Big\}.
\end{align*}
This completes the proof.
\end{proof}
To sum the Theorem \ref{thm-M-SG-F}, Theorem \ref{L-thm-control} and Theorem \ref{L-thm-value}, we obtain the Stackelberg equilibrium of Problem (M-ZLQ) as follows.
\begin{theorem}\label{thm-main}
Suppose the conditions of the Theorem \ref{L-thm-control} hold. Let $\mathbf{P}(\cdot)$, $\mathbf{\Sigma}(\cdot)$, $\left(\varphi(\cdot),\theta(\cdot),\gamma(\cdot)\right)$ and $\phi^{*}(\cdot)$ be the solution to CDREs \eqref{F-RE}, CDREs \eqref{L-RE}, BSDE \eqref{L-L-BSDE} and SDE \eqref{phi}, respectively.
Then, the Problem (M-ZLQ) admits a unique Stackelberg equilibrium:
\begin{equation}\label{Stackelberg-equilibrium}
  \left\{
  \begin{aligned}
  u_{1}^{*}(s)&=-\widehat{R}_{1}(s, \alpha(s))^{-1}\big[\widehat{S}_{1}(s,\alpha(s))X^{*}(s)+\Xi (s,\alpha(s))u_{2}^{*}(s)+B_{1}(s,\alpha(s))^{\top}Y^{*}(s)+D_{1}(s,\alpha(s))^{\top}Z^{*}(s)\\
		&\quad+D_{1}(s,\alpha(s))^{\top}P(s,\alpha(s))\sigma\big],\\
  u_{2}^{*}(s)&=T_{22}(s,\alpha(s))^{-1}\big[\widehat{\mathcal{H}}\left(\Sigma(s,\alpha(s))\right)^{\top}-T_{21}(s,\alpha(s))\widehat{\mathcal{T}}\left(\Sigma(s,\alpha(s))\right)^{-1}\Sigma(s,\alpha(s))\widehat{\mathcal{F}}\left(\Sigma(s,\alpha(s))\right)^{\top}\big]\phi^{*}(s)\\
&\quad-T_{22}(s,\alpha(s))^{-1}\Big\{T_{21}(s,\alpha(s))\widehat{\mathcal{T}}\left(\Sigma(s,\alpha(s))\right)^{-1}\big[\theta(s)-\Sigma(s,\alpha(s))\widetilde{S}_{1}(s,\alpha(s))\varphi(s)-\Sigma(s,\alpha(s))\widetilde{\rho}_{1}(s)\big]\\
&\quad +S_{2}(s,\alpha(s))\varphi(s)+\rho_{2}(s)\Big\},
  \end{aligned}
  \right.
\end{equation}
where $X^{*}(\cdot)$ is the solution to \eqref{F-optimal-state} by replacing $u_{2}(\cdot)$ with $u_{2}^{*}(\cdot)$ and $\left(Y^{*}(\cdot),Z^{*}(\cdot)\right)$ are defined in \eqref{L-P-Y}-\eqref{L-P-Z}.
In addition, the equilibrium value function is given by
\begin{equation}\label{equilibrium-value-function}
	\begin{aligned}
		V(x,i)&=-V_{L}(m,i)+\mathbb{E}\bigg\{\int_{0}^{T}\big[-\big<\widehat{R}_{1}(s,\alpha(s))^{-1}D_{1}(s,\alpha(s))^{\top}P(s,\alpha(s))\sigma(s),
		D_{1}(s,\alpha(s))^{\top}P(s,\alpha(s))\sigma(s)\big> \\
		&\quad+\big<P(s,\alpha(s))\sigma(s),\sigma(s)\big>\big]ds+\big<P(0,i)x,x\big>\bigg\}.
	\end{aligned}
\end{equation}
\end{theorem}

\section{The discussion on the solvability of Riccati equation} \label{section-RE-solvability}
\citet{Sun.J.R.2021_IBLQ} has investigated an indefinite BSLQ problem under the diffusion model. They constructed the associated Riccati-type equation solution using the forward formulation and limiting procedure developed in \citet{Lim.A.E.B.2001_BLQ}. It is natural to ask:  Do we establish such solvability for CDREs \eqref{L-RE} by using a similar method? To see this, we consider the following CDREs:
\begin{equation}\label{FSLQ-REs}
\left\{
\begin{aligned}
	&\dot{\mathcal{P}}_{\lambda}\left( s,i \right)= -\mathcal{P}_{\lambda}\left( s,i \right) \widehat{A}\left( s,i \right) -\widehat{A}\left( s,i \right) ^{\top}\mathcal{P}_{\lambda}\left( s,i \right)-G(s,i)
	+\left( \begin{array}{c}
		\widetilde{F}\left( s,i \right) ^{\top}\mathcal{P}_{\lambda}\left( s,i \right) +\widetilde{S}_1\left( s,i \right)\\
		\widehat{H}\left( s,i \right) ^{\top}\mathcal{P}_{\lambda}\left( s,i \right) +S_2\left( s,i \right)\\
	\end{array} \right) ^{\top}\\
	&\qquad\times\left( \begin{matrix}
		\widetilde{T}_{11}\left( s,i \right) +\mathcal{P}_{\lambda}\left( s,i \right)&		0\\
		0&		T_{22}\left( s,i \right)\\
	\end{matrix} \right)^{-1} \left( \begin{array}{c}
		\widetilde{F}\left( s,i \right) ^{\top}\mathcal{P}_{\lambda}\left( s,i \right) +\widetilde{S}_1\left( s,i \right)\\
		\widehat{H}\left( s,i \right) ^{\top}\mathcal{P}_{\lambda}\left( s,i \right) +S_2\left( s,i \right)\\
	\end{array} \right) \\
	&\qquad+\sum_{k=1}^D{\lambda _{ik}(s)\mathcal{P}_{\lambda}\left( s,i \right)
		\mathcal{P}_{\lambda}\left( s,k \right)^{-1}\mathcal{P}_{\lambda}\left( s,i \right) }\\
	&\mathcal{P}_{\lambda}(T,i)=\lambda I, \quad i \in \mathcal{S},
	\quad s \in[0, T].
\end{aligned}
\right.
\end{equation}

Then, we have the following result.

\begin{theorem}\label{thm-F-L-REs}
Let (H1)-(H4) hold. Suppose that, for every $\lambda\geq\lambda_{0}>0$, 
the Riccati equations \eqref{FSLQ-REs} admits a solution $\left[\mathcal{P}_{\lambda}(\cdot,1),\cdots,\mathcal{P}_{\lambda}(\cdot,D)\right]\in \mathcal{D}\left(C(0,T;\mathbb{S}_{+}^{n})\right)$ such that, for every $i\in\mathcal{S}$,
\begin{equation}\label{P-condition}
   	\mathcal{P}_{\lambda}\left( \cdot,i \right)+\widetilde{T}_{11}\left( \cdot,i \right) \gg 0 ,\quad
	\text{and}\quad \mathcal{P}_{\lambda_2}\left( \cdot,i \right)\geq\mathcal{P}_{\lambda_1}\left( \cdot,i \right),
	\qquad \forall \lambda_2>\lambda_1\geq\lambda_0.
\end{equation}
Then CDREs \eqref{L-RE} admits a solution $\mathbf{\Sigma}(\cdot)\equiv \left[\Sigma(\cdot,1),\cdots,\Sigma(\cdot,D)\right]\in\mathcal{D}\left(C(0,T;\overline{\mathbb{S}_{+}^{n}})\right)$ such that 
$$\widehat{\mathcal{T}}\left(\Sigma(\cdot,i)\right)^{-1}\in L^{\infty}(0, T ; \mathbb{R}^{n\times n}),\quad \forall i\in\mathcal{S}$$
\end{theorem}
\begin{proof}
For every $i\in\mathcal{S}$, let $\Sigma_{\lambda}(s,i):=
\mathcal{P}_{\lambda}(s,i)^{-1}, s\in[0,T]$.
Then $\Sigma_{\lambda}(s,i)$ decreases in $\lambda$ and is bounded below by zero. Therefore, $\left\{\Sigma_{\lambda}(s,i)\right\}_{\lambda\geq\lambda_0}$ is bounded uniformly in $s\in[0,T]$ and there is a $\overline{\mathbb{S}_{+}^n}$-valued function $\Sigma(\cdot,i)$ such that
$$\Sigma(s,i)=\lim_{\lambda\uparrow\infty}\Sigma_{\lambda}(s,i),\qquad s\in[0,T].$$
Next, we shall prove the following statements hold. 
\begin{enumerate}
	\item  $\widehat{\mathcal{T}}\left(\Sigma(s,i)\right)=I+\Sigma(s,i)\widetilde{T}_{11}(s,i)$ is  invertible a.e. on $[0,T]$;
	\item $\widehat{\mathcal{T}}\left(\Sigma(\cdot,i)\right)^{-1}\in L^{\infty}(0,T;\mathbb{R}^{n\times n})$;
	\item $\mathbf{\Sigma}(\cdot)$ solves the CDREs \eqref{L-RE}.
\end{enumerate}
From the definition of $\Sigma_{\lambda}$, we first observe that
$$I+\Sigma_{\lambda}\widetilde{T}_{11}=\mathcal{P}_{\lambda}^{-1}\big[\mathcal{P}_{\lambda}+\widetilde{T}_{11}\big].$$
Let $\mathcal{K}=\mathcal{P}_{\lambda_0}+\widetilde{T}_{11}$ and $\Delta_{\lambda}=\mathcal{P}_{\lambda}-\mathcal{P}_{\lambda_0}$. Then we have $\mathcal{K}\gg 0$, $\Delta_{\lambda}\geq 0$ and for every $\lambda>\lambda_0$,
$$0<\left(\mathcal{K}+\Delta_{\lambda}\right)^{-1}
=\left(\widetilde{T}_{11} +\mathcal{P}_{\lambda}\right)^{-1}
\leq \left(\widetilde{T}_{11} +\mathcal{P}_{\lambda_0}\right)^{-1}.$$
Hence, for every $x\in \mathbb{R}^{n}$, we have
\begin{align*}
	&\quad\big< \mathcal{P}_{\lambda}\big( \widetilde{T}_{11}+\mathcal{P}_{\lambda} \big) ^{-2}\mathcal{P}_{\lambda}x, x \big>
	\\
	&=\big| \left( \mathcal{K}+\Delta _{\lambda} \right) ^{-1}\left( \Delta _{\lambda}+\mathcal{P}_{\lambda _0} \right)  x \big|^2
	\\
	&\le 2\big| \left( \mathcal{K}+\Delta _{\lambda} \right)  ^{-1}\big(\mathcal{K}+\Delta _{\lambda}-\mathcal{K}\big)x \big|^2+2\big| \left( \mathcal{K}+\Delta _{\lambda} \right)  ^{-1}\mathcal{P}_{\lambda _0} x \big|^2
	\\
	&\le 2\big| x-\left( \mathcal{K}+\Delta _{\lambda} \right) ^{-1}\mathcal{K} x \big|^2+2\big| \big( \mathcal{K}+\Delta _{\lambda} \big)  ^{-1}\mathcal{P}_{\lambda _0}x \big|^2
	\\
	&\le 4\big[ 1+\big| \left( \mathcal{K}+\Delta _{\lambda} \right) ^{-1} \big|^2\big(
		\big| \mathcal{K} \big|^2+\big| \mathcal{P}_{\lambda _0} \big|^2
	 \big) \big] |x|^2
	\\
	&\le 4\big[ 1+\big| \big( \widetilde{T}_{11}+\mathcal{P}_{\lambda _0} \big) ^{-1} \big|^2\big(
		\big| \mathcal{K} \big|^2+\big| \mathcal{P}_{\lambda _0} \big|^2
	 \big) \big] \left| x \right|^2,
\end{align*}
which implies
\begin{align*}
	\big(I+\Sigma_{\lambda}\widetilde{T}_{11}\big)\big(I+\Sigma_{\lambda}\widetilde{T}_{11}\big)^{\top}
	&=\big[\mathcal{P}_{\lambda}\big( \widetilde{T}_{11}+\mathcal{P}_{\lambda} \big) ^{-2}\mathcal{P}_{\lambda}\big]^{-1}\\
	&\geq \frac{1}{4}\big[ 1+\big| \big( \widetilde{T}_{11}+\mathcal{P}_{\lambda _0} \big) ^{-1} \big|^2\big(
		\big| \mathcal{K} \big|^2+\big| \mathcal{P}_{\lambda _0} \big|^2\big) \big]^{-1}I
\end{align*}
Letting $\lambda\rightarrow\infty$ yields
$$\big(I+\Sigma \widetilde{T}_{11}\big)\big(I+\Sigma \widetilde{T}_{11}\big)^{\top}\gg0,$$
This implies $\widehat{\mathcal{T}}\left(\Sigma\right)$ is  invertible and $\widehat{\mathcal{T}}\left(\Sigma\right)^{-1}\in L^{\infty}(0,T;\mathbb{R}^{n\times n})$.

On the other hand, since $\Sigma_{\lambda}(s,i)\mathcal{P}_{\lambda}(s,i)=I$, we can easily obtain that 
$$\dot{\Sigma}_{\lambda}(s,i)\mathcal{P}_{\lambda}(s,i)
+\Sigma_{\lambda}(s,i)\dot{\mathcal{P}}_{\lambda}(s,i)=0.$$
Therefore,
\begin{align*}
	\dot{\Sigma}_{\lambda}(s,i)&=-\Sigma_{\lambda}(s,i)\dot{\mathcal{P}}_{\lambda}(s,i)
	\Sigma_{\lambda}(s,i)\\
	&=\widehat{A}(s,i)\Sigma_{\lambda}(s,i)+\Sigma_{\lambda}(s,i)\widehat{A}(s,i)^{\top}
	+\Sigma_{\lambda}(s,i)G(s,i)\Sigma_{\lambda}(s,i)\\
	&\quad-\widehat{\mathcal{F}}\left(\Sigma_{\lambda}(s,i)\right)\widehat{\mathcal{T}}\left(\Sigma_{\lambda}(s,i)\right)^{-1}
	\Sigma_{\lambda}(s,i)\widehat{\mathcal{F}}\left(\Sigma_{\lambda}(s,i)\right)^{\top}\\
	&\quad-\widehat{\mathcal{H}}\left(\Sigma_{\lambda}(s,i)\right)T_{22}(s,i)^{-1}
	\widehat{\mathcal{H}}\left(\Sigma_{\lambda}(s,i)\right)^{\top}
	-\sum_{k=1}^{D}\lambda_{i,k}(s)\Sigma_{\lambda}(s,k).
\end{align*}
Consequently, integrating both side of the above equation leads to
$$
\begin{aligned}
	\Sigma_{\lambda}(s,i)&=\lambda^{-1}I-\int_{s}^{T}\bigg[
	\widehat{A}(t,i)\Sigma_{\lambda}(t,i)+\Sigma_{\lambda}(t,i)\widehat{A}(t,i)^{\top}
	+\Sigma_{\lambda}(t,i)G(t,i)\Sigma_{\lambda}(t,i)\\
	&\quad-\widehat{\mathcal{F}}\left(\Sigma_{\lambda}(t,i)\right)\widehat{\mathcal{T}}\left(\Sigma_{\lambda}(t,i)\right)^{-1}
	\Sigma_{\lambda}(t,i)\widehat{\mathcal{F}}\left(\Sigma_{\lambda}(t,i)\right)^{\top}
	\\
	&\quad-\widehat{\mathcal{H}}\left(\Sigma_{\lambda}(t,i)\right)T_{22}(t,i)^{-1}
	\widehat{\mathcal{H}}\left(\Sigma_{\lambda}(t,i)\right)^{\top}
	-\sum_{k=1}^{D}\lambda_{i,k}(t)\Sigma_{\lambda}(t,k)\bigg]dt.
\end{aligned}
$$
Letting $\lambda\rightarrow\infty$ in the above equation, we obtain by the bounded convergence theorem that
$$
\begin{aligned}
	\Sigma(s,i)&=-\int_{s}^{T}\bigg[
	\widehat{A}(t,i)\Sigma(t,i)+\Sigma(t,i)\widehat{A}(t,i)^{\top}
	+\Sigma(t,i)G(t,i)\Sigma(t,i)-\sum_{k=1}^{D}\lambda_{ik}(t)\Sigma(t,k)\\
	&\quad-\widehat{\mathcal{F}}\left(\Sigma(t,i)\right)\widehat{\mathcal{T}}\left(\Sigma(t,i)\right)^{-1}
	\Sigma(t,i)\widehat{\mathcal{F}}\left(\Sigma(t,i)\right)^{\top}
	-\widehat{\mathcal{H}}\left(\Sigma(t,i)\right)T_{22}(t,i)^{-1}
	\widehat{\mathcal{H}}\left(\Sigma(t,i)\right)^{\top}
	\bigg]dt,\quad i\in\mathcal{S},
\end{aligned}
$$
which is the integral version of CDREs \eqref{L-RE}. This completes the proof.
\end{proof}

From Theorem \ref{thm-F-L-REs}, we know that the solution to CDREs \eqref{L-RE} can be obtained by taking the limit of the solution to \eqref{FSLQ-REs}. Regrettably, the solvability of CDREs \eqref{FSLQ-REs} can not like \cite{Sun.J.R.2021_IBLQ} be derived from a family of M-FSLQ problems: $\left\{\text{Problem (M-LQ)}_{\lambda} : \lambda\geq\lambda_{0}\right\}$ defined in \eqref{M-FSLQ-cost}-\eqref{M-FSLQ-state}. The main reason for this result is due to the coupling term $\sum_{k=1}^{D}\lambda_{ik}(s)\Sigma(s,k)$ in CDREs \eqref{L-RE}.
In fact, one can further verify that the CDREs \eqref{FSLQ-REs} can be derived from an M-FSLQ problem with the state process
\begin{equation}\label{problem-G-M-FSLQ-state}
	\left\{
	\begin{aligned}
		d\mathcal{X}(s)&=\left[\widehat{A}(s,\alpha(s))\mathcal{X}(s)+\widetilde{F}(s,\alpha(s))u_{1}(s)
		+\widehat{H}(s,\alpha(s))u_{2}(s)\right]ds\\
		&\quad+u_{1}(s)dW(s)+\sum_{k=1}^{D}\lambda_{\alpha(s-)k}(s)^{-\frac{1}{2}}v_{k}(s)d\widetilde{N}(s),\quad s\geq 0,\\
		\mathcal{X}(0)&=x,\qquad \alpha(0)=i,
	\end{aligned}
	\right.
\end{equation}
and cost functional
\begin{equation}\label{problem-G-M-FSLQ}
	\begin{aligned}
		\mathcal{J}_{\lambda}(x,i;\mathbf{u}(\cdot))\triangleq
		\mathbb{E}\left\{
		\int_{0}^{T}
		\left<
		\left(
		\begin{matrix}
			G(s,\alpha(s))&\widetilde{S}_{1}(s,\alpha(s))^{\top}&S_{2}(s,\alpha(s))^{\top}\\
			\widetilde{S}_{1}(s,\alpha(s))&\widetilde{T}_{11}(s,\alpha(s))& 0\\
			S_{2}(s,\alpha(s))&0&T_{22}(s,\alpha(s))
		\end{matrix}
		\right)
		\left(
		\begin{matrix}
			\mathcal{X}(s)\\
			u_{1}(s)\\
			u_{2}(s)
		\end{matrix}
		\right),
		\left(\begin{matrix}
			\mathcal{X}(s)\\
			u_{1}(s)\\
			u_{2}(s)
		\end{matrix}
		\right)
		\right>
		ds+\lambda\left|\mathcal{X}(T)\right|^{2}
		\right\},
	\end{aligned}
\end{equation}
where
$$\mathbf{u}(\cdot)\triangleq\left(u_{1}(\cdot),u_{2}(\cdot),v_{1}(\cdot),\cdots v_{D}(\cdot)\right)\in
L_{\mathcal{P}}^{2}\left(0, T ; \mathbb{R}^n\right)\times
L_{\mathcal{P}}^{2}\left(0, T ; \mathbb{R}^{m_{2}}\right)\times
\mathcal{D}\left(L_{\mathcal{P}}^{2}\left(0, T ; \mathbb{R}^n\right)\right),$$
represents the control of the problem.
This is a more general M-FSLQ problem compared with that studied in \citet{Zhang.X.2021_ILQM} in the sense that the control enters the terms of Markovian jumps, and the solvability to corresponding CDREs remains open.
We conjecture that the CDREs \eqref{FSLQ-REs} admits a unique solution $\mathbf{P_{\lambda}}(\cdot)\equiv\left[\mathcal{P}_{\lambda}(\cdot,1),\cdots,\mathcal{P}_{\lambda}(\cdot,D)\right]\in \mathcal{D}\left(C(0,T;\mathbb{S}_{+}^{n})\right)$ such that condition \eqref{P-condition} under the assumptions (H1)-(H4). However, we have not overcome some technical difficulties in proving such a result, and we hope to come back in our future publications.

As we all know, the solvability of CDREs for classical FSLQ problems has been completely solved. In \citet{Zhang.X.2021_ILQM}, it has shown that the strongly regular solution to CDREs for an FSLQ problem is equivalent to the uniform convexity of the corresponding cost functional. Now, let us observe closer at CDREs \eqref{L-RE}. We can find that the main difference between \eqref{L-RE} and a CDREs derived from the FSLQ problem is the term $\widehat{\mathcal{F}}\left(\Sigma\right)\widehat{\mathcal{T}}\left(\Sigma\right)^{-1}\Sigma\widehat{\mathcal{F}}\left(\Sigma\right)^{\top}$. Therefore, a natural idea arise, that is, can we simplify CDREs \eqref{L-RE} into one derived from the FSLQ problem in some special case? To this end, we introduce the following assumption.

\textbf{(H5)} The state process $X(\cdot)$ for Problem (M-ZLQ) is one-dimensional and the process $\widetilde{T}_{11}(\cdot,i)$ introduced in \eqref{transformations} is invertible almost everywhere for any $i\in\mathcal{S}$.

Clearly, under the assumption (H5), the process $\widehat{A}(\cdot,i),\, \widetilde{F}(\cdot,i),\, G(\cdot,i),\, \widetilde{S}_{1}(\cdot,i),\, \widetilde{T}_{11}(\cdot,i),\, \Sigma(\cdot,i)$ are all one-dimensional. Then we have
\begin{align*}
 &\quad \widehat{\mathcal{F}}\left(\Sigma(s,i)\right)\widehat{\mathcal{T}}\left(\Sigma(s,i)\right)^{-1}\Sigma(s,i)\widehat{\mathcal{F}}\left(\Sigma(s,i)\right)^{\top}  \\
 &=\big[\widetilde{F}(s,i)+\Sigma(s,i)\widetilde{S}_{1}(s,i)\big]^{2}\Sigma(s,i)\big[1+ \Sigma(s,i)\widetilde{T}_{11}(s,i)\big]^{-1}\\
 &=\widetilde{T}_{11}(s,i)^{-1}\big[\widetilde{F}(s,i)+\Sigma(s,i)\widetilde{S}_{1}(s,i)\big]^{2}\big[\Sigma(s,i)+\widetilde{T}_{11}(s,i)^{-1}\big]\big[\widetilde{T}_{11}(s,i)^{-1}+ \Sigma(s,i)\big]^{-1}\\
 &\quad-\widetilde{T}_{11}(s,i)^{-2}\big[\widetilde{F}(s,i)+\Sigma(s,i)\widetilde{S}_{1}(s,i)\big]^{2}\big[\widetilde{T}_{11}(s,i)^{-1}+ \Sigma(s,i)\big]^{-1}\\
 &=\widetilde{T}_{11}(s,i)^{-1}\big[\widetilde{F}(s,i)+\Sigma(s,i)\widetilde{S}_{1}(s,i)\big]^{2}-\big[\widetilde{F}(s,i)+\Sigma(s,i)\widetilde{S}_{1}(s,i)\big]^{2}\big[\widetilde{T}_{11}(s,i)+ \Sigma(s,i)\widetilde{T}_{11}(s,i)^{2}\big]^{-1}\\
 &=\widetilde{T}_{11}(s,i)^{-1}\widetilde{F}(s,i)^{2}+2\widetilde{T}_{11}(s,i)^{-1}\widetilde{F}(s,i)\widetilde{S}_{1}(s,i)\Sigma(s,i)+\widetilde{T}_{11}(s,i)^{-1}\widetilde{S}_{1}(s,i)^{2}\Sigma(s,i)^{2}\\
 &\quad-\big[\widetilde{F}(s,i)+\Sigma(s,i)\widetilde{S}_{1}(s,i)\big]^{2}\big[\widetilde{T}_{11}(s,i)+ \Sigma(s,i)\widetilde{T}_{11}(s,i)^{2}\big]^{-1}
\end{align*}
and
\begin{align*}
 &\quad \widehat{\mathcal{H}}\left(\Sigma(s,i)\right)T_{22}(s,i)^{-1}\widehat{\mathcal{H}}\left(\Sigma(s,i)\right)^{\top}  \\
 &=\big[\widehat{H}(s,i)+\Sigma(s,i)S_{2}(s,i)^{\top}\big]T_{22}(s,i)^{-1}\big[\widehat{H}(s,i)^{\top}+S_{2}(s,i)\Sigma(s,i)\big]\\
 &=\widehat{H}(s,i)T_{22}(s,i)^{-1}\widehat{H}(s,i)^{\top}+\widehat{H}(s,i)T_{22}(s,i)^{-1}S_{2}(s,i)\Sigma(s,i)+\Sigma(s,i)S_{2}(s,i)^{\top}T_{22}(s,i)^{-1}\widehat{H}(s,i)^{\top}\\
 &\quad+\Sigma(s,i)S_{2}(s,i)^{\top}T_{22}(s,i)^{-1}S_{2}(s,i)\Sigma(s,i).
\end{align*}
Let
\begin{align*}
  &\mathbf{\mathcal{A}}(s,i)\triangleq\widetilde{T}_{11}(s,i)^{-1}\widetilde{F}(s,i)\widetilde{S}_{1}(s,i)+S_{2}(s,i)^{\top}T_{22}(s,i)^{-1}\widehat{H}(s,i)^{\top}-\widehat{A}(s,i),\\
  &\mathbf{\mathcal{B}}(s,i)\triangleq\widetilde{S}_{1}(s,i), \quad \mathbf{\mathcal{D}}(s,i)=\mathbf{\mathcal{R}_{2}}(s,i)\triangleq \widetilde{T}_{11}(s,i),\quad \mathbf{\mathcal{S}}(s,i)\triangleq\widetilde{F}(s,i)\\
  &\mathbf{\mathcal{Q}}(s,i)\triangleq\widetilde{T}_{11}(s,i)^{-1}\widetilde{F}(s,i)^{2}+ \widehat{H}(s,i)T_{22}(s,i)^{-1}\widehat{H}(s,i)^{\top},\\
  &\mathbf{\mathcal{R}_{1}}(s,i)\triangleq G(s,i)-\widetilde{T}_{11}(s,i)^{-1}\widetilde{S}_{1}(s,i)^{2}-S_{2}(s,i)^{\top}T_{22}(s,i)^{-1}S_{2}(s,i).
\end{align*}
Then CDREs \eqref{L-RE} can be rewritten as follows:
\begin{equation}\label{L-RE-rewritten}
    \left\{
    \begin{aligned}
    \dot{\Sigma}(s,i)&=-\Sigma(s,i)\mathbf{\mathcal{A}}(s,i)-\mathbf{\mathcal{A}}(s,i)^{\top}\Sigma(s,i)-\mathbf{\mathcal{Q}}(s,i)-\sum_{k=1}^{D}\lambda_{ik}(s)\Sigma(s,k)
    +\Sigma(s,i)\mathbf{\mathcal{R}_{1}}(s,i)^{-1}\Sigma(s,i)\\
    &\quad +\big[\mathbf{\mathcal{B}}(s,i)^{\top}\Sigma(s,i)+\mathbf{\mathcal{S}}(s,i)\big]^{\top}\big[\mathbf{\mathcal{R}_{2}}(s,i)+\mathbf{\mathcal{D}}(s,i)^{\top}\Sigma(s,i)\mathbf{\mathcal{D}}(s,i)\big]^{-1}\big[\mathbf{\mathcal{B}}(s,i)^{\top}\Sigma(s,i)+\mathbf{\mathcal{S}}(s,i)\big],\\
    \Sigma(T,i)&=0.
    \end{aligned}
    \right.
\end{equation}

We point out that the CDREs \eqref{L-RE-rewritten} is closely related to an FSLQ problem, whose state process and cost functional are given by:
\begin{equation}\label{related-state}
 \left\{
 \begin{aligned}
 &d\mathbb{X}(s)=\left[\mathbf{\mathcal{A}}(s,\alpha(s))\mathbb{X}(s)+u_{1}(s)+\mathbf{\mathcal{B}}(s,\alpha(s))u_{2}(s)\right]ds+\mathbf{\mathcal{D}}(s,\alpha(s))u_{2}(s)dW(s),\quad s\in[0,T]\\
 &\mathbb{X}(0)=x,\quad \alpha(0)=i,
 \end{aligned}
 \right.
\end{equation}
and
\begin{equation}\label{related-cost-functional}
\begin{aligned}
\mathbb{J}(x,i,\mathbf{u}(\cdot))=\mathbb{E}\int_{0}^{T}\Big[&\big<\mathbf{\mathcal{Q}}(s,\alpha(s))\mathbb{X}(s),\mathbb{X}(s)\big>
+\big<\mathbf{\mathcal{R}_{1}}(s,\alpha(s))u_{1}(s),u_{1}(s)\big>+\big<\mathbf{\mathcal{R}_{2}}(s,\alpha(s))u_{2}(s),u_{2}(s)\big>\\
&+2\big<\mathbf{\mathcal{S}}(s,\alpha(s))\mathbb{X}(s),u_{2}(s)\big>\Big]ds,
\end{aligned}
\end{equation}
where $\mathbf{u}(\cdot)\triangleq [u_{1}(\cdot)^{\top},u_{2}(\cdot)^{\top}]^{\top}$ represents the control pair of the problem. For simplicity, we denote a problem with state \eqref{related-state} and cost functional \eqref{related-cost-functional} as \textbf{Problem (M-LQ)}. The following result provides a relation between the solvability of CDREs \eqref{L-RE-rewritten} and the uniform convexity of cost functional for Problem (M-LQ).

\begin{theorem}\label{RE-solvability}
 The CDREs \eqref{L-RE-rewritten} admits a unique solution $\mathbf{\Sigma}(\cdot)\equiv\left[\Sigma(s,1),\cdots,\Sigma(s,D)\right]$ such that
\begin{equation}\label{RE-1}
  \mathbf{\mathcal{R}_{1}}(\cdot,i)^{-1}\gg 0 \quad \text{and} \quad \mathbf{\mathcal{R}_{2}}(\cdot,i)+\mathbf{\mathcal{D}}(\cdot,i)^{\top}\Sigma(\cdot,i)\mathbf{\mathcal{D}}(\cdot,i)\gg 0,\quad \forall i\in\mathcal{S},
\end{equation}
if and only if the Problem (M-LQ) is uniformly convex, that is,
$$\mathbb{J}(0,i;\mathbf{u}(\cdot))\geq \delta \mathbf{E}\int_{0}^{T}\left[\big|u_{1}(s)\big|^{2}+\big|u_{2}(s)\big|^{2}\right]ds, \quad \text{for some } \delta>0.$$
\end{theorem}

\begin{proof}
Let
\begin{align*}
    &\mathbb{A}(s,i)=\mathbf{\mathcal{A}}(s,i),\quad \mathbb{B}(s,i)=[I,\mathbf{\mathcal{B}}(s,i)],\quad\mathbb{C}(s,i)=0,\quad \mathbb{D}(s,i)=[0,\mathbf{\mathcal{D}}(s,i)],\quad \mathbb{Q}(s,i)=\mathbf{\mathcal{Q}}(s,i),\\
    &\mathbb{S}(s,i)=\left[\begin{matrix} 0\\\mathbf{\mathcal{S}}(s,i) \end{matrix}\right],\quad\mathbb{R}(s,i)=\left[\begin{matrix} \mathbf{\mathcal{R}}_{1}(s,i) & 0\\0 & \mathbf{\mathcal{R}}_{2}(s,i) \end{matrix}\right].
\end{align*}
Then the state process \eqref{related-state} and cost functional \eqref{related-cost-functional} can be rewritten as
\begin{equation}\label{related-state-2}
    \left\{
    \begin{aligned}
        &d\mathbb{X}(s)=\big[\mathbb{A}(s,\alpha(s))\mathbb{X}(t)+\mathbb{B}(s,\alpha(s))\mathbf{u}(s)\big]ds+\big[\mathbb{C}(s,\alpha(s))\mathbb{X}(t)+\mathbb{D}(s,\alpha(s))\mathbf{u}(s)\big]dW(s),\\
        &\mathbb{X}(0)=x,\quad \alpha(0)=i,
    \end{aligned}
    \right.
\end{equation}
and
\begin{equation}
    \mathbb{J}(x,i,\mathbf{u})=\mathbb{E}\int_{0}^{T}\left<\left(\begin{matrix}
  \mathbb{Q}(s,\alpha(s)) & \mathbb{S}(s,\alpha(s))^{\top} \\ \mathbb{S}(s,\alpha(s))^{\top}  &  \mathbb{R}(s,\alpha(s))
    \end{matrix}\right)
    \left(\begin{matrix}
    \mathbb{X}(s)\\ \mathbf{u}(s)
    \end{matrix}\right),\left(\begin{matrix}
    \mathbb{X}(s)\\ \mathbf{u}(s)
    \end{matrix}\right)\right>ds.
\end{equation}
Further, one can easily verify that CDREs \eqref{L-RE-rewritten} is equivalent to the following form
\begin{equation}
\left\{
    \begin{aligned}
        \Sigma(s,i)&=-\Sigma(s,i)\mathbb{A}(s,i)-\mathbb{A}(s,i)^{\top}\Sigma(s,i)-\mathbb{C}(s,i)^{\top}\Sigma(s,i)\mathbb{C}(s,i)-\mathbb{Q}(s,i)-\sum_{k=1}^{D}\lambda_{ik}\Sigma(s,k)\\
        &\quad +\left[\Sigma(s,i)\mathbb{B}(s,i)+\mathbb{C}(s,i)^{\top}\Sigma(s,i)\mathbb{D}(s,i)+\mathbb{S}(s,i)^{\top}\right]\left[\mathbb{R}(s,i)+\mathbb{D}(s,i)^{\top}\Sigma(s,i)\mathbb{D}(s,i)\right]^{-1}\\
        &\quad \times \left[\Sigma(s,i)\mathbb{B}(s,i)+\mathbb{C}(s,i)^{\top}\Sigma(s,i)\mathbb{D}(s,i)+\mathbb{S}(s,i)^{\top}\right]^{\top},\\
        \Sigma(T,i)&=0.
    \end{aligned}
    \right.
\end{equation}
Hence, the desired result follows from \citet{Zhang.X.2021_ILQM} directly.
\end{proof}

\begin{remark}\label{rmk-RE-solvability}
We point out that the results in Theorem \ref{RE-solvability}  are still valid for the multi-dimensional case of {\rm Problem (M-LQ)}. 
Moreover,
if
$$
\left[
\begin{matrix}
\mathbf{\mathcal{Q}}(\cdot,i) & 0 & \mathbf{\mathcal{S}}(\cdot,i)^{\top}\\
0 & \mathbf{\mathcal{R}_{1}}(\cdot,i) & 0\\
\mathbf{\mathcal{S}}(\cdot,i) & 0 & \mathbf{\mathcal{R}_{2}}(\cdot,i)
\end{matrix}
\right]\gg 0,\quad \forall i\in\mathcal{S},
$$
then the CDREs \eqref{L-RE-rewritten} admits a unique solution $\mathbf{\Sigma}(\cdot)$ satisfying \eqref{RE-1}.
\end{remark}

\section{Examples}\label{section-example}
This section presents two concrete examples to illustrate the results of previous sections. For simplicity,  we suppose that both state process and control processes are one-dimensional and the Markovian chain $\alpha(\cdot)$ only has two states $\mathcal{S}=\{1,2\}$ with time invariant generator
\begin{equation}\label{generator}
 \left[\begin{matrix}
\lambda_{11}(s) & \lambda_{12}(s)\\
\lambda_{21}(s) & \lambda_{22}(s)
\end{matrix}\right]\equiv
\left[\begin{matrix}
-0.5 & 0.5 \\ 0.7 & -0.7
\end{matrix}\right].
\end{equation}

The first example below provides a special case under which the Problem (M-ZLQ) admits a Stackelberg equilibrium and both CDREs \eqref{F-RE} for the follower's problem and CDREs \eqref{L-RE} for the leader's problem are explicitly solvable.
\begin{example}\rm
 Consider the following state process
\begin{equation}\label{example1-state}
\left\{
\begin{aligned}
    &dX(s)=B_{2}(\alpha_{s})u_{2}(s)ds+D_{1}(\alpha_{s})u_{1}(s)dW(s),\quad s\in[0,1],\\
    &X(0)=x,\quad \alpha_{0}=i,
 \end{aligned}
 \right.
\end{equation}
and criterion functional
\begin{equation}\label{example1-functional}
 J(x,i;u_{1}(\cdot),u_{2}(\cdot))=\mathbb{E}\left[\int_{0}^{1}\left(R_{1}(\alpha_{s})\big|u_{1}(s)\big|^{2}+R_{2}(\alpha_{s})\big|u_{2}(s)\big|^{2}\right)ds-\big|X(1)\big|^{2}\right],
\end{equation}
where the coefficients in state process and weighting matrices in performance functional are only depend on Markovian chain $\alpha(\cdot)$ and given by:
$$
\left[\begin{matrix}
B_{2}(1)\\B_{2}(2)
\end{matrix}\right]=
\left[\begin{matrix}
1\\-2
\end{matrix}\right],\quad
\left[\begin{matrix}
D_{1}(1)\\D_{1}(2)
\end{matrix}\right]=
\left[\begin{matrix}
2\\1
\end{matrix}\right],\quad
\left[\begin{matrix}
R_{1}(1)\\R_{1}(2)
\end{matrix}\right]=
\left[\begin{matrix}
5\\2
\end{matrix}\right],\quad
\left[\begin{matrix}
R_{2}(1)\\R_{2}(2)
\end{matrix}\right]=
\left[\begin{matrix}
-1\\-4
\end{matrix}\right].$$
Then we can verify that
\begin{align*}
J(0,i;u_{1}(\cdot),0)&=\mathbb{E}\left[\int_{0}^{1}R_{1}(\alpha_{s})\big|u_{1}(s)\big|^{2}ds-\left(\int_{0}^{1}D_{1}(\alpha_{s})u_{1}(s)dW(s)\right)^{2}\right]\\
&=\mathbb{E}\left[\int_{0}^{1}\left(R_{1}(\alpha_{s})-D_{1}(\alpha_{s})^{2}\right)\big|u_{1}(s)\big|^{2}ds\right]\\
&=\mathbb{E}\int_{0}^{1}\big|u_{1}(s)\big|^{2}ds,
\end{align*}
and
\begin{align*}
J(0,i;0,u_{2}(\cdot))&\leq\mathbb{E}\int_{0}^{1}R_{2}(\alpha_{s})\big|u_{2}(s)\big|^{2}ds\leq -\mathbb{E}\int_{0}^{1}\big|u_{1}(s)\big|^{2}ds,
\end{align*}
which implies the uniform convexity-concavity conditions (H3)-(H4) hold.  Consequently, the game admits a unique Stackelberg equilibrium  $\left(u_{1}^{*},u_{2}^{*}\right)$. Next, we will derive the explicit representation of $\left(u_{1}^{*},u_{2}^{*}\right)$ using the obtained results in previous sections.

We first note that the associated CDREs \eqref{F-RE} in this special case are given by:
\begin{equation}\label{example1-F-RE}
    \left\{
    \begin{aligned}
        &\dot{P}(s,1)=0.5 P(s,1)-0.5 P(s,2)\\
        &\dot{P}(s,1)=-0.7 P(s,1)+0.7 P(s,2)\\
        &P(1,1)=P(1,2)=-1.
    \end{aligned}
    \right.
\end{equation}
Clearly, the unique solution to \eqref{example1-F-RE} is $P(s,1)=P(s,2)\equiv -1$. Then by Theorem \ref{thm-M-SG-F}, we have
\begin{equation}\label{example1-F-rational-reaction}
    \bar{u}_{1}[x,i,u_{2}](s)=-D_{1}(\alpha_{s})^{\top}Z^{u_{2}}(s),
\end{equation}
where $\left(Y^{u_{2}}(\cdot),Z^{u_{2}}(\cdot),\mathbf{\Gamma}^{u_{2}}(\cdot)\right)$ solves the BSDE:
\begin{equation}\label{example1-L-state}
    \left\{
    \begin{aligned}
       &dY^{u_{2}}(s)=\widehat{H}(\alpha_{s})u_{2}(s)dt+Z^{u_{2}}(s)dW(s)+\mathbf{\Gamma}^{u_{2}}(s)\cdot d\mathbf{\widetilde{N}}(s), \quad s\in[0,1],\\
       &Y^{u_{2}}(1)=0.
    \end{aligned}
    \right.
\end{equation}
In addition,
one has
$$J(x,i;\bar{u}_{1}[x,i,u_{2}](\cdot),u_{2}(\cdot))=-J_{L}(u_{2}(\cdot))-x^{2},$$
where
\begin{equation}\label{example1-L-cost}
J_{L}(i,u_{2}(\cdot))=\mathbb{E}\int_{0}^{1}\left[\left<T_{11}(\alpha_{s})Z(s),Z(s)\right>+\left<T_{22}(\alpha_{s})u_{2}(s),u_{2}(s)\right>+2\left<S_{2}(\alpha_{s})u_{2}(s),Y(s)\right>\right]ds-2\left<Y(0),x\right>.
\end{equation}
In the above, the coefficients $\widehat{H},\, T_{11},\, T_{22},\, S_{2}$ are given by
\begin{align*}
    \left[\begin{matrix}
    \widehat{H}(1)\\\widehat{H}(2)
    \end{matrix}\right]=
    \left[\begin{matrix}
    1\\-2
    \end{matrix}\right],\quad
    \left[\begin{matrix}
    T_{11}(1)\\T_{11}(2)
    \end{matrix}\right]=
    \left[\begin{matrix}
    4\\1
    \end{matrix}\right],\quad
    \left[\begin{matrix}
    T_{22}(1)\\T_{22}(2)
    \end{matrix}\right]=
    \left[\begin{matrix}
    1\\4
    \end{matrix}\right],\quad
    \left[\begin{matrix}
    S_{2}(1)\\S_{2}(2)
    \end{matrix}\right]=
    \left[\begin{matrix}
    -1\\2
    \end{matrix}\right].
\end{align*}
Hence, the leader's problem now is to find the optimal control to minimize the cost functional \eqref{example1-L-cost} with state constraint \eqref{example1-L-state}. Let $u_{2}^{*}$ be the optimal control for the leader. Then, according to Theorem \ref{L-openloop-solvability}, the optimality system for the leader's problem is given by
\begin{equation}\label{example1-L-optimality-system}
\left\{
\begin{aligned}
&d\phi^{*}(s)=S_{2}(\alpha_{s})u_{2}^{*}(s)ds+T_{11}(\alpha_{s})Z^{*}(s)dW(s),\\
&dY^{*}(s)=\widehat{H}(\alpha_{s})u_{2}^{*}(s)ds+Z^{*}(s)dW(s)+\mathbf{\Gamma}^{*}(s)d\widetilde{N}(s),\\
&-\widehat{H}(\alpha_{s})\phi^{*}(s)+S_{2}(\alpha_{s})Y^{*}(s)+T_{22}(\alpha_{s})u_{2}^{*}(s)=0,\\
&\phi^{*}(0)=-x, \quad Y^{*}(1)=0,\quad  \alpha_{0}=i.
\end{aligned}
\right.
\end{equation}
To decouple the above optimality system, we only need to make the ansatz $Y^{*}(s)=-\Sigma(s,\alpha_{s})\phi^{*}(s)$.  From \eqref{L-RE}, we know that $\Sigma(s,\alpha_{s})$ is determined by the following coupled differential equations:
\begin{equation}\label{example1-L-RE}
    \left\{
    \begin{aligned}
    &\dot{\Sigma}(s,1)=-\left[1-\Sigma(s,1)\right]^{2}+0.5 \Sigma(s,1)-0.5 \Sigma (s,2),\\
    &\dot{\Sigma}(s,2)=-\left[1-\Sigma(s,2)\right]^{2}-0.7 \Sigma(s,1)+0.7 \Sigma (s,2),\\
    &\Sigma(1,1)=\Sigma(1,2)=0.
    \end{aligned}
    \right.
\end{equation}
Solving the above differential equation, we obtain
$$\Sigma(s,1)=\Sigma(s,2)=\frac{s-1}{s-2},\quad s\in[0,1].$$
Then, by Theorem \ref{thm-main}, the Stackelberg equilibrium is given by
\begin{equation}\label{example1-Stackelberg-equilibrium}
    \left\{
    \begin{aligned}
    &u_{1}^{*}(s)\equiv 0,\\
    &u_{2}^{*}(s)=T_{22}(\alpha_{s})^{-1}\left[\widehat{H}(\alpha_{s})+\frac{s-1}{s-2}S_{2}(\alpha_{s})\right]\phi^{*}(s),
    \end{aligned}
    \right.
\end{equation}
where
\begin{equation}\label{example1-phi}
   \left\{
   \begin{aligned}
   &d\phi^{*}(s)=S_{2}(\alpha_{s})T_{22}(\alpha_{s})^{-1}\left[\widehat{H}(\alpha_{s})+\frac{s-1}{s-2}S_{2}(\alpha_{s})\right]\phi^{*}(s)ds,\\
   &\phi^{*}(0)=-x,
   \end{aligned}
   \right.
\end{equation}
and the equilibrium value function is
\begin{equation}\label{example1-equilibrium-value-function}
    V(x,i)=\frac{1}{2}x-x^{2}.
\end{equation}
\end{example}

The next example shows how to use the derived results to solve a BSLQ problem.
\begin{example}\rm
Consider the following state process
\begin{equation}\label{example2-state}
\left\{
\begin{aligned}
    &dY(s)=\left[\widehat{A}(\alpha_{s})Y(s)+\widehat{C}(\alpha_{s})Z(s)+\widehat{H}(\alpha_{s})u(s)\right]ds+Z(s)dW(s)+\mathbf{\Gamma}(s)\cdot d\mathbf{\widetilde{N}}(s),\quad s\in[0,1],\\
    &Y(1)=m\in L_{\mathcal{F}_{1}}^{2}(\mathbb{R}),\quad \alpha_{0}=i,
 \end{aligned}
 \right.
\end{equation}
and cost functional
\begin{equation}\label{example2-functional}
 J(x,i;u(\cdot))=\mathbb{E}\int_{0}^{1}\left< \left[\begin{matrix}
 G(\alpha_{s}) & S_{1}(\alpha_{s})^{\top} & S_{2}(\alpha_{s})^{\top}\\
 S_{1}(\alpha_{s}) & T_{11}(\alpha_{s}) & T_{12}(\alpha_{s})\\
 S_{2}(\alpha_{s}) & T_{22}(\alpha_{s}) & T_{22}(\alpha_{s})
 \end{matrix}\right]\left[\begin{matrix}
 Y(s)\\Z(s)\\u(s)
 \end{matrix}\right], \left[\begin{matrix}
 Y(s)\\Z(s)\\u(s)
 \end{matrix}\right]\right>ds,
\end{equation}
where
$$
\begin{array}{c}
  \left[\begin{matrix}
  \widehat{A}(1)\\\widehat{A}(2)
  \end{matrix}\right]= \left[\begin{matrix}
  1\\-1
  \end{matrix}\right],\quad
   \left[\begin{matrix}
  \widehat{C}(1)\\ \widehat{C}(2)
  \end{matrix}\right]= \left[\begin{matrix}
  2\\1
  \end{matrix}\right],\quad
   \left[\begin{matrix}
  \widehat{H}(1)\\ \widehat{H}(2)
  \end{matrix}\right]= \left[\begin{matrix}
  -4\\2
  \end{matrix}\right],\\[5mm]
   \left[\begin{matrix}
 G(1) & S_{1}(1)^{\top} & S_{2}(1)^{\top}\\
 S_{1}(1) & T_{11}(1) & T_{12}(1)\\
 S_{2}(1) & T_{21}(1) & T_{22}(1)
 \end{matrix}\right]= \left[\begin{matrix}
  3 & -1 & 1 \\ -1 & 4 & -1 \\ 1 & -1 & 2
  \end{matrix}\right],\quad
  \left[\begin{matrix}
 G(2) & S_{1}(2)^{\top} & S_{2}(2)^{\top}\\
 S_{1}(2) & T_{11}(2) & T_{12}(2)\\
 S_{2}(2) & T_{22}(2) & T_{22}(2)
 \end{matrix}\right]= \left[\begin{matrix}
  5 & 1 & 2 \\ 1 & 4 & -2 \\ 2 & -2 & 4
  \end{matrix}\right].
  \end{array}
$$
Now, we aim to find an optimal control $u^{*}(\cdot)$ such that
$$J(x,i;u^{*}(\cdot))=\inf_{u\in L_{\mathbb{F}}^{2}(0,1;\mathbb{R})}J(x,i;u(\cdot))\triangleq V(m,i).$$
One can easily verify the cost functional \eqref{example2-functional} is uniformly convex. Hence, the above M-BSLQ problem admits a unique optimal control. According to the results  derived in Section \ref{section-construction}, we first need to make the following transformations  to construct the explicit optimal control:
\begin{equation}\label{example2-transformations}
\left\{
\begin{aligned}
&\upsilon(s)=u_{2}(s)+T_{22}(s,\alpha(s))^{-1}T_{21}(s,\alpha(s))Z(s),\\
&\widetilde{F}(i)=\widehat{C}(i)-\widehat{H}(i)T_{22}(i)^{-1}T_{21}(i),\\
&\widetilde{S}_{1}(i)=S_{1}(i)-T_{12}(i)T_{22}(i)^{-1}S_{2}(i),\\
&\widetilde{T}_{11}(i)=T_{11}(i)-T_{12}(i)T_{22}(i)^{-1}T_{21}(i).
\end{aligned}
\right.
\end{equation}
After some simple algebraic calculations, we can obtain
$$\left[\begin{matrix}
  \widetilde{F}(1)\\\widetilde{F}(2)
  \end{matrix}\right]= \left[\begin{matrix}
  0\\2
  \end{matrix}\right],\quad
   \left[\begin{matrix}
  \widetilde{S}_{1}(1)\\ \widetilde{S}_{1}(2)
  \end{matrix}\right]= \left[\begin{matrix}
  -0.5\\2
  \end{matrix}\right],\quad
   \left[\begin{matrix}
  \widetilde{T}_{11}(1)\\ \widetilde{T}_{11}(2)
  \end{matrix}\right]= \left[\begin{matrix}
  3.5\\3
  \end{matrix}\right]$$
Substituting the above value into \eqref{L-notations-2},  the CDREs \eqref{L-RE} in this example can be simplified as follows:
\begin{equation}\label{example2-L-RE}
   \left\{
   \begin{aligned}
   \dot{\Sigma}(s,1)&=6\Sigma(s,1)-8+\frac{17}{7}\Sigma(s,1)^{2}+\left[14+49\Sigma(s,1)\right]^{-1}\Sigma(s,1)^{2}-\left[-0.5\Sigma(s,1)+0.5\Sigma(s,2)\right]\\
   \dot{\Sigma}(s,2)&=-\frac{20}{3}\Sigma(s,2)-\frac{7}{3}+\frac{8}{3}\Sigma(s,2)^{2}+\left[3+9\Sigma(s,2)\right]^{-1}\left[2+2\Sigma(s,2)\right]^{2}-\left[0.7\Sigma(s,1)-0.7\Sigma(s,2)\right]\\
   \Sigma(1,1)&=0,\quad \Sigma(1,2)=0.
   \end{aligned}
   \right.
\end{equation}
Taking
\begin{align*}
    &\left[\begin{matrix} \mathcal{A}(1)\\ \mathcal{A}(2)\end{matrix}\right]=\left[\begin{matrix} -3\\ 10/3 \end{matrix}\right],\quad
    \left[\begin{matrix} \mathcal{B}(1)\\ \mathcal{B}(2)\end{matrix}\right]=\left[\begin{matrix} 1\\ 2 \end{matrix}\right],\quad
    \left[\begin{matrix} \mathcal{D}(1)\\ \mathcal{D}(2)\end{matrix}\right]=\left[\begin{matrix} 7\\ 3 \end{matrix}\right],\\
    &\left[\begin{matrix} \mathcal{Q}(1)\\ \mathcal{Q}(2)\end{matrix}\right]=\left[\begin{matrix} 8\\ 7/3 \end{matrix}\right],\quad
    \left[\begin{matrix} \mathcal{S}(1)\\ \mathcal{S}(2)\end{matrix}\right]=\left[\begin{matrix} 0\\ 2 \end{matrix}\right],\quad
    \left[\begin{matrix} \mathcal{R}_{1}(1)\\ \mathcal{R}_{1}(2)\end{matrix}\right]=\left[\begin{matrix} 7/17\\3/8 \end{matrix}\right],\quad
     \left[\begin{matrix} \mathcal{R}_{1}(1)\\ \mathcal{R}_{1}(2)\end{matrix}\right]=\left[\begin{matrix} 14\\3 \end{matrix}\right],
\end{align*}
then CAREs \eqref{example2-L-RE} admits the representation \eqref{L-RE-rewritten}. Obviously, we have
$$
\left[\begin{matrix}
8 & 0 & 0\\
0 & \frac{7}{17} & 0\\
0 & 0 & 14
\end{matrix}
\right]\gg 0,\quad
\left[\begin{matrix}
\frac{7}{3} & 0 & 2\\
0 & \frac{3}{8} & 0\\
2 & 0 & 3
\end{matrix}
\right]\gg 0.
$$
It follows from Theorem \ref{RE-solvability} and Remark \ref{rmk-RE-solvability} that the CDREs \eqref{example2-L-RE} admit a unique solution. Using the well-known finite difference method, we can present the following numerical solution figure of $\left[\Sigma(\cdot,1),\Sigma(\cdot,2)\right]$ for a clearer visualization.
\begin{figure}[H]
  \centering
  \begin{tabular}{c}
    \includegraphics[width=250pt,height=150pt]{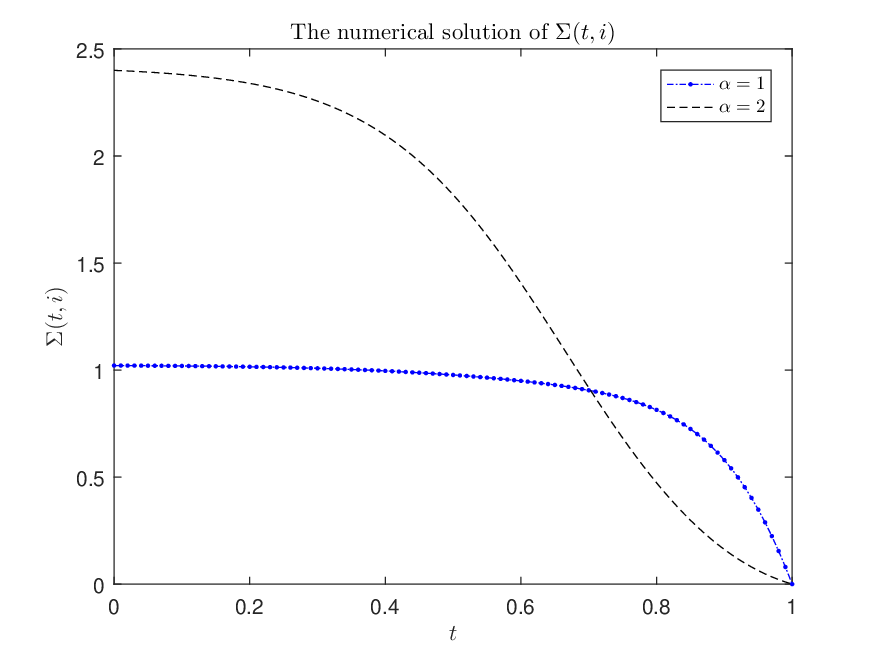}
  \end{tabular}
\end{figure}
With the solution to \eqref{example2-L-RE}, we now introduce the following BSDE:
\begin{equation}\label{example2-L-L-BSDE}
\left\{
\begin{aligned}
	d\varphi(s)&=\big\{\big[\widehat{A}-\widehat{\mathcal{H}}(\Sigma)T_{22}^{-1}S_{2}-\widehat{\mathcal{F}}(\Sigma)\widehat{\mathcal{T}}(\Sigma)^{-1}\Sigma \widetilde{S}_{1}+\Sigma G\big]\varphi
	+\widehat{\mathcal{F}}(\Sigma)\widehat{\mathcal{T}}(\Sigma)^{-1}\theta\big\}ds\\
	&\quad+\theta dW(s)+\mathbf{\gamma}\cdot d\mathbf{\widetilde{N}}(s), \quad s \in[0, T], \\
	\varphi(T)&=m,
\end{aligned}
\right.
\end{equation}
and SDE:
\begin{equation}\label{example2-phi}
	\left\{
	\begin{aligned}
		d\phi^{*}(s)&=\Big\{\big[-\widehat{A}^{\top}-G\Sigma+\widetilde{S}_{1}^{\top}\widehat{\mathcal{T}}(\Sigma)^{-1}\Sigma\widehat{\mathcal{F}}(\Sigma)^{\top}+S_{2}^{\top}T_{22}^{-1}\widehat{\mathcal{H}}(\Sigma)^{\top}\big]\phi^{*}-\big[\widetilde{S}_{1}^{\top}\widehat{\mathcal{T}}(\Sigma)^{-1}\Sigma \widetilde{S}_{1}\\
		&\quad+S_{2}^{\top}T_{22}^{-1}S_{2}-G\big]\varphi+\widetilde{S}_{1}^{\top}\widehat{\mathcal{T}}(\Sigma)^{-1}\theta\Big\}ds-\big[\widehat{\mathcal{T}}(\Sigma)^{-1}\big]^{\top}\big[\widehat{\mathcal{F}}(\Sigma)^{\top}\phi^{*}-\widetilde{T}_{11}\theta-\widetilde{S}_{1}\varphi\big]dW(s),\\
		\phi^{*}(0)=&0,
	\end{aligned}
	\right.
\end{equation}
where $\widehat{\mathcal{T}}(\Sigma), \, \widehat{\mathcal{F}}(\Sigma),\, \widehat{\mathcal{H}}(\Sigma)$ are defined in \eqref{L-notations-2}.
Then by Theorem \ref{L-thm-control}, the unique optimal control is given by
\begin{equation*}
\begin{aligned}
u_{2}^{*}(s)&=T_{22}(s,\alpha(s))^{-1}\big[\widehat{\mathcal{H}}\left(\Sigma(s,\alpha(s))\right)^{\top}-T_{21}(s,\alpha(s))\widehat{\mathcal{T}}\left(\Sigma(s,\alpha(s))\right)^{-1}\Sigma(s,\alpha(s))\widehat{\mathcal{F}}\left(\Sigma(s,\alpha(s))\right)^{\top}\big]\phi^{*}(s)\\
&\quad-T_{22}(s,\alpha(s))^{-1}\Big\{T_{21}(s,\alpha(s))\widehat{\mathcal{T}}\left(\Sigma(s,\alpha(s))\right)^{-1}\big[\theta(s)-\Sigma(s,\alpha(s))\widetilde{S}_{1}(s,\alpha(s))\varphi(s)\big]
 +S_{2}(s,\alpha(s))\varphi(s)\Big\}.
\end{aligned}
\end{equation*}
In addition, the value function follows from the Theorem \ref{L-thm-value} that
\begin{equation*}
	\begin{aligned}
		V(m,i)
		&=\mathbb{E}\int_{0}^{T}\Big[
		-\big<\varphi,\big(\widetilde{S}_{1}^{\top}\widehat{\mathcal{T}}(\Sigma)^{-1}\Sigma \widetilde{S}_{1}
		+S_{2}^{\top}T_{22}^{-1}S_{2}-G\big)\varphi\big>+2\big<\varphi,\widetilde{S}_{1}^{\top}\widehat{\mathcal{T}}(\Sigma)^{-1}\theta\big>+\big<\theta,\widetilde{T}_{11}\widehat{\mathcal{T}}(\Sigma)^{-1}\theta\big>\Big]ds.
	\end{aligned}
\end{equation*}

\end{example}




\begin{thebibliography}{34}
\expandafter\ifx\csname natexlab\endcsname\relax\def\natexlab#1{#1}\fi
\providecommand{\url}[1]{\texttt{#1}}
\providecommand{\href}[2]{#2}
\providecommand{\path}[1]{#1}
\providecommand{\DOIprefix}{doi:}
\providecommand{\ArXivprefix}{arXiv:}
\providecommand{\URLprefix}{URL: }
\providecommand{\Pubmedprefix}{pmid:}
\providecommand{\doi}[1]{\href{http://dx.doi.org/#1}{\path{#1}}}
\providecommand{\Pubmed}[1]{\href{pmid:#1}{\path{#1}}}
\providecommand{\bibinfo}[2]{#2}
\ifx\xfnm\relax \def\xfnm[#1]{\unskip,\space#1}\fi
\bibitem[{Zhang et~al.(2021)Zhang, Li, and Xiong}]{Zhang.X.2021_ILQM}
\bibinfo{author}{X.~Zhang}, \bibinfo{author}{X.~Li},
  \bibinfo{author}{J.~Xiong},
\newblock \bibinfo{title}{Open-loop and closed-loop solvabilities for
  stochastic linear quadratic optimal control problems of {M}arkovian regime
  switching system},
\newblock \bibinfo{journal}{ESAIM: Control, Optimisation and Calculus of
  Variations} \bibinfo{volume}{27} (\bibinfo{year}{2021}) \bibinfo{pages}{35}.
\bibitem[{Kushner(1962)}]{kushner.H.1962_OSC}
\bibinfo{author}{H.~Kushner},
\newblock \bibinfo{title}{Optimal stochastic control},
\newblock \bibinfo{journal}{IRE Transactions on Automatic Control}
  \bibinfo{volume}{7} (\bibinfo{year}{1962}) \bibinfo{pages}{120--122}.
\bibitem[{Wonham(1968)}]{Wonham.W.M.1968}
\bibinfo{author}{W.~M. Wonham},
\newblock \bibinfo{title}{On a matrix riccati equation of stochastic control},
\newblock \bibinfo{journal}{SIAM Journal on Control} \bibinfo{volume}{6}
  (\bibinfo{year}{1968}) \bibinfo{pages}{681--697}.
\bibitem[{Tang(2003)}]{Tang.S.J.2003_LQR}
\bibinfo{author}{S.~Tang},
\newblock \bibinfo{title}{General linear quadratic optimal stochastic control
  problems with random coefficients: linear stochastic {H}amilton systems and
  backward stochastic riccati equations},
\newblock \bibinfo{journal}{SIAM Journal on Control and Optimization}
  \bibinfo{volume}{42} (\bibinfo{year}{2003}) \bibinfo{pages}{53--75}.
\bibitem[{Wu and Wang(2003)}]{Wu.Z.2003_LQJR}
\bibinfo{author}{Z.~Wu}, \bibinfo{author}{X.-R. Wang},
\newblock \bibinfo{title}{{FBSDE} with {P}oisson process and its application to
  linear quadratic stochastic optimal control problem with random jumps},
\newblock \bibinfo{journal}{Acta Automatica Sinica} \bibinfo{volume}{29}
  (\bibinfo{year}{2003}) \bibinfo{pages}{821--826}.
\bibitem[{Hu and Oksendal(2008)}]{Hu.Y.Z.2008_LQJP}
\bibinfo{author}{Y.~Hu}, \bibinfo{author}{B.~Oksendal},
\newblock \bibinfo{title}{Partial information linear quadratic control for jump
  diffusions},
\newblock \bibinfo{journal}{SIAM Journal on Control and Optimization}
  \bibinfo{volume}{47} (\bibinfo{year}{2008}) \bibinfo{pages}{1744--1761}.
\bibitem[{Ji and Chizeck(1990)}]{Ji.Y.D.1990_LQM}
\bibinfo{author}{Y.~Ji}, \bibinfo{author}{H.~J. Chizeck},
\newblock \bibinfo{title}{Controllability, stabilizability, and continuous-time
  {M}arkovian jump linear quadratic control},
\newblock \bibinfo{journal}{{IEEE} Transactions on Automatic Control}
  \bibinfo{volume}{35} (\bibinfo{year}{1990}) \bibinfo{pages}{777--788}.
\bibitem[{Ji and Chizeck(1991)}]{Ji.Y.D.1991_LQGJ}
\bibinfo{author}{Y.~Ji}, \bibinfo{author}{H.~J. Chizeck},
\newblock \bibinfo{title}{Jump linear quadratic {G}aussian control in
  continuous time},
\newblock in: \bibinfo{booktitle}{1991 American Control Conference},
  \bibinfo{publisher}{IEEE}, \bibinfo{year}{1991}, pp.
  \bibinfo{pages}{2676--2681}.
\bibitem[{Qing and Yin(1999)}]{Zhang.Q.1999_LQG}
\bibinfo{author}{Z.~Qing}, \bibinfo{author}{G.~G. Yin},
\newblock \bibinfo{title}{On nearly optimal controls of hybrid {LQG} problems},
\newblock \bibinfo{journal}{IEEE Transactions on Automatic Control}
  \bibinfo{volume}{44} (\bibinfo{year}{1999}) \bibinfo{pages}{2271--2282}.
\bibitem[{Lim and Zhou(2001)}]{Lim.A.E.B.2001_BLQ}
\bibinfo{author}{A.~E.~B. Lim}, \bibinfo{author}{X.~Y. Zhou},
\newblock \bibinfo{title}{Linear-quadratic control of backward stochastic
  differential equations},
\newblock \bibinfo{journal}{SIAM Journal on Control and Optimization}
  \bibinfo{volume}{40} (\bibinfo{year}{2001}) \bibinfo{pages}{450--474}.
\bibitem[{Huang et~al.(2009)Huang, Wang, and Xiong}]{Huang.J.H.2009_MPBSDEP}
\bibinfo{author}{J.~Huang}, \bibinfo{author}{G.~Wang},
  \bibinfo{author}{J.~Xiong},
\newblock \bibinfo{title}{A maximum principle for partial information backward
  stochastic control problems with applications},
\newblock \bibinfo{journal}{SIAM Journal on Control and Optimization}
  \bibinfo{volume}{48} (\bibinfo{year}{2009}) \bibinfo{pages}{2106--2117}.
\bibitem[{Wang et~al.(2012)Wang, Wu, and Xiong}]{Wang.G.C.2012_BLQP}
\bibinfo{author}{G.~Wang}, \bibinfo{author}{Z.~Wu}, \bibinfo{author}{J.~Xiong},
\newblock \bibinfo{title}{Partial information {LQ} optimal control of backward
  stochastic differential equations},
\newblock in: \bibinfo{booktitle}{Proceedings of the 10th World Congress on
  Intelligent Control and Automation}, \bibinfo{publisher}{IEEE},
  \bibinfo{year}{2012}, pp. \bibinfo{pages}{1694--1697}.
\bibitem[{Sun and Wang(2021)}]{Sun.J.R.2021_BLQR}
\bibinfo{author}{J.~Sun}, \bibinfo{author}{H.~Wang},
\newblock \bibinfo{title}{Linear-quadratic optimal control for backward
  stochastic differential equations with random coefficients},
\newblock \bibinfo{journal}{ESAIM: Control, Optimisation and Calculus of
  Variations} \bibinfo{volume}{27} (\bibinfo{year}{2021}) \bibinfo{pages}{27}.
\bibitem[{Huang et~al.(2016)Huang, Wang, and Wu}]{Huang.J.H.2016_BLQGMFP}
\bibinfo{author}{J.~Huang}, \bibinfo{author}{S.~Wang}, \bibinfo{author}{Z.~Wu},
\newblock \bibinfo{title}{Backward mean-field linear-quadratic-{G}aussian
  ({LQG}) games: full and partial information},
\newblock \bibinfo{journal}{IEEE Transactions on Automatic Control}
  \bibinfo{volume}{61} (\bibinfo{year}{2016}) \bibinfo{pages}{3784--3796}.
\bibitem[{Wang et~al.(2018)Wang, Xiao, and Xiong}]{Wang.G.C.2018_NZBLQ}
\bibinfo{author}{G.~Wang}, \bibinfo{author}{H.~Xiao},
  \bibinfo{author}{J.~Xiong},
\newblock \bibinfo{title}{A kind of {LQ} non-zero sum differential game of
  backward stochastic differential equation with asymmetric information},
\newblock \bibinfo{journal}{Automatica} \bibinfo{volume}{97}
  (\bibinfo{year}{2018}) \bibinfo{pages}{346--352}.
\bibitem[{Du et~al.(2018)Du, Huang, and Wu}]{Du.K.2018_BLQMF}
\bibinfo{author}{K.~Du}, \bibinfo{author}{J.~Huang}, \bibinfo{author}{Z.~Wu},
\newblock \bibinfo{title}{Linear quadratic mean-field-game of backward
  stochastic differential systems},
\newblock \bibinfo{journal}{Mathematical Control \& Related Fields}
  \bibinfo{volume}{8} (\bibinfo{year}{2018}) \bibinfo{pages}{653--678}.
\bibitem[{Chen et~al.(1998)Chen, Li, and Zhou}]{Chen.S.P.1998_ILQ}
\bibinfo{author}{S.~Chen}, \bibinfo{author}{X.~Li}, \bibinfo{author}{X.~Y.
  Zhou},
\newblock \bibinfo{title}{Stochastic linear quadratic regulators with
  indefinite control weight costs},
\newblock \bibinfo{journal}{SIAM Journal on Control and Optimization}
  \bibinfo{volume}{36} (\bibinfo{year}{1998}) \bibinfo{pages}{1685--1702}.
\bibitem[{Li et~al.(2018)Li, Wu, and Yu}]{Li.N.2018_ILQJ}
\bibinfo{author}{N.~Li}, \bibinfo{author}{Z.~Wu}, \bibinfo{author}{Z.~Yu},
\newblock \bibinfo{title}{Indefinite stochastic linear-quadratic optimal
  control problems with random jumps and related stochastic riccati equations},
\newblock \bibinfo{journal}{Science China Mathematics} \bibinfo{volume}{61}
  (\bibinfo{year}{2018}) \bibinfo{pages}{563--576}.
\bibitem[{Li et~al.(2001)Li, Zhou, and Rami}]{Li.X.2001_ILQJ}
\bibinfo{author}{X.~Li}, \bibinfo{author}{X.~Y. Zhou},
  \bibinfo{author}{M.~Rami},
\newblock \bibinfo{title}{Indefinite stochastic {LQ} control with jumps},
\newblock in: \bibinfo{booktitle}{Proceedings of the 40th IEEE Conference on
  Decision and Control (Cat. No.01CH37228)}, volume~\bibinfo{volume}{2},
  \bibinfo{publisher}{IEEE}, \bibinfo{year}{2001}, pp.
  \bibinfo{pages}{1693--1698}.
\bibitem[{Li et~al.(2003)Li, Zhou, and Ait~Rami}]{Li.X.2003_ILQMI}
\bibinfo{author}{X.~Li}, \bibinfo{author}{X.~Y. Zhou},
  \bibinfo{author}{M.~Ait~Rami},
\newblock \bibinfo{title}{Indefinite stochastic linear quadratic control with
  {M}arkovian jumps in infinite time horizon},
\newblock \bibinfo{journal}{Journal of Global Optimization}
  \bibinfo{volume}{27} (\bibinfo{year}{2003}) \bibinfo{pages}{149--175}.
\bibitem[{Sun and Yong(2014)}]{Sun.J.R.2014_NILQ}
\bibinfo{author}{J.~Sun}, \bibinfo{author}{J.~Yong},
\newblock \bibinfo{title}{Linear quadratic stochastic differential games:
  open-loop and closed-loop saddle points},
\newblock \bibinfo{journal}{SIAM Journal on Control and Optimization}
  \bibinfo{volume}{52} (\bibinfo{year}{2014}) \bibinfo{pages}{4082--4121}.
\bibitem[{Sun et~al.(2016)Sun, Li, and Yong}]{Sun.J.R.2016_open-closed}
\bibinfo{author}{J.~Sun}, \bibinfo{author}{X.~Li}, \bibinfo{author}{J.~Yong},
\newblock \bibinfo{title}{Open-loop and closed-loop solvabilities for
  stochastic linear quadratic optimal control problems},
\newblock \bibinfo{journal}{SIAM Journal on Control and Optimization}
  \bibinfo{volume}{54} (\bibinfo{year}{2016}) \bibinfo{pages}{2274--2308}.
\bibitem[{Sun et~al.(2023)Sun, Wu, and Xiong}]{Sun.J.R.2021_IBLQ}
\bibinfo{author}{J.~Sun}, \bibinfo{author}{Z.~Wu}, \bibinfo{author}{J.~Xiong},
\newblock \bibinfo{title}{Indefinite backward stochastic linear-quadratic
  optimal control problems},
\newblock \bibinfo{journal}{ESAIM: Control, Optimisation and Calculus of
  Variations} \bibinfo{volume}{29} (\bibinfo{year}{2023}) \bibinfo{pages}{30}.
\bibitem[{Sun et~al.(2022)Sun, Wen, and Xiong}]{Sun.J.R.2021_GIBLQ}
\bibinfo{author}{J.~Sun}, \bibinfo{author}{J.~Wen}, \bibinfo{author}{J.~Xiong},
\newblock \bibinfo{title}{General indefinite backward stochastic
  linear-quadratic optimal control problems},
\newblock \bibinfo{journal}{ESAIM: Control, Optimisation and Calculus of
  Variations} \bibinfo{volume}{28} (\bibinfo{year}{2022}) \bibinfo{pages}{17}.
\bibitem[{Von~Stackelberg(1934)}]{Stackelberg.H.V.1934}
\bibinfo{author}{H.~Von~Stackelberg}, \bibinfo{title}{Marktform und
  gleichgewicht}, \bibinfo{publisher}{Springer}, \bibinfo{year}{1934}.
\bibitem[{Bagchi and Ba{\c{s}}ar(1981)}]{Bagchi.A.1981_SLQ}
\bibinfo{author}{A.~Bagchi}, \bibinfo{author}{T.~Ba{\c{s}}ar},
\newblock \bibinfo{title}{Stackelberg strategies in linear-quadratic stochastic
  differential games},
\newblock \bibinfo{journal}{Journal of optimization theory and applications}
  \bibinfo{volume}{35} (\bibinfo{year}{1981}) \bibinfo{pages}{443--464}.
\bibitem[{Yong(2002)}]{Yong.J.M.2002_SLQ}
\bibinfo{author}{J.~Yong},
\newblock \bibinfo{title}{A leader-follower stochastic linear quadratic
  differential game},
\newblock \bibinfo{journal}{SIAM Journal on Control and Optimization}
  \bibinfo{volume}{41} (\bibinfo{year}{2002}) \bibinfo{pages}{1015--1041}.
\bibitem[{Shi et~al.(2016)Shi, Wang, and Xiong}]{Shi.J.T.2016_SLQAI}
\bibinfo{author}{J.~Shi}, \bibinfo{author}{G.~Wang},
  \bibinfo{author}{J.~Xiong},
\newblock \bibinfo{title}{Leader-follower stochastic differential game with
  asymmetric information and applications},
\newblock \bibinfo{journal}{Automatica} \bibinfo{volume}{63}
  (\bibinfo{year}{2016}) \bibinfo{pages}{60--73}.
\bibitem[{Moon and Yang(2020)}]{Moon.J.2020.SLQMFTC}
\bibinfo{author}{J.~Moon}, \bibinfo{author}{H.~J. Yang},
\newblock \bibinfo{title}{Linear-quadratic time-inconsistent mean-field type
  {S}tackelberg differential games: Time-consistent open-loop solutions},
\newblock \bibinfo{journal}{IEEE Transactions on Automatic Control}
  \bibinfo{volume}{66} (\bibinfo{year}{2020}) \bibinfo{pages}{375--382}.
\bibitem[{Moon(2021)}]{Moon.J.2021_SLQJ}
\bibinfo{author}{J.~Moon},
\newblock \bibinfo{title}{Linear-quadratic stochastic {S}tackelberg
  differential games for jump-diffusion systems},
\newblock \bibinfo{journal}{SIAM Journal on Control and Optimization}
  \bibinfo{volume}{59} (\bibinfo{year}{2021}) \bibinfo{pages}{954--976}.
\bibitem[{Du and Wu(2019)}]{Du.K.2019_SBLQMF}
\bibinfo{author}{K.~Du}, \bibinfo{author}{Z.~Wu},
\newblock \bibinfo{title}{Linear-quadratic {S}tackelberg game for mean-field
  backward stochastic differential system and application},
\newblock \bibinfo{journal}{Mathematical Problems in Engineering}
  \bibinfo{volume}{2019} (\bibinfo{year}{2019}) \bibinfo{pages}{1--17}.
\bibitem[{Lin et~al.(2012)Lin, Zhang, and Siu}]{lin_stochastic_2012}
\bibinfo{author}{X.~Lin}, \bibinfo{author}{C.~Zhang}, \bibinfo{author}{T.~K.
  Siu},
\newblock \bibinfo{title}{Stochastic differential portfolio games for an
  insurer in a jump-diffusion risk process},
\newblock \bibinfo{journal}{Mathematical Methods of Operations Research}
  \bibinfo{volume}{75} (\bibinfo{year}{2012}) \bibinfo{pages}{83--100}.
\bibitem[{Sun et~al.(2023)Sun, Wang, and Wen}]{Sun.J.R.2021_ZSILQ}
\bibinfo{author}{J.~Sun}, \bibinfo{author}{H.~Wang}, \bibinfo{author}{J.~Wen},
\newblock \bibinfo{title}{Zero-sum {S}tackelberg stochastic linear-quadratic
  differential games},
\newblock \bibinfo{journal}{SIAM Journal on Control and Optimization}
  \bibinfo{volume}{61} (\bibinfo{year}{2023}) \bibinfo{pages}{252--284}.
\bibitem[{Wu et~al.(2024)Wu, Xiong, and Zhang}]{wu_zero-sum_2024}
\bibinfo{author}{F.~Wu}, \bibinfo{author}{J.~Xiong},
  \bibinfo{author}{X.~Zhang},
\newblock \bibinfo{title}{Zero-sum stochastic linear-quadratic {S}tackelberg
  differential games with jumps},
\newblock \bibinfo{journal}{Applied Mathematics \& Optimization}
  \bibinfo{volume}{89} (\bibinfo{year}{2024}).

\end{thebibliography}
\end{document}